\documentclass[aos,preprint]{imsart}

\RequirePackage[numbers]{natbib}
\RequirePackage[colorlinks,citecolor=blue,urlcolor=blue]{hyperref}

\usepackage{amsmath}
\usepackage{amssymb}
\usepackage{amsthm}
\usepackage{amsbsy}
\usepackage{amsfonts}
\usepackage{graphicx}
\usepackage{psfrag}
\usepackage{epstopdf}
\usepackage{enumerate}
\usepackage{latexsym}
\usepackage{xr}

\ifdefined\separatesupplement
\fi

\usepackage{xifthen}
\usepackage{fdivs}

\begin{document}

\begin{frontmatter}
  \title{Multiclass Classification, Information, Divergence,
    and Surrogate Risk}
  \runtitle{Multiclass Classification, Information, and Surrogate Risk}

  \begin{aug}
    \author{\fnms{John} \snm{Duchi}\thanksref{t1}\ead[label=e1]{jduchi@stanford.edu}},
    \author{\fnms{Khashayar} \snm{Khosravi}\ead[label=e2]{khosravi@stanford.edu}},
    \and
    \author{\fnms{Feng} \snm{Ruan}\thanksref{t1}\ead[label=e3]{fengruan@stanford.edu}}

    \runauthor{Duchi, Khosravi, Ruan}

    \affiliation{Stanford University}

    \address{Stanford University \\
      Stanford, California 94305 \\
      \printead{e1} \\
      \phantom{E-mail:\ }\printead*{e2} \\
      \phantom{E-mail:\ }\printead*{e3}}

    \thankstext{t1}{Partially supported by
      NSF-CAREER Award 1553086 and the SAIL-Toyota
      Center for AI Research. FR additionally
      supported by the Stanford Graduate Fellowship.}


  \end{aug}

  \begin{abstract}
    We provide a unifying view of statistical information measures,
    multi-way Bayesian hypothesis
    testing, loss functions for multi-class classification problems,
    and multi-distribution $f$-divergences,
    elaborating equivalence results between all
    of these objects, and extending existing results for binary outcome
    spaces to more general ones. 
    We consider a generalization of $f$-divergences to multiple
    distributions, and we provide a constructive equivalence
    between divergences, statistical information (in the sense
    of DeGroot), and losses for multiclass classification.
    A major application of our results is in
    multi-class classification problems in which we must both infer a
    discriminant function $\discfunc$---for making predictions
    on a label $Y$ from datum $X$---and a data representation (or, in the
    setting of a hypothesis testing problem, an experimental design),
    represented as a quantizer $\quant$ from a family of possible quantizers
    $\quantfam$. In this setting, we characterize the
    equivalence between loss functions, meaning that optimizing either of
    two losses yields an optimal discriminant and quantizer $\quant$,
    complementing and extending earlier results
    of \citet{NguyenWaJo09} to the multiclass case. Our results
    provide a more substantial basis than standard classification
    calibration results for comparing different losses:
    we describe the convex losses
    that are consistent for jointly choosing a data
    representation and minimizing the (weighted) probability of error in
    multiclass classification problems.
  \end{abstract}

\end{frontmatter}


\section{Introduction}
\label{sec:intro}

Consider the multiclass classification problem: a
decision maker receives a pair of random variables $(X, Y) \in \mc{X} \times
\{1, \ldots, k\}$, where $Y$ is unobserved, and wishes to assign the
variable $X$ to one of the $k$ classes $\{1, 2, \ldots, k\}$ to
minimize the probability of a misclassification.
We represent the decision maker 
via a discriminant function $\discfunc : \mc{X} \to \R^k$, where
each component $\discfunc_y(x)$, $y = 1, \ldots, k$, represents the
\emph{margin} (or a score or perceived likelihood) the decision
maker assigns to class $y$ for datum $x$. The goal is then to minimize the
expected loss, or \emph{$\loss$-risk},
\begin{equation}
  \label{eqn:surrogate-risk}
  \surrrisk(\discfunc) \defeq \E[\loss(\discfunc(X), Y)],
\end{equation}
where $\loss(\discfunc(x), y)$ measures the
loss of margins $\discfunc(x)\in \R^k$ when the true label of $x$ is $y$
and the expectation~\eqref{eqn:surrogate-risk} is taken jointly over $(X,
Y)$. When $\loss$ is the 0-1 loss, $\loss(\discfunc(x), y)
= \indic{\discfunc_y(x) \le \discfunc_i(x) ~ \mbox{for~some~} i \neq y}$,
the formulation~\eqref{eqn:surrogate-risk} is the misclassification
probability $\P(\argmax_y \discfunc_y(X) \neq Y)$. We may also
consider the classical $k$-category Bayesian experiment: given a random
variable $X \in \mc{X}$ drawn according to one of the $k$ hypotheses
\begin{equation*}
  H_1 : X \sim P_1,
  ~~~ H_2 : X \sim P_2,
  ~~~ \ldots,
  ~~~ H_k : X \sim P_k
\end{equation*}
with prior probabilities $\prior_1, \ldots, \prior_k$, we wish to
choose $\discfunc$ minimizing the expected $\E[\loss(\discfunc(X), Y)]$ or
posterior $\sum_y \P(Y = y \mid X = x)
\loss(\discfunc(x), y)$ loss.

In many applications, making decisions based on the raw $X$ is
undesirable---the vector $X$ may be high-dimensional, carry
useless information impinging on statistical efficiency, or we may need to
store or communicate the sample using limited memory or bandwidth.  If all
we wish to do is to classify a person as being taller or shorter than 160
centimeters, it makes little sense to track his or her blood type and eye
color. With the increase in the number and variety of measurements we
collect, such careful design choices are
important for maintaining statistical power, interpretability,
efficient downstream use, and mitigating false
discovery~\cite{BenjaminiHo95}.  This desire to give ``better''
representations of data $X$ has led to a rich body of work in statistics,
machine learning, and engineering, highlighting the importance of careful
measurement, experimental design, and data representation
strategies~\cite{Robbins52,DeGroot62,Pukelsheim93,TishbyPeBi99}.

As \citet{NguyenWaJo09} note in the binary case, one
thus frequently extends the classical formulation~\eqref{eqn:surrogate-risk}
to include a stage in which
a (data-dependent) $\quant : \mc{X} \to \mc{Z}$ maps the vector $X$
into a vector $Z$. A number of situations suggest such an approach.  In
most practical classification scenarios~\cite{SchapireFr12}, an
equivalent feature selection reduces the dimension of $X$ or
increases its interpretability. As a second motivation, consider the
decentralized detection problem~\cite{Tsitsiklis93,NguyenWaJo09} in
communication applications in engineering, where
remote sensors communicate data $X \sim P_i$
through limited bandwidth or memory.
In this case, the central processor can infer the
distribution $P_i$ only after communication of the transformed
vector $Z = \quant(X)$, and one chooses a \emph{quantizer} $\quant$
from a family $\quantfam$ of (low complexity) quantizers.  In
fuller abstraction, we may treat the problem as a Bayesian experimental
design problem, where the mapping $\quant : \mc{X} \to \mc{Z}$ may be
stochastic and is chosen from a family $\quantfam$ of possible experiments
(observation channels).  In each of the preceding examples, the
incorporation of a quantizer $\quant$ into the classification procedure poses a
more complex problem, as one must simultaneously find a data representation
$\quant$ and discriminant $\discfunc$. The goal, paralleling that for the
risk~\eqref{eqn:surrogate-risk}, thus becomes joint minimization of the
\emph{quantized $\loss$-risk}
\begin{equation}
  \label{eqn:quantized-risk}
  \risk_\loss(\discfunc \mid \quant)
  \defeq \E\left[\loss(\discfunc(\quant(X)), Y)\right]
\end{equation}
over a prespecified family $\quantfam$ of quantizers
$\quant : \mc{X} \to \mc{Z}$, where $\discfunc : \mc{Z} \to \R^k$.

Often---for example, in the zero-one error case---the loss functions
$\loss(\cdot, y)$ are non-convex (even discontinuous), so
population or empirical minimization is intractable. It is thus common to
replace the loss with a convex \emph{surrogate} and minimize this surrogate
instead.  A surrogate is \emph{Fisher consistent} if its minimization
yields a Bayes optimal discriminant $\discfunc$ for the original loss
$\loss$ (for any distribution $\P$ on $(X, Y)$); researchers have
characterized the Fisher consistency of convex surrogates for binary
and multiclass classification~\cite{Zhang04a,
  LugosiVa04, BartlettJoMc06, TewariBa07}.  A weakness of such results is
that they rely strongly on using the unrestricted class of all measurable
discriminants $\discfunc : \mc{X} \to \R^k$, and thus most
``natural'' convex losses are consistent~\cite{Zhang04a, BartlettJoMc06}. In
this context, a major difficulty is to understand the consequences of using
various surrogate losses, and requiring a restricted quantizer class
$\quantfam$ is one approach to discovering nuanced properties of the
relationship between surrogate and Bayes risk.  \citet{NguyenWaJo09} tackle
this in the binary case, considering the problem of joint selection of the
discriminant function $\discfunc : \mc{Z} \to \R$ and quantizer $\quant :
\mc{X} \to \mc{Z}$.  They exhibit a precise correspondence between binary
margin-based loss functions and $f$-divergences---measures of the similarity
between two probability distributions developed in information theory and
statistics~\cite{AliSi66,Csiszar67,Vajda72}---to give a general
characterization of loss equivalence through classes of divergences. An
interesting consequence of their results is that, in spite of positive
results for Fisher consistency in
binary classification problems~\cite{Zhang04a,LugosiVa04,BartlettJoMc06},
essentially only hinge-like losses are
consistent for the 0-1 loss. We provide the extension of these results to
the multiclass case.

\subsubsection*{Outline and discussion of our contributions}

We build on this prior work
to provide a unifying framework that relates statistical information
measures, loss functions, and generalized notions of entropy in the context
of multi-class classification.  To begin, we recall a generalization
of $f$-divergences that applies to multiple
distributions~\cite{GyorfiNe78,GarciaWi12}, enumerating analogues of the
positivity properties, data-processing inequalities, and discrete
approximation available in the binary case, as multi-way $f$-divergences may
be unfamiliar and they motivate our approach
(Section~\ref{sec:f-divergence}).  We begin our main contributions in
Section~\ref{sec:relation}, where we establish a correspondence between loss
functions $\loss$, generalized entropy on discrete
distributions~\cite{GrunwaldDa04}, and multi-way $f$-divergences. To make this
precise, define the probability simplex
$\simplex_k \defeq \{p \in \R^k_+ \mid \ones^T p = 1\}$.
Let $\prior \in \simplex_k$ be a
prior distribution on the class label $Y$ and $\wt{\prior}(x) \in
\simplex_k$ be the posterior probabilities for $Y$ conditional on
the observation $X = x$.
For concave $\entropy : \simplex_k \to \R$,
\citet{DeGroot62} defines the \emph{information}
associated with the experiment $(X, Y)$ as
\begin{equation}
  \label{eqn:information}
  \infoentropy
  \defeq \entropy(\prior) - \E[\entropy(\wt{\prior}(X))].
\end{equation}
The notion of $\entropy$ as a generalized entropy is clear here, as $I$
is the gap between prior and posterior entropy and is always non-negative.
In this context, the value $\entropy(\prior)$ measures the
\emph{uncertainty} of the experimenter (in some appropriate units) about the
unknown class label $y$ when his or her prior belief over $y$ is $\prior$,
so $I$ is the gap between prior and posterior uncertainty~\cite{DeGroot62}.

To relate this type of entropy to
loss functions, recall the well-known
result~\cite{DeGroot62,GrunwaldDa04} that any loss $\loss : \R^k \times [k]
\to \R \cup \{+\infty\}$
induces an entropy
$\entropy_L : \simplex_k \to \R$, also called the \emph{pointwise Bayes
risk}~\cite{GrunwaldDa04,ReidWi11,GarciaWi12,WilliamsonVeRe16}, via
\begin{equation}
  \label{eqn:inf-representation-full}
  \entropy_\loss(\prior) = \inf_{\alpha \in \R^k} \bigg\{
  \sum_{i=1}^k \prior_i \loss(\alpha, i)
  \bigg\}
  - \bindic{\prior \in \simplex_k}
\end{equation}
where $\bindic{\cdot}$ is $+\infty$ if its argument is false (we
drop this indicator in the future, defining $\entropy_\loss$
implicitly on $\simplex_k$).  We show
an inverse construction, providing an explicit and constructable
mapping from any concave function $\entropy$ to a loss $\loss$
inducing $\entropy$ as the pointwise Bayes risk~\eqref{eqn:inf-representation-full},
where for each $y$ the
loss $\alpha \mapsto \loss(\alpha, y)$ is convex.  In
Section~\ref{sec:uncertainty-consistency}, we
also develop the natural connections between
these generalized entropies $\entropy$ and classification
calibration~\cite{Zhang04a,LugosiVa04,BartlettJoMc06,TewariBa07}, in that
our explicit $\loss$ is generally calibrated.

In Section~\ref{sec:equiv}, we address the comparison of loss
functions---building off of Nguyen et al.'s approach in the
binary case~\cite{NguyenWaJo09}---and present our main results on
consistency of joint selection of quantizer (data representation)
$\quant$ and discriminant $\discfunc$. Using our correspondence
between concave $\entropy$, losses $\loss$, and $f$-divergences,
we characterize the pairs of losses $\loss^{(1)}$ and $\loss^{(2)}$
for which equivalent quantizers and discriminants minimize the quantized
risk~\eqref{eqn:quantized-risk} in the sense that there is a continuous
concave $h$ with $h(0) = 0$ such that
\begin{equation*}
  \risk_{\loss^{(2)}}(\discfunc \mid \quant)
  - \inf_{\discfunc, \quant \in \quantfam}
  \risk_{\loss^{(2)}}(\discfunc \mid \quant)
  \le h\left(
  \risk_{\loss^{(1)}}(\discfunc \mid \quant)
  - \inf_{\discfunc, \quant \in \quantfam}
  \risk_{\loss^{(1)}}(\discfunc \mid \quant)\right)
\end{equation*}
for any $\discfunc$ and $\quant \in \quantfam$.
Another way to understand our results is as providing insight into
classification calibration when the Bayes' act (i.e.\ optimal discriminant
$\discfunc$) does not belong to the class of functions the statistician may
choose in a classification problem. A substantial challenge for and
criticism of the line of work on classification calibration and surrogate
risk consistency~\cite{Zhang04a,LugosiVa04,BartlettJoMc06,TewariBa07} is
that the results say little for restricted families of classifiers. In this
context, a corollary of our main contribution is as follows.
The loss $\loss^{(1)}$ is
\emph{calibrated}~\cite{Zhang04a,BartlettJoMc06,TewariBa07}
for $\loss^{(2)}$ if for any distribution $P$ on $X \times
Y$ and sequence $\discfunc_n : \mc{X} \to \R^k$,
$\risk_{\loss^{(1)}}(\discfunc_n) \to \inf_\discfunc
\risk_{\loss^{(2)}}(\discfunc)$ implies $\risk_{\loss^{(1)}}(\discfunc_n)
\to \inf_\discfunc \risk_{\loss^{(2)}}(\discfunc)$.  Now, consider a
collection $\quantfam$ of functions $\quant : \mc{X} \to \mc{Z}$ for some set
$\mc{Z}$, and then define the class of functions
\begin{equation*}
  \mc{G}(\quantfam)
  \defeq \left\{\discfunc \circ \quant
  \mid \quant \in \quantfam
  ~ \mbox{and} ~\discfunc : \mc{Z} \to \R^k ~ \mbox{is~measurable}
  \right\}.
\end{equation*}
Translated to this scenario, our main
results---Theorems~\ref{theorem:loss-equivalent-entropy}
and~\ref{theorem:loss-equivalent}---imply
\begin{corollary}
\label{corollary:classification-calibration-restricted-families}
  Assume that the loss $\loss^{(1)}$ is 
  calibrated for $\loss^{(2)}$ and let $\entropy_i = \entropy_{\loss^{(i)}}$ denote
  the associated pointwise Bayes risk~\eqref{eqn:inf-representation-full}.
  Then
  \begin{equation}
    \label{eqn:restricted-risk-convergence}
    \risk_{\loss^{(1)}}(g_n)
    \to \inf_{g \in \mc{G}(\quantfam)}
    \risk_{\loss^{(1)}}(g)
    ~~ \mbox{implies} ~~
    \risk_{\loss^{(2)}}(g_n)
    \to \inf_{g \in \mc{G}(\quantfam)}
    \risk_{\loss^{(2)}}(g)
  \end{equation}
  for any collection $\quantfam$ of mappings $\mc{X} \to \mc{Z}$, any set
  $\mc{Z}$, any distribution $P$ on $\mc{X} \times \{1, \ldots, k\}$,
  and sequence $g_n \in \mc{G}(\quantfam)$
  if and only if there exist $a > 0$, $b \in \R^k$, and $c \in \R$ such
  that
  $\entropy_1(\prior) = a \entropy_2(\prior) + b^T \prior + c$
  for all $\prior \in \simplex_k$.
\end{corollary}
\noindent


This corollary reposes on the connections we develop between losses, 
uncertainty measures and generalized $f$-divergences. 
Such measures of statistical information and
divergence have been central to the design of communication and quantization
schemes in signal processing~\cite{Tsitsiklis93, PoorTh77,
  LongoLoGr90, Kailath67}; we
characterize the divergence measures that, when
optimized, yield optimal quantizers and detectors.
We also provide a result showing when empirical minimization of a surrogate
risk is consistent for the desired (original) risk.

A number of researchers have studied the connections between divergence
measures and risk for binary and multi-category experiments; these point to
the results we present. Indeed, \citet{Blackwell51} shows that if a
quantizer $\quant_1$ induces class-conditional distributions with larger
divergence than those induced by $\quant_2$, then there are prior
probabilities such that $\quant_1$ allows tests with lower probability of
error than $\quant_2$. \citet{LieseVa06} give a broad treatment of
$f$-divergences, using their representation as the difference between prior
and posterior risk in a binary experiment~\cite{OsterreicherVa93} to derive
a number of their properties; see also the
paper~\cite{ReidWi11}. \citet{GarciaWi12} show how multi-distribution
$f$-divergences~\cite{GyorfiNe78} arise naturally in the context of
multi-class classification problems as the gap between prior and
posterior risk in classification, as in the work~\cite{LieseVa06}. In the
binary case, these results elucidate the characterization of
Fisher consistency for quantization and binary classification
\citet{NguyenWaJo09} realize. We pursue this line of research to draw the
connections between Fisher consistency, information measures, multi-class
classification, surrogate losses, and divergences.

\paragraph{Notation}

We let $\zeros$ and $\ones$ denote the all-zeros and all-ones vectors,
respectively. For a vector or collection of objects
we define $t_{1:m} = \{t_1, \ldots, t_m\}$. The indicator function
$\bindic{\cdot}$ is $+\infty$ if its argument is false, $0$ otherwise, while
$\indic{\cdot}$ is $1$ if its argument is true, $0$ otherwise.
We let $\simplex_k = \{v \in \R^k_+ : \ones^T v = 1\}$ denote the
probability simplex in $\R^k$. For $m \in \N$, we set $[m] = \{1,
\ldots, m\}$. We let $\affine{A} = \{\sum_{i = 1}^m \lambda_i x_i \mid
\lambda^T \ones = 1, x_i \in A, m \in \N\}$ denote the affine hull of a set
$A$, and $\relint A$ denotes the interior of $A$ relative to $\affine{A}$.
We let $\extendedR = \R \cup \{+\infty\}$ and $\extendedRlow = \R \cup
\{-\infty\}$. 
For $f : \R^k
\to \extendedR$, we let $\epi f = \{(x, t) : f(x) \le t\}$ denote the
epigraph of $f$. We say a convex function $f$ is closed if $\epi f$ is a
closed set, though we abuse notation and say that a concave $f$ is
closed if $\epi(-f)$ is closed.  For a convex function $f : \R^k \to
\extendedR$, we say that $f$ is \emph{strictly convex at a point $t \in
  \R^k$} if for all $\lambda \in (0, 1)$ and $t_1, t_2 \neq t$ such that $t
= \lambda t_1 + (1 - \lambda) t_2$ we have $f(t) < \lambda f(t_1) + (1 -
\lambda) f(t_2)$. The (Fenchel) conjugate of a function $f : \R^k
\to \extendedR$ is
\begin{equation}
  \label{eqn:fenchel-conjugate}
  f^*(s) = \sup_{t \in \R^k} \left\{s^T t - f(t) \right\}.
\end{equation}
For any $f$, the conjugate $f^*$ is closed convex~\cite[Chapter
  X]{HiriartUrrutyLe93ab}. For measures $\nu$ and $\mu$, we let $d\nu / d\mu$
denote the Radon-Nikodym derivative of $\nu$ with respect to $\mu$. For
random variables $X_n$, we say $X_n \clp{p} X_\infty$ if $\E[|X_n -
  X_\infty|^p] \to 0$.

\section{Multi-distribution $f$-divergences}
\label{sec:f-divergence}

Divergence measures for probability distributions have significant
statistical, decision-, and information-theoretic applications, including in
optimal testing, minimax rates of convergence, and the design of
communication schemes~\cite{AliSi66,Csiszar67,PoorTh77,Kailath67}.
Central to this work is the \emph{$f$-divergence}, introduced by
\citet{AliSi66} and \citet{Csiszar67} (see~\cite{LieseVa06} for
an overview). Given distributions $P, Q$ defined on a common set $\mc{X}$, a
closed convex function $f : \openright{0}{\infty} \to \extendedR$ satisfying
$f(1) = 0$, and any measure $\mu$ dominating $P$ and $Q$, the $f$-divergence
between $P$ and $Q$ is
\begin{equation}
  \fdiv{P}{Q} \defeq \int_{\mc{X}} f\left(\frac{p(x)}{q(x)}\right) q(x)
  d\mu(x)
  = \int f\left(\frac{dP}{dQ}\right) dQ.
  \label{eqn:f-divergence}
\end{equation}
Here $p = \frac{dP}{d\mu}$ and $q = \frac{dQ}{d\mu}$ denote the densities of
$P$ and $Q$, respectively, and the value $u f(t / u)$ is defined
appropriately for $t = 0$ and $u = 0$ (e.g.~\cite{LieseVa06}). A number of
classical divergence measures arise out of the $f$-divergence; taking
$f(t) = t \log t$, $f(t) = \half (\sqrt{t} - 1)^2$, or $f(t) =
|t - 1|$ yields (respectively) the KL-divergence, squared Hellinger
distance, or total variation distance.

Central to our study of multi-way hypothesis testing and classification is
an understanding of relationships between multiple distributions.
We use the following generalization~\cite{GyorfiNe78,
  GarciaWi12} of the $f$-divergence to multiple distributions.
\begin{definition}
  \label{def:general-fdiv}
  Let $P_1, \ldots, P_k$ be probability distributions on a common
  $\sigma$-algebra $\mc{F}$ over a set $\mc{X}$.  Let $f : \R^{k-1}_+ \to
  \extendedR$ be a closed convex function satisfying $f(\ones) = 0$. Let $\mu$ be
  any $\sigma$-finite measure such that $P_i \ll \mu$ for all $i$, and let
  $p_i = dP_i / d\mu$.
  The \emph{$f$-divergence} between $P_1, \ldots, P_{k-1}$ and
  $P_k$ is
  \begin{equation}
    \label{eqn:general-fdiv}
    \fdiv{P_1, \ldots, P_{k-1}}{P_k}
    \defeq \int f\left(\frac{p_1(x)}{p_k(x)},
    \ldots, \frac{p_{k-1}(x)}{p_k(x)}\right) p_k(x) d\mu(x).
  \end{equation}
\end{definition}
\noindent
We must specify the value of the integrand~\eqref{eqn:general-fdiv} when
$p_k(x) = 0$. In this case, the function $\wt{f} : \R^k_+ \to \extendedR$
defined, for an arbitrary $t' \in \relint \dom f$, by
\begin{equation}
  \label{eqn:closed-perspective}
  \wt{f}(t, u) = \begin{cases}
    u f(t / u) & \mbox{if~} u > 0 \\
    \displaystyle{\lim_{s \to 0}}\, s f(t' - t + t / s) & \mbox{if~} u = 0 \\
    + \infty & \mbox{otherwise}
  \end{cases}
\end{equation}
is a closed convex function with value independent of $t'$; $\wt{f}$ is
the closure of the \emph{perspective} function
$\R_+ \times \R^k \ni (u, t) \mapsto u f(t/u)$
of $f$~\cite[Prop.~IV.2.2.2]{HiriartUrrutyLe93ab}.  Any time we consider the
perspective we implicitly treat it as
its closure~\eqref{eqn:closed-perspective}.

We now enumerate a few properties of multi-way $f$-divergences, showing how
they naturally generalize classical binary $f$-divergences. We
focus on basic properties that are useful for our further results on Bayes
risk, classification, and hypothesis testing
and that parallel results in the binary case~\eqref{eqn:f-divergence}:
they are well-defined,
have continuity properties with respect to discrete approximations, and
satisfy data-processing inequalities. While Gy\"orfi and Nemetz's
original work~\cite{GyorfiNe78} essentially contains the
results,
we carefully address infinite
values (the closure~\eqref{eqn:closed-perspective})
and strict convexity,
and we use them as definitional building blocks; we defer all proofs
to Supplement~\ref{sec:fdiv-proofs}.

As our first step, we note that Definition~\ref{def:general-fdiv} is
independent of the base measure $\mu$.
(See Supp.~\ref{sec:proof-fdiv-mu-fine} for a proof
generalizing~\cite[Cor.~1]{GyorfiNe78}.)
\begin{lemma}
  \label{lemma:fdiv-mu-fine}
  In expression~\eqref{eqn:general-fdiv}, the value of the divergence
  does not depend on the choice of the dominating measure $\mu$.
  Moreover,
  \begin{equation*}
    \fdiv{P_1, \ldots, P_{k-1}}{P_k} \ge 0.
  \end{equation*}
  The inequality is strict if $f$ is strictly convex at $\ones$ and the
  $P_i$ are not identical.
\end{lemma}

Given the importance of quantization to
come,
we now consider discrete approximations to the divergence.  For an at
most countable
partition $\mc{P}$ of $\mc{X}$, we define the partitioned $f$-divergence
\begin{equation*}
  \fdiv{P_1, \ldots, P_{k-1}}{P_k \mid \mc{P}}
  = \sum_{A \in \mc{P}}
  f\left(\frac{P_1(A)}{P_k(A)}, \ldots, \frac{P_{k-1}(A)}{P_k(A)}\right)
  P_k(A).
\end{equation*}
As in the binary case~\cite{Vajda72,LieseVa06}, we have the following
approximability result generalizing~\cite[Thm.~6]{GyorfiNe78}
to possibly infinite integrands:
quantizers give arbitrarily good approximations to
$f$-divergences (see \S~\ref{sec:proof-fdiv-sup} for a proof).
\begin{proposition}
  \label{proposition:fdiv-sup}
  If $f$ is a closed convex function with $f(\ones) = 0$, then
  \begin{equation*}
    \fdiv{P_1, \ldots, P_{k-1}}{P_k}
    = \sup_{\mc{P}} \fdiv{P_1, \ldots, P_{k-1}}{P_k \mid \mc{P}}
  \end{equation*}
  where the supremum is over finite partitions of $\mc{X}$.
\end{proposition}

In the binary case, $f$-divergences satisfy \emph{data
  processing inequalities}~\cite{Csiszar67, CoverTh06,
  LieseVa06} which state that processing or transforming an observation $X$
drawn from the distributions $P_1, P_2$, decreases the
divergence between them. 
To formalize this, recall that $Q$ is a \emph{Markov kernel} from a set
$\mc{X}$ to $\mc{Z}$ if $Q(\cdot \mid x)$ is a probability distribution on
$\mc{Z}$ for each $x \in \mc{X}$, and for each measurable $A \subset
\mc{Z}$, the mapping $x \mapsto Q(A \mid x)$ is measurable. The
following general data processing inequality shows that this holds
in the multi-distribution case as well,
generalizing~\cite[Thm.~4]{GyorfiNe78}
to possibly infinite $f$ and the closure~\eqref{eqn:closed-perspective};
we include a proof
in Appendix~\ref{sec:proof-data-processing}.
\begin{proposition}
  \label{proposition:data-processing}
  Let $f$ be closed convex with $f(\ones) = 0$, $Q$ be a
  Markov kernel from $\mc{X}$ to $\mc{Z}$, and define the marginals $Q_P(A)
  =\int_{\mc{X}} Q(A \mid x) dP(x)$. Then
  \begin{equation*}
    \fdiv{Q_{P_{1}}, \ldots, Q_{P_{k-1}}}{Q_{P_k}}
    \leq \fdiv{P_1, \ldots, P_{k-1}}{P_k}.
  \end{equation*}
\end{proposition}

This proposition is related to the relationships between
risk, information, and quantization we develop in
Sections~\ref{sec:relation} and \ref{sec:equiv}.
Defining a \emph{quantizer} $\quant$ to be a measurable mapping $\quant :
\mc{X} \to \mc{Z}$ between measurable spaces $\mc{X}$ and $\mc{Z}$,
the \emph{quantized $f$ divergence} is
\begin{equation*}
  \fdiv{P_1, \ldots, P_{k-1}}{P_k \mid \quant}
  \defeq
  \sup_\mc{P} \sum_{A \in \quant^{-1}(\mc{P})}
  f\left(\frac{P_1(A)}{P_k(A)}, \ldots, \frac{P_{k-1}(A)}{P_k(A)}\right)
  P_k(A),
\end{equation*}
where $\mc{P}$ ranges over finite partitions of $\mc{Z}$ and
$\quant^{-1}(\mc{P}) = \{\quant^{-1}(B) \mid B \in \mc{P}\}$.
Proposition~\ref{proposition:data-processing}
immediately yields that quantization
reduces information: the indicator $Q(A \mid x) =
\indic{\quant(x) \in A}$ defines a Markov kernel, yielding
\begin{corollary}
  \label{corollary:data-processing}
  Let $f$ be closed convex, satisfy $f(\ones) = 0$, and $\quant$ be
  a quantizer of $\mc{X}$. Then
  \begin{equation*}
    \fdiv{P_1, \ldots, P_{k-1}}{P_k \mid \quant}
    \leq \fdiv{P_1, \ldots, P_{k-1}}{P_k}.
  \end{equation*}
\end{corollary}
\noindent
We also see that if $\quant_1$ and $\quant_2$ are quantizers of $\mc{X}$,
and $\quant_1$ induces a finer partition of $\mc{X}$ than $\quant_2$,
meaning that for $x, x' \in \mc{X}$ the equality $\quant_1(x) =
\quant_1(x')$ implies $\quant_2(x) = \quant_2(x')$, we have
\begin{equation*}
  \fdiv{P_1, \ldots, P_{k-1}}{P_k \mid \quant_2}
  \le \fdiv{P_1, \ldots, P_{k-1}}{P_k \mid \quant_1}.
\end{equation*}
This type of ordering is central to this work: any
multiclass loss $\loss$ induces a unique $f$-divergence, and
consistency of discriminants $\discfunc : \mc{X} \to \R^k$ for a loss $\loss$
after quantization is intimately tied to the preservation (and relative
ordering) of information as related to the quantized
risk~\eqref{eqn:quantized-risk}.



\section{Risks, information measures, and $f$-divergences}
\label{sec:relation}

Having reviewed the basic properties of $f$-divergences, we turn
to a more detailed look at their relationships with multi-way hypothesis
tests, multi-class classification, generalized entropies, and statistical
informations relating multiple distributions.  We build
a correspondence between these that parallels that for binary
experiments and classification
problems~\cite{LieseVa06,NguyenWaJo09,ReidWi11}.

We first recapitulate the probabilistic model for classification and
Bayesian hypothesis testing problems from the introduction. We have
a prior $\prior \in \simplex_k$ 
and probability distributions $P_1, \ldots, P_k$ defined on a set
$\mc{X}$. The coordinate $Y \in [k]$ is drawn according
to a multinomial with probabilities $\prior$, and conditional on $Y = y$, we
draw $X \sim P_y$. Following DeGroot~\cite{DeGroot62}, we refer to this as
an \emph{experiment}.
Associated with any experiment is a
family of informations as follows. Let $\wt{\prior}$ be the
posterior distribution on $Y$ given observation $X = x$,
$\wt{\prior}_i(x) = \prior_i dP_i(x) / (\sum_{j=1}^k \prior_j dP_j(x))$.
Given any closed concave $\entropy : \R^k_+ \to \R$, 
which we refer to as \emph{generalized entropy}
(see~\cite[\S~3.3]{GrunwaldDa04};
\citet{DeGroot62} calls $\entropy$ an uncertainty function), 
the \emph{information} associated with the experiment is the reduction 
of entropy (uncertainty) from prior to posterior~\eqref{eqn:information},
\begin{equation*}
  \infoentropy = \entropy(\prior) - \E[\entropy(\wt{\prior}(X))].
\end{equation*}
The expectation is taken over $X \sim
\sum_{i=1}^k \prior_i P_i$. That $\infoentropy \ge 0$ is immediate by
concavity; \citet[Thm.~2.1]{DeGroot62} shows that $\infoentropy \ge 0$
for all distributions $P_1, \ldots, P_k$ and priors $\prior$ if and only if
$\entropy$ is concave on $\simplex_k$.

In this section we develop equivalence results between multiclass
classification losses and risk, multi-way $f$-divergences, and
entropy measures.  Concretely, consider $\loss : \R^k \times [k] \to \R$, and
recall the risk~\eqref{eqn:surrogate-risk}, defined as $\surrrisk(\discfunc)
= \E[\loss(\discfunc(X), Y)]$, where $\discfunc \in \Discfuncs$, the set of
measurable functions $\discfunc : \mc{X} \to \R^k$.  As in
equation~\eqref{eqn:inf-representation-full} in the introduction, each loss
$\loss$ induces the entropy $\entropy_\loss$ on $\simplex_k$ via
$\entropy_\loss(\prior) = \inf_{\alpha \in \R^k} \sum_{i = 1}^k \prior_i
\loss(\alpha, i)$, also called the pointwise Bayes
risk~\cite{GrunwaldDa04,GarciaWi12,ReidWi11,WilliamsonVeRe16}.  In
Section~\ref{sec:infimal-representations}, we give an explicit inverse
mapping showing how each generalized entropy $\entropy$ is induced by (at least
one) \emph{convex} loss function $\loss$, i.e., $\loss(\cdot, i)$ is convex
for each $i$. In Section~\ref{sec:uncertainty-consistency}, we illustrate
consistency properties the entropy $\entropy$ implies about the
convex loss $\loss$ inducing it. We connect these results in
Section~\ref{sec:correspondence-u-f} with multi-way $f$-divergences. For any
loss $\loss$ and associated entropy/Bayes risk $\entropy_\loss$, for all
$\prior \in \simplex_k$ there exists a convex function $f_{\loss,\prior} :
\R^{k-1}_+ \to \R$ with $f_{\loss,\prior}(\ones) = 0$ such that the gap
between the prior Bayes $\loss$-risk---the best expected loss attainable
without observing $X$---and the posterior Bayes risk 
$\inf_\discfunc \risk_\loss(\discfunc)$ is
\begin{equation*}
  \entropy_\loss(\prior)
  - \inf_{\discfunc \in \Discfuncs} \surrrisk(\discfunc)
  =
  \entropy_\loss(\prior) - \E[\entropy_\loss(\wt{\prior}(X))]
  = \genfdiv{f_{\loss,\prior}}{P_1, \ldots, P_{k-1}}{P_k}
\end{equation*}
(see~\cite{GrunwaldDa04,GarciaWi12}).
The inverse direction is new, and given any closed convex function
$f : \R^{k-1}_+ \to
\R$ with $f(\ones) = 0$, we construct convex losses $\loss(\cdot, i)$,
an associated generalized entropy $\entropy_\loss$, and
prior $\prior = \ones / k\in \simplex_k$ satisfying
\begin{equation*}
  \fdiv{P_1, \ldots, P_{k-1}}{P_k}
  = \inf_{\alpha \in \R^k} \sum_{i=1}^k \prior_i \loss(\alpha, i)
  - \inf_{\discfunc \in \Discfuncs} \surrrisk(\discfunc).
\end{equation*}

\subsection{Generalized entropies and losses}
\label{sec:infimal-representations}

We construct a natural bidirectional mapping between losses and
generalized entropies, giving a few examples to illustrate.
For any loss $\loss : \R^k \times [k] \to \extendedR$,
the construction~\eqref{eqn:inf-representation-full} of $\entropy_\loss$
yields a closed concave function, as $\entropy_\loss$
is the infimum of linear functionals of
$\prior$. It is thus a generalized entropy~\cite{GrunwaldDa04} (or 
uncertainty function~\cite{DeGroot62}), and the
gap $\entropy_\loss(\prior) - \E[\entropy_\loss(\wt{\prior}(X))]$ 
between prior and posterior entropy is non-negative. 
The following two examples with zero-one loss are illustrative.

\begin{example}[Zero-one loss]
  \label{example:zero-one-to-uncertainty}
  Consider the zero one loss
  \begin{equation*}
    \zoloss(\alpha, y) = \indic{\alpha_y \le \alpha_j
      ~ \mbox{for~some}~ j \neq y}
  \end{equation*}
  where $y \in [k]$.
  Then we have
  \begin{equation*}
    \entropy_{\zoloss}(\prior)
    = \inf_\alpha \left\{\sum_{i=1}^k \prior_i \indic{
      \alpha_i \le \alpha_j ~ \mbox{for~some}~ j \neq i}\right\}
    = 1 - \max_j \prior_j.
  \end{equation*}
  This generalized entropy is concave,
  nonnegative, and
  satisfies $\entropy_{\zoloss}(\prior) = 0$ if and only if
  $\prior = e_i$ for a standard basis vector $e_i$.
\end{example}

\begin{example}[Cost-weighted classification]
  \label{example:weighted-misclassification}
  In some scenarios, we allow different costs for classifying certain
  classes $y$ as others; for example, it may be less costly to misclassify a
  benign tumor as cancerous than the opposite. In this case, we use a
  matrix $C = [c_{yi}]_{y,i=1}^k \in \R^{k \times k}_+$, where $c_{yi} \ge
  0$ is the cost for classifying an observation of class $y$ as class $i$
  (i.e.\ assigning $X \sim P_i$ instead of $P_y$ in the
  experiment). We assume $c_{yy} = 0$ for each $y$ and define
  \begin{equation}
    \label{eqn:weighted-zero-one}
    \wzoloss(\alpha, y) = \max_i\{c_{yi} \mid \alpha_i = \max_j \alpha_j \},
    ~~ \alpha \in \R^k,
  \end{equation}
  the maximal loss for those indices of $\alpha$ attaining $\max_j
  \alpha_j$. Let $C = \left[c_1 ~ \cdots ~ c_k \right]$ be the column
  representation of $C$. If $c_y^T \prior = \min_l c_l^T \prior$, then by
  choosing any $\alpha$ such that $\alpha_y > \alpha_j$ for all $j \neq y$,
  we have
  \begin{equation*}
    \entropy_{\wzoloss}(\prior)
    = \inf_\alpha \bigg\{\sum_{y=1}^k \prior_y \max_i\{c_{yi}
    \mid \alpha_i = \max_j \alpha_j\} \bigg\}
    = \min_l \prior^T c_l.
  \end{equation*}
  The entropy $\entropy_{\wzoloss}$ is concave, nonnegative, and
  satisfies $\entropy_{\wzoloss}(e_i) = 0$ for standard basis vectors $e_i$;
  Example~\ref{example:zero-one-to-uncertainty} corresponds to $C = \ones
  \ones^T - I_{k \times k}$.
\end{example}

The forward mapping~\eqref{eqn:inf-representation-full} from losses
$\loss$ to entropy $\entropy_\loss$ is straightforward, though it is
many-to-one. Using convex duality and conjugacy arguments, we can
show an inverse mapping. This construction is new, though
precursors for proper scoring rules and predictions in the probability
simplex exist~\cite[Thm.~2]{GrunwaldDa04, GneitingRa07}; these
\emph{characterize} proper scoring rules, but it is not always
clear how to generate \emph{convex} losses from these.  Before stating the
proposition, we recall the definition~\eqref{eqn:fenchel-conjugate} of the
Fenchel conjugate
$f^*(s) = \sup_t \{s^T t - f(t)\}$.
\begin{proposition}
  \label{proposition:u-to-loss-correspondence}
  For any closed concave $\entropy : \simplex_k \to
  \extendedRlow$, the losses
  \begin{equation}
    \label{eqn:loss-def}
    \loss(\cdot, i) : \R^k \to \extendedR,
    ~~~
    \loss(\alpha, i) = -\alpha_i + (-\entropy)^*(\alpha),
  \end{equation}
  $i \in \{1, \ldots, k\}$, are closed, convex, and satisfy the
  equality~\eqref{eqn:inf-representation-full}
  that $\entropy \equiv \entropy_L$.
\end{proposition}
\begin{proof}
  Standard Fenchel conjugacy
  relationships~\cite[Chapter X]{HiriartUrrutyLe93ab} imply
  \begin{equation*}
    \entropy(\prior) = \inf_{\alpha \in \R^k} \left\{- \prior^T \alpha
    + (-\entropy)^*(\alpha) \right\}
    ~~ \mbox{where} ~~ (-\entropy)^*(\alpha) = \sup_{\prior \in \simplex_k}
    \left\{\alpha^T \prior + \entropy(\prior)\right\}.
  \end{equation*}
  Defining $\loss(\alpha, i)
  = -\alpha_i + (-\entropy)^*(\alpha)$ for $i = 1, \ldots, k$,
  we can write
  \begin{align*}
    \entropy(\prior)
    = \inf_{\alpha \in \R^k} \left\{ -\prior^T\alpha + (-\entropy)^*(\alpha)\right\}
    & = \inf_{\alpha \in \R^k} \left\{ -\prior^T\alpha + \prior^T \ones
    \cdot (-\entropy)^*(\alpha)\right\} \\
    & = \inf_{\alpha \in \R^k} \bigg\{ \sum_{i=1}^k \prior_i
    \loss(\alpha, i) \bigg\}. \qedhere
  \end{align*}
\end{proof}

Proposition~\ref{proposition:u-to-loss-correspondence} shows that associated
with every concave entropy defined on the simplex, there is at least one set
of \emph{convex} loss functions $\loss(\cdot, i)$ generating the entropy via
the infimal representation~\eqref{eqn:inf-representation-full}, and there is
thus a mapping from loss functions to entropies and from entropies to
\emph{convex} losses: given any loss $\loss$, we may construct a convex loss
$\loss^{\rm cvx}$ with $\entropy_\loss = \entropy_{\loss^{\rm cvx}}$. The
mapping from entropies $\entropy$ to loss functions generating $\entropy$ as
in~\eqref{eqn:inf-representation-full} is one-to-many, as any losses
$\loss^{(1)}$ and $\loss^{(2)}$ with the same range satisfy
$\entropy_{\loss^{(1)}} = \entropy_{\loss^{(2)}}$.


\subsection{Surrogate risk consistency and generalized entropies}
\label{sec:uncertainty-consistency}

Our construction~\eqref{eqn:loss-def} of loss functions is a somewhat
privileged construction, as it often yields desirable properties of the
convex loss function itself, especially as related to the non-convex
zero-one loss.  Indeed, it is often the case that the convex loss $\loss$ so
generated is Fisher consistent; to make this explicit, we recall the
following definition~\cite{Zhang04a,TewariBa07}.
\begin{definition}
  \label{definition:classification-calibrated}
  Let $\loss : \R^k \times [k] \to \extendedR$.
  Then $\loss$ is \emph{classification
    calibrated for the zero-one loss} if for any $\prior \in \simplex_k$ and
  $i^*$ such that $\prior_{i^*} < \max_j \prior_j$,
  \begin{equation}
    \inf_{\alpha \in \R^k}
    \left\{\sum_{i=1}^k \prior_i \loss(\alpha, i)\right\}
    < \inf_{\alpha \in \R^k}
    \left\{\sum_{i=1}^k \prior_i \loss(\alpha, i) : \alpha_{i^*} \ge \max_j
    \alpha_j \right\}.
    \label{eqn:classification-calibrated}
  \end{equation}
  Given a matrix $C \in \R^{k \times k}_+$ as in
  Example~\ref{example:weighted-misclassification}, $\loss$ is
  \emph{classification calibrated for the cost matrix $C$} if for any
  $\prior \in \simplex_k$ and any $i^*$ with $c_{i^*}^T \prior > \min_j
  c_j^T \prior$,
  \begin{equation}
    \label{eqn:weighted-classification-calibrated}
    \inf_{\alpha \in \R^k}
    \left\{\sum_{i=1}^k \prior_i \loss(\alpha, i)\right\}
    < \inf_{\alpha \in \R^k}
    \left\{\sum_{i=1}^k \prior_i \loss(\alpha, i) : \alpha_{i^*} \ge \max_j
    \alpha_j \right\}.
  \end{equation}
\end{definition}
\noindent
\citet[Thm.~2]{TewariBa07} and \citet[Thm.~3]{Zhang04a} show the importance
of Definition~\ref{definition:classification-calibrated}: let
$\risk(\discfunc)$ be the zero-one or cost-weighted risk
(Exs.~\ref{example:zero-one-to-uncertainty}--\ref{example:weighted-misclassification}).
If $\loss$ is lower-bounded, then it is
classification calibrated (with respect to zero-one or the cost-weighted
loss) if and only if for any sequence $\discfunc_n : \mc{X} \to \R^k$ and
distribution $\P$ on $X \times Y$ we
have \emph{Fisher consistency}, i.e.
\begin{equation*}
  \surrrisk(\discfunc_n) \to
  \inf_{\discfunc \in \Discfuncs} \surrrisk(\discfunc)
  ~~ \mbox{implies} ~~
  \risk(\discfunc_n) \to \inf_{\discfunc \in \Discfuncs} \risk(\discfunc).
\end{equation*}
That is, classification calibration (with respect to zero-one-risk or 
the cost-weighted risk) is equivalent to surrogate risk consistency of the 
loss $\loss$.
Because of the predominance of the zero-one loss in the literature,
we use 
``classification calibration'' without any qualification to
mean ``classification calibration with respect to zero-one loss''.

We now show how---under minor restrictions on the generalized entropy function 
$\entropy$---the construction~\eqref{eqn:loss-def} yields classification
calibrated losses.
\begin{definition}
  \label{def:uniform-convex}
  A convex function $f : \R^k \to \extendedR$ is \emph{$(\lambda, \kappa,
    \norm{\cdot})$-uniformly convex}
  over $C \subset \R^k$ if it is closed
  and for all $t \in [0, 1]$ and $x_1, x_2 \in C$
  \begin{align*}
    \lefteqn{f(t x_1 + (1 - t) x_2)} \\
    & \qquad \le
    t f(x_1) + (1 - t) f(x_2)
    - \frac{\lambda}{2} t(1 - t) \norm{x_1 - x_2}^\kappa
    \left[(1 - t)^{\kappa - 1} + t^{\kappa - 1}\right].
  \end{align*}
\end{definition}
\noindent
We say, without qualification, that $f$ is uniformly convex on
$C$ if $\dom f \supset C$ and there exist $\lambda > 0$, a norm
$\norm{\cdot}$, and constant $\kappa < \infty$ such that
Definition~\ref{def:uniform-convex} holds; we say $f$ is uniformly
concave if $-f$ is uniformly convex.  Definition~\ref{def:uniform-convex} is
an extension of the usual notion of \emph{strong convexity}, which holds
when $\kappa = 2$, and is essentially a quantified notion of strict
convexity.

With this definition, we have the following two propositions.
These two propositions, whose proofs we provide in
Appendix~\ref{sec:calibration-proofs}, show that generalized entropies
naturally give rise to classification calibrated loss functions; we
provide examples of these results in
Section~\ref{sec:loss-uncertainty-examples} to come.
\begin{proposition}
  \label{proposition:calibration}
  Assume that $\entropy$ is closed concave, symmetric,
  and has $\dom \entropy = \simplex_k$, and let $\loss$ have
  definition~\eqref{eqn:loss-def}.  Additionally assume that (a) $\entropy$
  is strictly concave, and $\inf_\alpha \sum_{i=1}^k \prior_i
  \loss(\alpha,i)$ is attained for all $\prior \in \simplex_k$, or (b) $\entropy$
  is uniformly concave. Then $\loss$ is classification calibrated.
\end{proposition}
\noindent
Even when $\entropy$ is not strictly concave, we can give classification calibration
results. Indeed, recall Example~\ref{example:zero-one-to-uncertainty},
which showed that for the zero-one-loss, we have
$\entropy_\loss(\prior) = 1 - \max_j \prior_j$.
\begin{proposition}
  \label{proposition:zo-uncertainty-calibrated}
  Let $\entropy(\prior) = 1 - \max_j \prior_j$. The loss~\eqref{eqn:loss-def}
  defined by $\loss(\alpha,i) = -\alpha_i + (-\entropy)^*(\alpha)$ is
  classification calibrated. Moreover, we have for any
  $\prior \in \simplex_k$ and $\alpha \in \R^k$ that
  \begin{equation*}
    \sum_{i = 1}^k \prior_i \loss(\alpha, i)
    - \inf_{\alpha} \sum_{i = 1}^k \prior_i \loss(\alpha, i)
    \ge \frac{1}{k}
    \bigg(\sum_{i=1}^k \prior_i \zoloss(\alpha, i) - \inf_\alpha 
    \sum_{i=1}^k \prior_i \zoloss(\alpha, i) \bigg).
  \end{equation*}
\end{proposition}
\noindent

\subsection{Divergences, risk, and generalized entropies}
\label{sec:correspondence-u-f}

In this section, we show that $f$-divergences as in
Definition~\ref{def:general-fdiv} have a precise correspondence with
generalized entropies and losses; \citet{GarciaWi12} establish the
correspondence between $f$-divergences and entropy/pointwise Bayes risk
$\entropy$; our results show the important link from $f$ directly \emph{to}
the loss $\loss$. We begin as in
equation~\eqref{eqn:inf-representation-full} with a concave generalized
entropy $\entropy$ and loss $\loss$ satisfying $\entropy(\prior) =
\inf_{\alpha \in \R^k} \sum_{i=1}^k \prior_i \loss(\alpha, i)$; by
Proposition~\ref{proposition:u-to-loss-correspondence} it is no loss of
generality to assume this correspondence.  Let $\Discfuncs$ be the
collection of measurable functions $\discfunc : \mc{X} \to \R^k$.
The posterior Bayes risk for $\loss$ is
\begin{align}
  \label{eqn:posterior-risk-def}
  \postbayesrisk &
  \defeq \inf_{\discfunc \in \Discfuncs}
  \int_\mc{X} \sum_{i=1}^{k} \prior_i\loss(\discfunc(x), i)dP_i(x)
  = \E[\entropy_\loss(\wt{\prior}(X))],
\end{align}
where $\wt{\prior}(x)$ is the posterior distribution on $Y$ conditional
on $X = x$.
The information measure~\eqref{eqn:information} is thus
the gap between the prior Bayes $\loss$-risk and
posterior Bayes $\loss$-risk. We may then write
\begin{align*}
  \lefteqn{\inf_{\alpha \in \R^k}
  \sum_{i=1}^k \prior_i \loss(\alpha, i)
  - \inf_{\discfunc \in \Discfuncs} \surrrisk(\discfunc)
  =
  \priorrisk  - \postbayesrisk = \sif} \\
  &  = \int_\mc{X} \sup_\alpha
  \left( \priorrisk - \sum_{i=1}^{k - 1} \prior_i\loss(\alpha, i)
  \frac{dP_i}{dP_k}
  - \prior_k \loss(\alpha, k) \right) dP_k 
  = \genfdiv{f_{\loss,\prior}}{P_{1:k-1}}{P_k},
\end{align*}
where the closed convex function
$f_{\loss, \prior} : \R_+^{k-1} \to \extendedR$ has definition
\begin{equation}
  \label{eqn:f-from-loss}
  f_{\loss,\prior}(t) \defeq
  \sup_{\alpha \in \R^k}
  \left(\priorrisk - \sum_{i=1}^{k-1} \prior_i \loss(\alpha, i) t_i
  - \prior_k \loss(\alpha, k) \right). 
\end{equation}
As $f_{\loss,\prior}$ is the supremum of affine functions of its argument
$t$, it is closed convex, and $f_{\loss,\prior}(\ones) = \entropy_\loss(\prior) -
\entropy_\loss(\prior) = 0$.  That is, equation~\eqref{eqn:f-from-loss} shows that
given any loss $\loss$ or generalized entropy
$\entropy$, the information measure $\sif$, gap between prior and
posterior $\loss$-risk, and $f_{\loss,\prior}$-divergence between
distributions $P_1, \ldots, P_{k-1}$ and $P_k$ are identical.

We can also give a converse result that shows that every $f$-divergence can
be written as the gap between prior and posterior risks for a convex loss
function. 
We first recall the result that
$\fdiv{P_{1:k-1}}{P_k}$ is a statistical information~\eqref{eqn:information}
based on an generalized entropy $\entropy$ 
associated with $f$.
Except for the closure operation, this result is
known~\cite[Thm.~3]{GarciaWi12}.
\begin{proposition}
  \label{proposition:f-to-entropy}
  For closed convex $f : \R^{k-1} \to \extendedR$ with
  $f(\ones) = 0$,
  let
  \begin{equation*}
    \entropy(t_1, \ldots, t_k)
    = -k t_k f\left(\frac{t_1}{t_k}, \ldots, \frac{t_{k-1}}{t_k}\right),
  \end{equation*}
  where we implicitly use the closure of the perspective
  (Def.~\eqref{eqn:closed-perspective}). Then
  \begin{equation*}
    \fdiv{P_1, \ldots, P_{k-1}}{P_k} =
    \entropy(\ones/k) - \E[\entropy(\wt{\prior}(X))],
  \end{equation*}
  where the prior $\prior = \ones / k$ and
  the expectation is taken according to $\sum_i \prior_i P_i$.
\end{proposition}

By combining Propositions~\ref{proposition:u-to-loss-correspondence}
and~\ref{proposition:f-to-entropy} with the infimal
representation~\eqref{eqn:inf-representation-full} of $\entropy_\loss$, we
immediately obtain the following corollary, which
is our explicit construction of
a closed convex loss from an $f$-divergence.
\begin{corollary}
  \label{corollary:f-to-loss}
  \newcommand{\nullprior}{\prior^0}
  Let $\nullprior = \ones/k$. For any closed and convex function $f
  : \R^{k-1} \to \extendedR$ such that
  $f(\ones) = 0$, the convex losses defined by
  \begin{equation*}
    \loss(\alpha, i)
    = -\alpha_i + \sup_{\prior \in \simplex_k}
    \left\{\prior^T \alpha
    - k \prior_k f\left(\frac{\prior_1}{\prior_k},
    \ldots, \frac{\prior_{k-1}}{\prior_k}\right) \right\}
  \end{equation*}
  satisfy Eq.~\eqref{eqn:f-from-loss}, i.e.\
  $f(t)
  = \sup_{\alpha}
  \{\entropy_\loss(\nullprior) - \sum_{i=1}^{k-1} \prior_i^0 \loss(\alpha, i) t_i
  - \prior_k^0 \loss(\alpha, k)\}$. Additionally,
  \begin{equation*}
    \fdiv{P_{1:k-1}}{P_k}
    = \inf_{\alpha \in \R^k}
    \sum_{i=1}^k \nullprior_i \loss(\alpha, i)
    - \inf_\discfunc \E[\loss(\discfunc(X), Y)],
  \end{equation*}
  where the expectation is over
  $Y \sim \nullprior$ and $X \sim P_y$
  conditional on $Y = y$.
\end{corollary}
\noindent
For binary classification problems,
\citet[Thm.~1]{NguyenWaJo09} provide an explicit construction
of a closed convex margin-based loss inducing the $f$-divergence
as in Eq.~\eqref{eqn:f-from-loss}; the binary case
allows a complete characterization of all such
convex functions, which appears difficult in the multiclass case.

Corollary~\ref{corollary:f-to-loss}, coupled with the information
representation given by the $f$-divergence~\eqref{eqn:f-from-loss}, shows
the complete equivalence between $f$-divergences,
loss functions $\loss$, and entropies $\entropy$.  For any
$f$-divergence, there exists a loss function $\loss$ and prior $\prior =
\ones/k$ such that $\fdiv{P_{1:k-1}}{P_k} = \priorrisk -
\postbayesrisk$. Conversely, for any loss function $\loss$ and
prior $\prior$, there exists a
multi-way $f$-divergence such that the gap $\priorrisk
- \postbayesrisk = \fdiv{P_{1:k-1}}{P_k}$.

\subsection{Examples of generalized entropies and loss correspondences}
\label{sec:loss-uncertainty-examples}

To complement our general results, we illustrate the correspondence between
(concave) generalized entropies and the loss
construction~\eqref{eqn:loss-def} through several examples, using
Propositions~\ref{proposition:calibration}
and~\ref{proposition:zo-uncertainty-calibrated} to guarantee classification
calibration.

\begin{example}[Zero-one loss, Example~\ref{example:zero-one-to-uncertainty},
    continued]
  \label{example:max-order}
  We use the generalized entropy $\entropy(\prior) = 1 - \max_j
  \prior_j$ generated by the zero-one loss to derive a convex loss
  function $\loss$ that gives the same entropy via the
  representation~\eqref{eqn:inf-representation-full}.
  The conjugate to $-\entropy$ is
  \begin{equation}
    (-\entropy)^*(\alpha)
    = 1 + \max\bigg\{\alpha_{(1)} - 1,
    \frac{\alpha_{(1)} + \alpha_{(2)}}{2} - \half,
    \ldots, \frac{\sum_{i=1}^k \alpha_{(i)}}{k} - \frac{1}{k} \bigg\},
    \label{eqn:conjugate-zo-loss}
  \end{equation}
  where $\alpha_{(1)} \ge \alpha_{(2)} \ge \cdots$ are the entries of
  $\alpha \in \R^k$ in sorted order (see
  \S~\ref{sec:proof-conjugate-zo-loss} for a proof).  Then the convex
  ``family-wise'' loss, named for its similarity to family-wise error
  control in hypothesis tests,
  \begin{equation*}
    \maxloss(\alpha, i) = 1 - \alpha_i + \max_{l \in \{1, \ldots, k\}}
    \bigg\{\frac{1}{l} \sum_{j = 1}^l \alpha_{(j)}
    - \frac{1}{l} \bigg\}
  \end{equation*}
  generates the same entropy $\entropy_{\maxloss}$ and associated
  $f$-divergence as the zero-one loss. Moreover,
  Proposition~\ref{proposition:zo-uncertainty-calibrated} guarantees that
  $\maxloss$ is classification calibrated
  (Def.~\ref{definition:classification-calibrated}). It appears that the
  loss $\maxloss$ is a new convex classification-calibrated loss
  function.
\end{example}

Rather than re-considering Example~\ref{example:weighted-misclassification},
which we do later in the context of showing that distinct convex losses can
yield the same generalized entropy, we now consider the multiclass logistic loss.
The loss does \emph{not} correspond to the zero-one loss, but it generates
Shannon entropy and information.

\begin{example}[Logistic loss and entropy]
  \label{example:logistic-loss}
  For $1\leq i\leq k$, define  $p_i(\alpha) = e^{\alpha_i}/\sum_{j=1}^k e^{\alpha_j}$.
  The multi-class logistic loss is then
  \begin{equation*}
    \loss(\alpha, i) = -\log p_i(\alpha)
    = \log \bigg(\sum_{j = 1}^k e^{\alpha_j - \alpha_i}\bigg)
    ~~ \mbox{for}~~1\leq i\leq k. 
  \end{equation*}
  The entropy associated with the loss is the
  familiar Shannon entropy,
  \begin{equation}
    \label{eqn:logistic-bayes-risk}
    \entropy_\loss(\prior) = \inf_{\alpha\in \R^k} \bigg\{
    -\sum_{i=1}^k \prior_i \log p_i (\alpha )\bigg\}
    = -\sum_{i=1}^k \prior_i \log \prior_i.
  \end{equation}
  The conjugacy calculation~\eqref{eqn:loss-def} (i.e.\
  our inverse construction from $\entropy$ to loss $\loss$) to
  generate $\loss$ also yields the multi-class logistic loss.
  That the multiclass logistic loss is calibrated for the zero-one
  loss~\cite[\S~4.4]{Zhang04a} is now immediate: the negative Shannon
  entropy is strongly convex over the simplex $\simplex_k$ (this is
  Pinsker's inequality~\cite[Ch.~17.3]{CoverTh06}), so the fact that
  logistic loss and Shannon entropy are dual via Eq.~\eqref{eqn:loss-def}
  and Proposition \ref{proposition:calibration} yield calibration.
  The information measure~\eqref{eqn:information} associated with the
  logistic loss is the mutual information between the
  observation $X$ and label $Y$. Indeed, we have
  \begin{align*}
    \infoentropy
    = \entropy(\prior) - \E[\entropy(\wt{\prior}(X))]
    & = H(Y) - \int_{\mathcal{X}} H(Y \mid X= x) d\wb{P}(x) \\
    & = H(Y) - H(Y \mid X)
    = I(X ; Y)
  \end{align*}
  where $H$ denotes the Shannon entropy, $\wb{P} = \sum_{i=1}^k \prior_i
  P_i$, and $I(X; Y)$ is the usual (Shannon) mutual information between $X$ and
  $Y$.
\end{example}

We include one final example to show that in some instances, many different
\emph{convex} losses can yield the same generalized entropy $\entropy$.

\begin{example}[Hinge losses]
  \label{example:hinge-loss}
  Define the pairwise multiclass hinge loss
  \begin{equation*}
    \hingeloss(\alpha, i) = \sum_{j \neq i}
    \hinge{1 + \alpha_j} + \bindic{\ones^T \alpha = 0}.
  \end{equation*}
  We also consider the slight extension to weighted loss functions to address
  asymmetric losses of the form~\eqref{eqn:weighted-zero-one}
  from Example~\ref{example:weighted-misclassification}. In this case,
  given the loss matrix $C \in \R^{k \times k}_+$, we set
  \begin{equation*}
    \hingeloss(\alpha, i) = \sum_{j = 1}^k c_{ij} \hinge{1 + \alpha_j}
    + \bindic{\ones^T \alpha = 0}.
  \end{equation*}
  The loss $\loss(\alpha, i) = \sum_{j \neq i} c_{ij} \hinge{1 + \alpha_j -
    \alpha_i}$ yields a completely identical set of calculations without the
  restriction $\ones^T \alpha = 0$, as it is invariant to
  shifts.  A calculation (see
  \S~\ref{sec:proof-conjugate-hinge-loss} for completeness) shows the
  generalized entropy~\eqref{eqn:inf-representation-full}
  associated with the hinge loss with loss matrix $C = [c_1 ~ \cdots ~
    c_k]$ is
  \begin{equation}
    \label{eqn:conjugate-hinge-loss}
    \entropy_{\hingeloss}(\prior) =
    \inf_{\alpha\in \R^k}
    \left\{\sum_{i=1}^{k} \prior_i \hingeloss(\alpha, i)\right\}
    = k \min_l \prior^T c_l.
  \end{equation}

  Such losses satisfy a number of classification calibration
  guarantees; we note one,
  essentially due to \citet[Thm.~8]{Zhang04a}.  For completeness, we include
  a proof in Appendix~\ref{sec:proof-weighted-calibration}.
  \begin{observation}
    \label{observation:weighted-calibration}
    Let $\phi : \R \to \R$ be any bounded
    below convex function,
    differentiable on $\openleft{-\infty}{0}$, with  $\phi'(0) < 0$.
    Then $\loss(\alpha, y) = \sum_{i = 1}^k c_{yi} \phi(-\alpha_i)$
    is classification calibrated for the cost matrix $C$
    (Def.~\ref{definition:classification-calibrated},
    Eq.~\eqref{eqn:weighted-classification-calibrated}).
  \end{observation}
  \noindent
  Taking $C = \ones\ones^T - I_{k \times k}$, we see that the hinge loss is
  calibrated for the zero-one loss
  (Ex.~\ref{example:zero-one-to-uncertainty}); taking arbitrary $C \in \R^{k
    \times k}_+$, the weighted hinge loss is calibrated for the cost matrix
  $C$. Even more, we have the following
  quantitative calibration guarantee
  in analogy with
  Proposition~\ref{proposition:zo-uncertainty-calibrated}:
  \begin{equation*}
    \sum_{i = 1}^k \prior_i \hingeloss(\alpha, i)
    - \inf_{\alpha'} \sum_{i = 1}^k \prior_i \hingeloss(\alpha',i)
    \ge \sum_{i=1}^k \prior_i \wzoloss(\alpha, i)
    - \inf_{\alpha'} \sum_{i = 1}^k \prior_i \wzoloss(\alpha',i)
  \end{equation*}
  for all $\prior \in \simplex_k$ and $\alpha \in \R^k$, strengthening
  Observation~\ref{observation:weighted-calibration}.  (We prove this as
  Lemma~\ref{lemma:hinge-great-gap} in
  \S~\ref{sec:hinge-loss-calibration}.)
  In the binary case~\cite[Lem.~2]{NguyenWaJo09},
  similar quantitative guarantees hold for
  \emph{any} margin-based classification calibrated loss
  $\loss$ for which $\entropy_\loss = \entropy_{\zoloss}$; we do not know
  if this extends to the multiclass case.
\end{example}


\section{Comparison of loss functions}
\label{sec:equiv}

In Section~\ref{sec:relation}, we demonstrated the correspondence between
loss functions, generalized entropies, statistical information,
$f$-divergences, and (in restricted cases)
classification calibration.
These correspondences assume that decision makers have
access to the entire observation $X$, which is often not the case; as noted
in the introduction, it is often beneficial to pre-process data to make it
lower dimensional, communicate or store it efficiently, or to improve
statistical behavior.
Thus, we now explore the impact quantization has on these concepts.


To motivate this further, consider that each of the family-wise loss
$\maxloss$ of Ex.~\ref{example:max-order}, logistic loss
(Ex.~\ref{example:logistic-loss}), and any loss of the form $\loss(\alpha, y)
= \sum_{i \neq y} \phi(-\alpha_i)$ for $\phi$ convex and decreasing with
$\phi'(0) < 0$ (Ex.~\ref{example:hinge-loss},
Observation~\ref{observation:weighted-calibration}) is classification
calibrated. This relates to one of the major criticisms of classification
calibration: if the Bayes classifier (minimizer of risk over all
functions $\mc{X} \to \R^k$) does not belong to the class of functions
considered, classification calibration says little.  In this context, we
shed light on this issue by identifying losses that are consistent
(calibrated) even with the additional selection of quantizer or data
representation---a restriction of the class of possible
functions as in the implication~\eqref{eqn:restricted-risk-convergence} in
the introduction.



\subsection{A model of quantization and experimental design}
\label{sec:quantization-model}

Abstractly, we treat the design of an experiment or choice of data
representation as a quantization problem, where a quantizer $\quant$ maps
the space $\mc{X}$ to a measurable space $\mc{Z}$.  Then, for a
loss $\loss$, prior $\prior \in \simplex_k$ on the
label $Y$, and discriminant $\discfunc : \mc{Z} \to \R^k$,
we consider the quantized risk~\eqref{eqn:quantized-risk},
which we recall is
\begin{equation*}
  \quantrisk
  \defeq \E\left[\loss (\discfunc(\quant(X)), Y)\right].
\end{equation*}
Given class-conditional
distributions $P_{1:k}$ (equivalently, hypotheses $H_i : P_i$ in
the Bayesian testing setting) and collection $\quantfam$ of quantizers, our
criterion is to choose the quantizer $\quant$ that allows the best
attainable risk. That is, we consider the quantized Bayes $\loss$-risk,
defined as the infimum of the risk~\eqref{eqn:quantized-risk} over
discriminants $\Discfuncs = \{\discfunc : \mc{Z} \to
\R^k\}$,
\begin{equation}
  \label{eqn:quantized-bayes-risk}
  \inf_{\discfunc \in \Discfuncs}
  \quantrisk
  = \int_{\mc{Z}}
  \inf_\alpha \sum_{i = 1}^k
  \prior_i \loss(\alpha, i) dP_i^\quant(z)
\end{equation}
where $P^\quant(A) = P(\quant^{-1}(A))$ denotes the push-forward measure.
The risk~\eqref{eqn:quantized-bayes-risk} measures the best
attainable risk for a fixed choice of $\quant \in \quantfam$;
one thus seeks the design $\quant$ giving the lowest
quantized Bayes $\loss$-risk.

Whether for computational or analytic reasons, minimizing the
loss~\eqref{eqn:quantized-bayes-risk} is often intractable; the zero-one
loss $\zoloss$ (Ex.~\ref{example:zero-one-to-uncertainty}), for example, is
non-convex and discontinuous. It is thus of interest to understand the
asymptotic consequences of using a surrogate loss $\loss$ in place of the
desired loss (say
$\zoloss$)~\cite{Zhang04a,LugosiVa04,BartlettJoMc06,TewariBa07}, including
the setting in which one incorporates further dimension reduction via
the choice $\quant \in \quantfam$. \citet{NguyenWaJo09} introduce and
study this problem for binary classification, giving a
correspondence between $f$-divergences, loss functions, and surrogate
consistency with quantization.  The consequences of using a surrogate for
consistency of the resulting quantization and classification procedure in
the multiclass case are \emph{a-priori} unclear: we do not know when using
such a surrogate can be done without penalty.  To that end, we now
characterize when two loss functions $\loss^{(1)}$ and $\loss^{(2)}$ provide
equivalent criteria for choosing quantizers (experimental designs or data
representations) according to the Bayes
$\loss$-risk~\eqref{eqn:quantized-bayes-risk}.


\subsection{Universal Equivalence of Loss Functions}
\label{sec:universal-equivalence-losses}

Recalling our arguments in Section~\ref{sec:correspondence-u-f} that
statistical information (the gap between prior and posterior risks) is a
multi-way $f$-divergence between distributions $P_1, \ldots, P_{k-1}$ and
$P_k$, we give a quantized version of this construction. In analogy with the
results of Section~\ref{sec:correspondence-u-f}, the quantized statistical
information is
\begin{equation}
  \label{eqn:quantized-gap}
  \begin{split}\lefteqn{\qsif
      \defeq
      \entropy_\loss(\prior) - \E[\entropy_\loss(\wt{\prior}(\quant(X)))]} \\
    & = \inf_{\alpha \in \R^k}
    \sum_{i=1}^k \prior_i \loss(\alpha, i)
    - \inf_\discfunc \quantrisk
    = \genfdiv{f_{\loss,\prior}}{P_1, \ldots, P_{k-1}}{P_k \mid \quant},
  \end{split}
\end{equation}
where $\entropy_\loss(\prior) = \inf_{\alpha \in \R^k} \sum_{i=1}^k \prior_i
\loss(\alpha, i)$ as in~\eqref{eqn:inf-representation-full}, the convex
function
$f_{\loss,\prior}$ is defined as in expression~\eqref{eqn:f-from-loss} and
does not depend on the quantizer $\quant$, and $\wt{\prior}(\quant(X))$
denotes the posterior distribution on $Y \in [k]$ conditional on observing
$\quant(X)$.
We extend
\citet{NguyenWaJo09}'s notion of universal equivalence from the binary
case, defining losses as
equivalent if they induce the same ordering of quantizers
$\quant$ under the information measure~\eqref{eqn:quantized-gap}.
\begin{definition}
  \label{def:equivalent-losses}
  Loss functions $\loss^{(1)}$ and $\loss^{(2)}$ are
  \emph{universally equivalent} for the prior $\prior$, denoted
  $\loss^{(1)} \uniequiv_\prior \loss^{(2)}$, if for any distributions
  $P_1, \ldots, P_k$ on $X$ and
  quantizers $\quant_1$ and $\quant_2$
  \begin{align*}
      & \quantinfo{\entropy_{\loss^{(1)}}}{\quant_1}
      \le \quantinfo{\entropy_{\loss^{(1)}}}{\quant_2}
      ~~ \mbox{if and only if} ~~ \\
      & ~~ \quantinfo{\entropy_{\loss^{(2)}}}{\quant_1}
      \le \quantinfo{\entropy_{\loss^{(2)}}}{\quant_2}.
  \end{align*}
\end{definition}
\noindent
Definition~\ref{def:equivalent-losses}
evidently is equivalent to the ordering condition
\begin{equation}
  \label{eqn:loss-universal-equivalence}
  \begin{split}
    & \inf_\discfunc \risk_{\loss^{(1)},\prior}(\discfunc \mid \quant_1)
    \le \inf_\discfunc \risk_{\loss^{(1)}, \prior}(\discfunc \mid \quant_2)
    ~~ \mbox{if and only if} \\
    & ~~ \inf_\discfunc \risk_{\loss^{(2)},\prior}(\discfunc \mid \quant_1)
    \le \inf_\discfunc \risk_{\loss^{(2)}, \prior}(\discfunc \mid \quant_2),
  \end{split}
\end{equation}
for all distributions $P_1, \ldots, P_k$, on the quantized Bayes
$\loss$-risk~\eqref{eqn:quantized-bayes-risk}.  This definition is somewhat
stringent: losses are universally equivalent only if they induce the same
quantizer ordering for all population distributions. If a quantizer
$\quant_1$ is finer than $\quant_2$, all losses yield
$\quantinfo{\entropy_\loss}{\quant_2} \le \quantinfo{\entropy_\loss}{\quant_1}$ by the
data processing inequality (Corollary~\ref{corollary:data-processing} of
Section~\ref{sec:f-divergence}).  The stronger equivalence notion is
important for nonparametric classification settings in which the underlying
distribution on $(X, Y)$ is only weakly constrained and neither of a pair of
quantizers $\quant_1, \quant_2 \in \quantfam$ is finer than the other.

Definition~\ref{def:equivalent-losses} and the
representation~\eqref{eqn:quantized-gap} suggest that the entropy function
$\entropy_\loss$ associated with the loss $\loss$ through the infimal
representation~\eqref{eqn:inf-representation-full} and the $f$-divergence
associated with $\loss$ via the construction~\eqref{eqn:f-from-loss} are
important for the equivalence of two loss functions. This is indeed
the case. First, we have the following result on universal equivalence
of loss functions based on their associated entropies.
\begin{theorem}
  \label{theorem:loss-equivalent-entropy}
  Let $\loss^{(1)}$ and $\loss^{(2)}$ be bounded below losses
  and $\entropy_{\loss^{(1)}}$ and $\entropy_{\loss^{(2)}}$ be the associated
  generalized entropies as in the
  construction~\eqref{eqn:inf-representation-full}. Then
  $\loss^{(1)}$ and $\loss^{(2)}$ are universally equivalent with respect 
  to all priors $\prior$ if and
  only if there exist $a > 0, b \in \R^k$, and $c \in \R$ such that
  for all $\prior \in \simplex_k$,
  \begin{equation*}
    \label{eqn:u-loss-equivalent}
    \entropy_{\loss^{(1)}}(\prior) = a \entropy_{\loss^{(2)}}(\prior) + b^T\prior + c.
  \end{equation*}
\end{theorem}
\noindent
We can also
characterize universal equivalence for a prior
$\prior$. 
\begin{theorem}
  \label{theorem:loss-equivalent}
  Let $\prior \in \simplex_k$ and as in
  Theorem~\ref{theorem:loss-equivalent-entropy}, let $\loss^{(1)}$ and
  $\loss^{(2)}$ be bounded below loss functions, with $f^{(1)}_\prior$ and
  $f^{(2)}_\prior$ the associated $f$-divergences as in the
  construction~\eqref{eqn:f-from-loss}.  Then $\loss^{(1)}$ and
  $\loss^{(2)}$ are universally equivalent with respect to the prior
  $\prior$ if and only if there exist $a > 0, b \in \R^{k-1}$, and $c
  \in \R$ such that
  \begin{equation}
    \label{eqn:rel-f-loss-equivalent}
    f^{(1)}_{\pi}(t) = a f^{(2)}_{\pi}(t) + b^T t + c
    ~~~ \mbox{for~all~} t \in \R_+^{k-1}.
  \end{equation}
\end{theorem}
\noindent
\citet[Thm.~3]{NguyenWaJo09} prove Theorem~\ref{theorem:loss-equivalent}
for binary classification problems ($k = 2$),
using convex-conjugacy arguments. We outline our
proofs (which apply for arbitrary $k$ and so require a different
set of tools) in Section~\ref{sec:proof}.

\subsection{Consistency of empirical risk minimization}
\label{sec:consistency-erm}

A major application of these theorems is to show that
certain non-convex loss functions (such as the zero-one loss) are
universally equivalent to convex loss functions, including variants of the
hinge loss, by showing that their associated entropies are
scalar multiples.  As a first application of
Theorems~\ref{theorem:loss-equivalent-entropy} and~\ref{theorem:loss-equivalent},
however, we consider the Bayes consistency of empirical risk minimization
for selecting a discriminant $\discfunc$ and quantizer $\quant$
(in analogy with~\cite[Thm.~2]{NguyenWaJo09}). In this case, we receive
a sample $\{(X_1, Y_1), \ldots, (X_n, Y_n)\}$ and define
the empirical risk
\begin{equation*}
  \what{\risk}_{\loss,n}(\discfunc \mid \quant)
  \defeq \frac{1}{n} \sum_{i = 1}^n \loss(\discfunc(\quant(X_i)), Y_i).
\end{equation*}
Now, let $\quantfam_1 \subset \quantfam_2 \subset \cdots \subset \quantfam$
be a non-decreasing collection of quantizers, indexed by sample size $n$,
and similarly let $\Discfuncs_1 \subset \Discfuncs_2 \subset \cdots \subset
\Discfuncs$ be a non-decreasing collection of discriminant functions, where
we assume the collections satisfy the estimation and approximation
error conditions
  \begin{equation}
    \label{eqn:error-estimates}
    \begin{split}
    \E\left[\sup_{\discfunc \in \Discfuncs_n, \quant \in \quantfam_n}
      \left|\what{\risk}_{\loss,n}(\discfunc \mid \quant) -
      \risk_\loss(\discfunc \mid \quant) \right|\right]
    & \le \epsestimate_n \\
    \inf_{\discfunc \in \Discfuncs_n, \quant \in \quantfam_n}
    \risk_\loss(\discfunc \mid \quant)
    - \inf_{\discfunc \in \Discfuncs, \quant \in \quantfam}
    \risk_\loss(\discfunc \mid \quant)
    & \le \epsapprox_n,
    \end{split}
  \end{equation}
where $\epsestimate_n \to 0$ and $\epsapprox_n \to 0$ as $n \to \infty$.
Additionally, let $\risk$ be the risk functional for the cost-weighted
misclassification loss $\wzoloss$
(Example~\ref{example:weighted-misclassification}), where
$\wzoloss(\alpha, y) = \max_i \{c_{yi} \mid \alpha_i = \max_j \alpha_j\}$.
Then we have the following result.
\begin{theorem}
  \label{theorem:bayes-consistency}
  Assume the conditions~\eqref{eqn:error-estimates} and that
  $\discfunc_n$ and $\quant_n$ are approximate empirical
  $\loss$-risk minimizers satisfying
  \begin{equation*}
    \epsopt_n \defeq
    \E\Big[\what{\risk}_{\loss,n}(\discfunc_n \mid \quant_n) - \inf_{\discfunc
        \in \Discfuncs_n, \quant \in \quantfam_n}
      \what{\risk}_{\loss,n}(\discfunc \mid \quant)\Big] \to 0
    ~~ \mbox{as}~ n \to \infty.
  \end{equation*}
  Let $\risk\opt(\quantfam) = \inf_{\discfunc \in \Discfuncs, \quant \in
    \quantfam} \risk(\discfunc \mid \quant)$.  If the loss $\loss$ is
  classification calibrated and universally
  equivalent to the cost-weighted loss $\wzoloss$, then
  \begin{equation*}
    \risk(\discfunc_n \mid \quant_n)
    - \risk\opt(\quantfam) \clp{1} 0.
  \end{equation*}
\end{theorem}

Theorem~\ref{theorem:bayes-consistency} guarantees that under the
estimation and approximation conditions~\eqref{eqn:error-estimates},
empirical risk minimization is consistent for minimizing
the quantized Bayes risk whenever the loss $\loss$ is classification
calibrated and equivalent to the desired loss.
The proof of Theorem~\ref{theorem:bayes-consistency} reposes on the
following risk inequality, which may be of independent interest.  The lemma
is a consequence of the results on surrogate risk consistency for
classification calibration~\cite{Zhang04a,TewariBa07,Steinwart07} and our
universal equivalence guarantees that exhibits the power of calibration and
universal equivalence.
\begin{lemma}
  \label{lemma:fisher-gap}
  Assume $\loss$ is classification-calibrated and universally
  equivalent to the weighted misclassification loss $\wzoloss$ with cost
  matrix $C \in \R^{k \times k}_+$.
  Then there exists a continuous concave function
  $h$ 
  with $h(0) = 0$ such that
  \begin{equation*}
    \risk(\discfunc \mid \quant)
    - \inf_{\discfunc \in \Discfuncs, \quant \in \quantfam}
    \risk(\discfunc \mid \quant)
    \le h\left(\surrrisk(\discfunc \mid \quant)
    - \inf_{\quant \in \quantfam} \surrrisk\opt(\quant)\right).
  \end{equation*}
  With the choice $\loss(\alpha, y) = \sum_{i=1}^k c_{yi} \hinge{1 +
    \alpha_i} + \bindic{\ones^T \alpha = 0}$ or $\loss(\alpha, y) = \sum_{i
    = 1}^k c_{yi} \hinge{1 + \alpha_i - \alpha_y}$,
  we may take
  $h(\epsilon) = (1 + \frac{1}{k})\epsilon$,
  that is,
  \begin{equation*}
    \risk(\discfunc \mid \quant) - \inf_{\discfunc \in \Discfuncs,
      \quant \in \quantfam} \risk(\discfunc \mid \quant)
    \le \left(1 + \frac{1}{k}\right) \left[
      \surrrisk(\discfunc \mid \quant)
      - \inf_{\discfunc \in \Discfuncs,
        \quant \in \quantfam} \surrrisk(\discfunc \mid \quant)
      \right].
  \end{equation*}
\end{lemma}
\noindent
Lemma~\ref{lemma:fisher-gap} shows that the gap in \emph{surrogate} risk
provides a guaranteed upper bound on the true \emph{cost-weighted} risk; in
the case of the modified hinge losses of Example~\ref{example:hinge-loss},
this gap is linear.  In the binary case, even stronger results are
possible~\cite[Lemma~2]{NguyenWaJo09}---one may
take $h(\epsilon) = a \epsilon$ (for some $a < \infty$)
in Lemma~\ref{lemma:fisher-gap} for any margin-based
classification-calibrated loss universally equivalent to the 0-1
loss---this relies
on the specific form any such binary convex loss must
take~\cite[Eq.~(9)]{NguyenWaJo09}; our Examples~\ref{example:max-order} (the
family-wise loss) and~\ref{example:hinge-loss} show that fairly
different-looking losses can be classification calibrated and universally
equivalent to zero-one loss.  We provide the proof of
Lemma~\ref{lemma:fisher-gap} in \S~\ref{appendix:proof-fisher-gap}.
Theorem~\ref{theorem:bayes-consistency}, which we prove in
\S~\ref{appendix:proof-bayes-consistency}, is then a consequence of this
lemma and Theorem~\ref{theorem:loss-equivalent-entropy}.


\subsection{Examples of universal equivalence}
\label{sec:examples-universal-equivalence}

In this section, we give several examples that build off of
Theorems~\ref{theorem:loss-equivalent-entropy} and~\ref{theorem:loss-equivalent},
showing that there exist convex losses that allow optimal joint
design of quantizers (or measurement strategies) and
discriminant functions, opening the way for potentially efficient convex
optimization strategies.  To that end, we give two hinge-like
loss functions that are universally equivalent to the zero-one loss for all
prior distributions $\prior$.  We also give examples of classification
calibrated loss functions that are not universally equivalent to the
zero-one loss, although minimizing them without quantization yields
Bayes-optimal classifiers.

\begin{example}[Cost-weighted losses]
  \label{example:hinge-loss-equiv}
  We return to Example~\ref{example:hinge-loss}, where we have
  $\hingeloss(\alpha, i) = \sum_{j \neq i} c_{ij} \hinge{1 + \alpha_j} +
  \bindic{\ones^T\alpha = 0}$. In this case, we have
  $\entropy_{\hingeloss}(\prior) = k \min_l \prior^T c_l = k
  \entropy_{\wzoloss}(\prior)$, where $\wzoloss$ denotes the cost-weighted
  misclassification error as in
  Example~\ref{example:weighted-misclassification}.
  Theorem~\ref{theorem:loss-equivalent-entropy} immediately guarantees that
  the (weighted) hinge loss is universally equivalent to the (weighted) 0-1
  loss. The weighted hinge loss $\hingeloss$ is also, as in
  Example~\ref{example:hinge-loss}, calibrated for the cost-weighted
  misclassification error.
\end{example}

\begin{example}[Max-type losses and zero-one loss]
  We return to Example~\ref{example:max-order} and let
  $\maxloss(\alpha, i) = 1 - \alpha_i + \max\{
  \alpha_{(1)} - 1, \frac{\alpha_{(1)} + \alpha_{(2)}}{2} - \half,
  \ldots, \frac{\ones^T \alpha}{k} - \frac{1}{k}\}$,
  the convex family-wise loss.
  By Example~\ref{example:max-order}, the associated
  entropy is
  $\entropy_{\maxloss}(\prior) = 1 - \max_j \prior_j = \entropy_{\zoloss}$
  for $\prior \in \simplex_k$, and
  Proposition~\ref{proposition:zo-uncertainty-calibrated} shows that
  $\maxloss$ is classification calibrated. We
  thus have that $\maxloss$ and the zero-one loss
  $\zoloss$ are universally equivalent by
  Theorems~\ref{theorem:loss-equivalent-entropy}
  and~\ref{theorem:loss-equivalent}.
\end{example}

For our final example, we consider the logistic loss, which is
classification calibrated but not universally
equivalent to the zero-one loss.

\begin{example}[Logistic loss]
  The loss
  $\logitloss(\alpha, i) = \log(\sum_{j=1}^k e^{\alpha_j - \alpha_i})$ has
  (Shannon) entropy $\entropy(\prior) = -\sum_{i=1}^k \prior_i \log \prior_i$,
  as in Ex.~\ref{example:logistic-loss}.
  There are no $a, b,
  c$ such that $\entropy_{\zoloss}(\prior) = 1 - \max_j \prior_j = a
  \entropy_{\logitloss}(\prior) + b^T\prior + c$ for all $\prior \in
  \simplex_k$. Theorem~\ref{theorem:loss-equivalent-entropy} shows that the
  logistic loss is not universally equivalent to the zero-one loss.
  That is, in spite of its classification calibration, there are
  distributions $P_1, \ldots, P_k$, a collection
  $\quantfam$ of quantizers $\mc{X} \to \mc{Z}$,
  and a sequence $\discfunc_n : \mc{Z} \to \R^k$ such that
  $\risk_{\logitloss}(\discfunc_n \mid \quant_n)
  \to \inf_{\discfunc, \quant \in \quantfam}
  \risk_{\logitloss}(\discfunc \mid \quant)$, but
  $\risk_{\zoloss}(\discfunc_n \mid \quant_n)
  \not\to \inf_{\discfunc, \quant \in \quantfam}
  \risk_{\zoloss}(\discfunc \mid \quant)$.
\end{example}


\section{Proof of the Theorems~\ref{theorem:loss-equivalent-entropy}
  and~\ref{theorem:loss-equivalent}}
\label{sec:proof}

The remainder of the main body of the 
paper consists of the major parts of our arguments for
Theorems~\ref{theorem:loss-equivalent-entropy} and~\ref{theorem:loss-equivalent}.
We divide the proof of the theorems into two parts. The ``if'' part is
straightforward; the ``only if'' is substantially more complex.

\paragraph{Proof (if direction)}
We give the proof for Theorem~\ref{theorem:loss-equivalent};
that for Theorem~\ref{theorem:loss-equivalent-entropy} is identical.
Assume that $\domain{f^{(1)}_{\prior}} = \domain{f^{(2)}_{\prior}}$ and there
exist $a > 0, b \in \R^{k-1},$ and $c \in \R$ such that
Eq.~\eqref{eqn:rel-f-loss-equivalent} holds.  By
Definition~\ref{def:general-fdiv} of multi-way $f$-divergences, for
any quantizer $\quant$, we have
\begin{equation*}
  \genfdiv{f_\prior^{(1)}}{P_1, \ldots, P_{k-1}}{P_k \mid \quant}
  = a \genfdiv{f_\prior^{(2)}}{P_1, \ldots, P_{k-1}}{P_k \mid \quant}
  + b^T \ones + c,
\end{equation*}
as $\int_{\mc{X}} dP_i = 1$. Applying the
relationship~\eqref{eqn:quantized-gap}, we obtain
\begin{align*}
  \quantinfo{\entropy_{\loss^{(1)}}}{\quant}
  & = a \quantinfo{\entropy_{\loss^{(2)}}}{\quant} + b^T \ones + c.
\end{align*}
As $a>0$, the universal
equivalence of $\loss^{(1)}$ and $\loss^{(2)}$ follows immediately.
\medskip

We turn to the ``only if'' part of the proofs of
Theorems~\ref{theorem:loss-equivalent-entropy} and~\ref{theorem:loss-equivalent}.
A roadmap is as follows: we first define what
we call \emph{order equivalence} of convex functions, which is related to
the equivalence of $f$-divergences and generalized entropies
(Def.~\ref{def:order-equivalent}). Then, for any two loss functions
$\loss^{(1)}$ and $\loss^{(2)}$ that are universally equivalent, we show
that the associated entropies $\entropy_{\loss^{(1)}}$ and
$\entropy_{\loss^{(2)}}$, as constructed in the infimal
representation~\eqref{eqn:inf-representation-full}, and the functions
$f^{(1)}$ and $f^{(2)}$ generating the $f$-divergences via
expression~\eqref{eqn:f-from-loss}, are order equivalent
(Lemmas~\ref{lemma:order-univ-entropy} and~\ref{lemma:order-univ}).  After this,
we provide a characterization of order equivalent closed convex functions
(Lemma~\ref{lemma:equivalence-subset}), which is the linchpin of our
analysis. The lemma shows that for any two order equivalent closed convex
functions $f_1$ and $f_2$ with $\dom f_1 = \dom f_2$, there
are parameters $a > 0, b \in \R^k$, and $c \in \R$ such that $f^{(1)}(t)
= a f^{(2)}(t) + b^Tt + c$ for all $t \in \dom f_1 = \dom f_2$.  This proves
the ``only if'' part of the Theorems~\ref{theorem:loss-equivalent-entropy}
and~\ref{theorem:loss-equivalent}, yielding the desired result.
We present the main parts of the proof in the body of the paper, deferring
technical nuances to the supplement.

\subsection{Universal equivalence and order equivalence}

By Definition~\ref{def:equivalent-losses}  (and its equivalent
variant stated~\eqref{eqn:loss-universal-equivalence}), universally
equivalent losses $\loss^{(1)}$ and $\loss^{(2)}$ induce the same ordering
of quantized information measures and $f$-divergences.  The next definition
captures this ordering slightly differently.

\begin{definition}
  \label{def:order-equivalent}
  Let $f_1 : \Omega \to \extendedR$ and $f_2 : \Omega \to \extendedR$ be
  closed convex functions, where $\Omega \subset \R^k$ is
  closed convex. Let $m \in \N$ be arbitrary and the matrices $A,
  B \in \R^{k \times m}$ satisfy $A \ones = B \ones$, where $A$ has
  columns $a_i \in \Omega$ and $B$ has columns $b_i \in \Omega$.
  Then $f_1$ and $f_2$ are \emph{order-equivalent} if for
  all $m \in \N$ and all such matrices $A$ and $B$ we have
  \begin{equation}
    \label{eqn:equivalent-order-equivalence}
    \sum_{j=1}^m f_1(a_j) \le \sum_{j=1}^m f_1(b_j)
    ~~ \mbox{if~and~only~if} ~~
    \sum_{j=1}^m f_2(a_j) \le \sum_{j=1}^m f_2(b_j).
  \end{equation}
\end{definition}

As the above context suggests, order equivalence has strong connections
with universal equivalence of loss functions $\loss$ and associated
$f$-divergences and generalized entropies.
The next two lemmas make this explicit.
\begin{lemma}
  \label{lemma:order-univ-entropy}
  If losses $\loss^{(1)}$ and $\loss^{(2)}$ are lower bounded and
  universally equivalent, then the associated entropies of the
  construction~\eqref{eqn:inf-representation-full} are order equivalent over
  $\simplex_k \subset \R^k_+$.
\end{lemma}
\begin{proof}
  Let $\entropy_i$ be the entropy
  (pointwise Bayes risk) associated with $\loss^{(i)}$, noting that
  $\dom \entropy_1 = \dom \entropy_2 = \simplex_k$ because $\inf_{\prior \in \simplex_k}
  \entropy_i(\prior) > -\infty$. Let the matrices $A = [a_1 ~ \cdots ~ a_m]
  \in \R^{k \times m}_+$ and $B \in \R^{k \times m}_+$ satisfy $a_i, b_i \in
  \simplex_k$ for each $i = 1, \ldots, m$, and let $v = \frac{1}{m} A \ones
  = \frac{1}{m} B \ones \in \simplex_k$. We show that $\sum_{j=1}^m
  \entropy_1(a_j) \le \sum_{j=1}^m \entropy_1(b_j)$ if and only if $\sum_{j=1}^m \entropy_2(a_j)
  \le \sum_{j=1}^m \entropy_2(b_j)$, that is,
  expression~\eqref{eqn:equivalent-order-equivalence} holds, by constructing
  appropriate distributions $P_{1:k}$ and $\prior$, then applying
  the universal equivalence of $\loss^{(1)}$ and $\loss^{(2)}$.

  Let $M_0$ be any integer large enough that $v_0 = \frac{1}{k}(1 +
  \frac{1}{M_0}) \ones - \frac{1}{M_0} v \in \R^k_+$, so that $v_0 \in
  \simplex_k$. Then define the vectors
  $\wt{a}_1 = v_0, \ldots, \wt{a}_{m M_0} = v_0$, and let
  \begin{align*}
    A\extend & = [a_1 ~ \cdots ~ a_m
      ~ \wt{a}_1 ~ \cdots ~ \wt{a}_{m M_0}]
    \in \R^{k \times M}_+
    ~~ \mbox{and} ~~ \\
    B\extend & = [b_1 ~ \cdots ~ b_m
      ~ \wt{a}_1 ~ \cdots ~ \wt{a}_{m M_0}]
    \in \R^{k \times M}_+,
  \end{align*}
  where $M = (M_0 + 1)m$. These satisfy $A\extend \ones = B\extend \ones =
  \frac{M}{k} \ones$. We let $a\extend$ and $b\extend$ denote the columns
  of these extended matrices.

  Now, let the spaces $\mc{X} = [M] \times [M]$ and $\mc{Z} = [M]$. Define
  quantizers $\quant_1, \quant_2 : \mc{X} \to \mc{Z}$ by $\quant_1(i, j)
  = i$ and $\quant_2(i,j) = j$. For $l = 1, \ldots, k$, define
  the distributions $P_l$ on $\mc{X}$ by
  \begin{equation*}
    P_l(i,j) = \frac{k^2}{M^2} \cdot a\extend_{il} b\extend_{jl},
    ~~ \mbox{so} ~~
    \sum_{j=1}^M P_l(i,j) = \frac{k}{M} a\extend_{il}
    \frac{k}{M} \sum_{j=1}^M b\extend_{jl} = \frac{k}{M} a\extend_{il}
  \end{equation*}
  and similarly $\sum_i P_l(i,j) = \frac{k}{M} b\extend_{jl}$.  Let $\prior
  = \frac{1}{k} \ones$ be the uniform prior distribution on the label $Y \in
  \{1, \ldots, k\}$, and note that the posterior probability
  \begin{equation*}
    \wt{\prior}(\quant_1^{-1}(\{i\}))
    = \left[\frac{\prior_l \sum_j P_l(i, j)}{
        \sum_{l'} \prior_{l'} \sum_j P_l(i, j)}\right]_{l=1}^k
    = \left[\frac{a\extend_{il}}{\sum_{l'} a\extend_{il'}}\right]_{l=1}^k
    = a\extend_i \in \simplex_k,
  \end{equation*}
  because $P_l(\quant_1^{-1}(i)) = \sum_j P_l(i,j) = \frac{k}{M}
  a_{il}\extend$, and similarly $\wt{\prior}(\quant_2^{-1}(\{j\})) =
  b\extend_j \in \simplex_k$. Taking the expectation over $X
  \sim \sum_{l=1}^k \prior_l P_l$, we have
  \begin{equation*}
    \E[\entropy_\loss(\wt{\prior}(\quant_1^{-1}(\quant_1(X))))]
    = \frac{1}{k} \sum_{i,l}
    P_l(\quant_1^{-1}(i)) \entropy_\loss(\wt{\prior}(\quant_1^{-1}(i)))
    = \frac{1}{M} \sum_{i=1}^M \entropy_\loss(a_i\extend),
  \end{equation*}
  because $\sum_l a\extend_{il} = 1$. Similarly,
  $\E[\entropy_\loss(\wt{\prior}(\quant_2^{-1}(\quant_2(X))))] = \frac{1}{M}
  \sum_{j=1}^M \entropy_\prior(b_j\extend)$.  Recalling the
  definitions~\eqref{eqn:information} and~\eqref{eqn:quantized-gap} of the
  (quantized) information associated with $\entropy$, we have
  $\quantinfo{\entropy}{\quant_1} = \entropy(\prior) - \frac{1}{M} \sum_{i = 1}^M
  \entropy(a_i\extend)$ and $\quantinfo{\entropy}{\quant_2} = \entropy(\prior) - \frac{1}{M}
  \sum_{i = 1}^M \entropy(b_i\extend)$.  Then the universal equivalence of losses
  $\loss^{(1)}$ and $\loss^{(2)}$ immediately implies
  for $\prior = \frac{1}{k} \ones$ that
  \begin{align*}
    \entropy_1(\prior)
    - \frac{1}{M} \sum_{i=1}^M  \entropy_1(a_i\extend)
    & \le
    \entropy_1(\prior) - \frac{1}{M} \sum_{i=1}^M
    \entropy_1(b_i\extend) ~~ \mbox{iff} \\
    \entropy_2(\prior)
    - \frac{1}{M} \sum_{i=1}^M  \entropy_2(a_i\extend)
    & \le
    \entropy_2(\prior) - \frac{1}{M} \sum_{i=1}^M
    \entropy_2(b_i\extend).
  \end{align*}
  Noting that $a_i\extend = b_i\extend$ for each $i \ge m+1$, we rearrange
  the preceding equivalent statements by adding $\frac{1}{M} \sum_{i \ge m +
    1} \entropy(a_i\extend)$ to each side to obtain that the $\entropy_i$ satisfy
  inequality~\eqref{eqn:equivalent-order-equivalence}.
\end{proof}

For $f$-divergences, a parallel result is possible; as the techniques are
similar to those we use to prove Lemma~\ref{lemma:order-univ-entropy}
(by constructing an explicit discrete space $\mc{X}$ and quantizers
$\quant$), we defer the proof to Supplementary
\S~\ref{sec:proof-order-univ}.
\begin{lemma}
  \label{lemma:order-univ}
  If losses $\loss^{(1)}$ and $\loss^{(2)}$ are universally equivalent for
  the prior $\prior$ (Def.~\ref{def:equivalent-losses}) and lower-bounded,
  the corresponding $f$-divergences $f_{\loss^{(1)}, \prior}$ and
  $f_{\loss^{(2)}, \prior}$ of construction~\eqref{eqn:f-from-loss} are
  order equivalent.
\end{lemma}



\subsection{Characterization of the order equivalence of convex functions}
\label{sec:char}

Lemmas~\ref{lemma:order-univ-entropy} and~\ref{lemma:order-univ} illustrate the
intrinsic relationship between the universal equivalence
(Def.~\ref{def:equivalent-losses}) of losses and the order equivalence
(Def.~\ref{def:order-equivalent}) of their associated generalized entropies
and $f$-divergences. Therefore, it is natural to ask when
convex functions are order equivalent. The lemma
below characterizes this order equivalence, and coupled
with Lemmas~\ref{lemma:order-univ-entropy} and~\ref{lemma:order-univ},
it immediately implies Theorems~\ref{theorem:loss-equivalent-entropy}
and~\ref{theorem:loss-equivalent}.

\begin{lemma}
  \label{lemma:equivalence-subset}
  Let $f_1, f_2 : \Omega \to \R$ be closed convex functions, where $\Omega
  \subset \R^k$ is a convex set. Then $f_1$ and $f_2$ are order equivalent
  on $\Omega$ if and only if there exist $a > 0$, $b \in \R^k$, and $c \in
  \R$ such that for all $t \in \Omega$
  \begin{equation}
    \label{eqn:relation-f-order-equivalent}
    f_1(t) = a f_2(t) + b^Tt + c.
  \end{equation}
\end{lemma}

While the proof of Lemma~\ref{lemma:equivalence-subset} is complex,
we provide a partial proof highlighting the most important parts of the
argument, deferring technical details to the supplement. The essential idea
is that Lemma~\ref{lemma:equivalence-subset} holds for simplices
(and so it certainly holds for $\entropy_\loss$); we can
then cover any convex set $\Omega$ with a number of overlapping simplices
to extend the result to all of $\Omega$, which
we do fully in Supplement~\ref{sec:proof-lemma-eq-sub}.
To demonstrate Lemma~\ref{lemma:equivalence-subset} for simplices,
we require
\begin{definition}
  \label{def:affine-independent}
  Vectors $u_0, u_1, \ldots u_m$ are \emph{affinely
    independent} if
  \begin{equation*}
    u_1 - u_0, ~ u_2 - u_0, ~ \ldots, ~ u_m - u_0,
  \end{equation*}
  are linearly independent. A set $E \subset \R^k$ is a \emph{simplex} if $E
  = \conv\{u_0, u_1, \ldots, u_k\}$ where $u_0, \ldots, u_k$ are affinely
  independent.
\end{definition}
\noindent
Then the essential special case of Lemma~\ref{lemma:equivalence-subset}
is the following result.
\begin{lemma}
  \label{lemma:equivalence-on-simplex}
  Let $E = \conv \{u_0, \ldots, u_k\} \subset \Omega$ where $u_0, \ldots, u_k$
  are affinely independent. If $f_1$ and $f_2$ are order equivalent,
  then there exist $a > 0$, $b \in \R^k$, and $c \in \R$ such that
  \begin{equation*}
    f_1(t) = a f_2(t) + b^Tt + c
    ~~ \mbox{for~all~} t \in E.
  \end{equation*}
\end{lemma}

The proof of Lemma~\ref{lemma:equivalence-on-simplex} proceeds in a series
of intermediate results, which we provide in turn,
deferring proofs to
Supplement~\ref{sec:proof-auxiliary-simplex-lemmas}.
Our first step is to argue that we need only prove equivalence results
for convex functions on dense subsets of their domains.
\begin{lemma}[\cite{HiriartUrrutyLe93ab}, Prop.~IV.1.2.5]
  \label{lemma:equal-functions}
  Let $f_1, f_2 : \Omega \to \R$ be closed convex
  and satisfy $f_1(t)
  = f_2(t)$ for $t$ in a dense subset of $\Omega$. Then
  $f_1 = f_2$ on $\Omega$.
\end{lemma}

The first technical lemma we prove is essentially a direct consequence of
the definition of order equivalence.
\begin{lemma}
  \label{lemma:order-equivalence-rationals}
  Let $u_1, \ldots, u_m \in \Omega$, $\alpha \in \Q^m$ satisfy
  $\ones^T\alpha = 1$, and $v \in
  \Omega$ with $v = \sum_{i=1}^m \alpha_i u_i$. If $f_1, f_2 : \Omega \to \R$
  are order equivalent, then
  \begin{equation*}
    \sum_{i=1}^m \alpha_i f_1(u_i) \le f_1(v)
    ~~ \mbox{if and only if} ~~
    \sum_{i=1}^m \alpha_i f_2(u_i) \le f_2(v).
  \end{equation*}
\end{lemma}
\noindent
Thus
if $\alpha \in
\Q^n$ satisfies $\ones^T\alpha = 1$ and $u_1, \ldots, u_n \in \Omega$,
then
\begin{equation}
  \label{eqn:order-equivalent-equality}
  f_1\left(\sum_{i=1}^n \alpha_i u_i \right)
  = \sum_{i=1}^n \alpha_i f_1(u_i)
  ~~ \mbox{iff} ~~
  f_2\left(\sum_{i=1}^n \alpha_i u_i\right)
  = \sum_{i=1}^n \alpha_i f_2(u_i).
\end{equation}

The next lemma shows that we can force
equality~\eqref{eqn:relation-f-order-equivalent} to hold for the $k + 1$
extreme points and centroid of any simplex in $\R^k$; it is intuitive
because there are $k + 2$ free parameters in the choices
of $a > 0$, $b \in \R^k$, and $c \in \R$.
\begin{lemma}
  \label{lemma:make-same-on-basis}
  Let $f_1, f_2 : \Omega \to \R$
  be closed convex and let $u_0,\ldots, u_k \in \Omega$ be
  affinely independent. There exist $a > 0, b \in \R^k$, and $c$ such
  that $f_1(u) = a f_2(u) + b^T u + c$ for $u \in \{u_0, \ldots, u_k,
  \ucent\}$, where $\ucent = \frac{1}{k+1} \sum_{i=0}^k u_i$.
\end{lemma}
\noindent
Lastly, we have the following
characterization of the linearity of convex functions over convex
hulls.
\begin{lemma}
  \label{lemma:center-makes-linear}
  Let $f : \Omega \to \R$ be convex with $u_1, \ldots, u_m \in \Omega$ and
  $\ucent = \frac{1}{m} \sum_{i=1}^m u_i$. If
  $f(\ucent) = \frac{1}{m} \sum_{i=1}^m f(u_i)$, then
  \begin{equation*}
    f\left(\sum_{i=1}^m \lambda_i u_i\right) = \sum_{i=1}^m \lambda_i f(u_i)
    ~~ \mbox{for~all~} \lambda \in \R^m_+
    ~ \mbox{with}~ \ones^T \lambda = 1.
  \end{equation*}
\end{lemma}

With the four
lemmas~\ref{lemma:equal-functions}--\ref{lemma:center-makes-linear}, we can
now prove Lemma~\ref{lemma:equivalence-on-simplex}.  By rotating with $u_i -
u_0$ and shifting by $u_0$, it is no loss of generality to assume that the
functions $f_i$ are defined on $V = \{v \in \R^k_+ \mid \ones^T v \le 1\}$,
so that $f_1$ and $f_2$ are continuous, defined, convex, and order
equivalent on $V$. We make one further reduction.
%
%
Let $e_i \in \R^k$ for $1 \leq i \leq k$ be the standard basis for $\R^k$
and $e_0 = \zeros$ be shorthand for the all-zeros vector. Further, let
$\ecenter = \frac{1}{k+1} \sum_{i=0}^k e_i$ be the centroid of $V$ (so $V
= \conv\{e_0, \ldots, e_k\}$).  Lemma~\ref{lemma:make-same-on-basis}
guarantees the existence of $a > 0, b\in \R^k$, and $c \in \R$ such that
\begin{equation*}
  f_1(v) = a f_2(v) + b^T v + c
  ~~ \mbox{for} ~ v \in \{e_0, e_1, \ldots, e_k, \ecenter\}.
\end{equation*}
Now, let $h_1(v) = f_1(v)$ and $h_2(v) = a f_2(v) + b^T v + c$, so $h_1$
and $h_2$ are convex, order equivalent on $V$, and satisfy $h_1(v) =
h_2(v)$ for $v \in \{e_0, \ldots, e_k ,\ecenter\}$.
Thus, Lemma~\ref{lemma:equivalence-on-simplex} is equivalent to showing
that if $h_1, h_2$ are convex, order equivalent,
and equal on the extreme points and centroid of $V$,
then
\begin{equation}
  h_1(v) = h_2(v) ~~ \mbox{for~} v \in V = \{v \in \R^k_+ \mid
  \ones^Tv \le 1\}.
  \label{eqn:h-equal-goody}
\end{equation}
We divide our discussion into two cases.

\paragraph{Linear case} Suppose that $h_1(\ecenter) =
\frac{1}{k+1}\sum_{i=0}^k h_1(e_i)$.
Then by order equivalence of $h_1$ and
$h_2$ (Eq.~\eqref{eqn:order-equivalent-equality}) we have
$h_2(\ecenter) = \frac{1}{k+1} \sum_{i=0}^{k} h_2(e_i)$.
Lemma~\ref{lemma:center-makes-linear} thus implies that $h_1$ and $h_2$ are
linear on $V = \conv\{e_0, \ldots, e_k\}$, equal on
the vertices of $V$, and hence equal on its interior.

\paragraph{Nonlinear case} By convexity we have
$h_1(\ecenter) < \frac{1}{k+1} \sum_{i=0}^k h_1(e_i)$, and order
equivalence (Lemma~\ref{lemma:order-equivalence-rationals}) implies
$h_2(\ecenter) < \frac{1}{k+1} \sum_{i=0}^k h_2(e_i)$.  For $v \in V =
\conv\{e_0, \ldots, e_k\}$, we use $v_0 = 1 - \ones^T v$ for shorthand, so
we may write $v = \sum_{i=0}^k v_i e_i$ and have
$[v_0 ~ v_1 ~ \cdots ~ v_k]^T \in \simplex_{k+1}$.  Now, fix an
arbitrary $v \in V \cap \Q^k$. We wish to show that $h_1(v) =
h_2(v)$.  To that end we consider consider the gaps due convexity of
$h_j(\ecenter)$ to the values of $h_j(e_i)$ \emph{relative} to those from
$h_j(v)$ to $h_j(e_i)$, defining the linear functions
$\varphi_j : [0, 1] \to \R$ by
\begin{equation*}
  \varphi_j(r)
  \defeq
  (1 - r)\left[h_j(\ecenter)
    - \frac{1}{k + 1} \sum_{i = 0}^k h_j(e_i)\right]
  + r \left[\sum_{i = 0}^k v_i h_j(e_i)
    - h_j(v) \right]
\end{equation*}
for $j = 1, 2$.
Then
\begin{equation*}
  \varphi_j(0) = h_j(\ecenter) - \frac{1}{k+1} \sum_{i=0}^k h_j(e_i) < 0
\end{equation*}
by assumption, and by convexity,
\begin{equation*}
  \varphi_j(1) = \sum_{i = 0}^k v_i h_j(e_i) - h_j(v) \ge 0.
\end{equation*}
The key is that the order equivalence of $h_1$ and $h_2$ on
$V$ implies that
\begin{equation}
  \label{eqn:equal-signs-super}
  \sign(\varphi_1(r)) = \sign(\varphi_2(r))
  ~~ \mbox{for~} r \in [0, 1],
\end{equation}
so that $\varphi_1$ and $\varphi_2$ have the same zero crossing
$r\opt > 0$, i.e.\
there exists $0 < r\opt \le 1$ with $\varphi_1(r\opt) = \varphi_2(r\opt)
= 0$. (We prove equality~\eqref{eqn:equal-signs-super} presently.)
At this $r\opt > 0$, we find
\begin{align*}
  0 & = \varphi_1(r\opt) - \varphi_2(r\opt)
  = -r\opt h_1(v) + r\opt h_2(v),
\end{align*}
where we use that $h_1(e_i) = h_2(e_i)$ for $i = 0, \ldots, k$
and $h_1(\ecenter) = h_2(\ecenter)$. 
That is, $h_1(v) = h_2(v)$, and as $v \in V \cap \Q^k$ is arbitrary and
$\Q^k$ is dense,
Lemma~\ref{lemma:equal-functions} extends the equality $h_1 = h_2$
to all of $V$.
Expression~\eqref{eqn:h-equal-goody} holds.

Returning to the sign equivalence~\eqref{eqn:equal-signs-super},
for $r > 0$, we may divide $\varphi_j(r)$ by $r$,
and we have
$\varphi_j(r) \le 0$ if and only if
\begin{equation*}
  \frac{1 - r}{r} \left[h_j(\ecenter)
    - \frac{1}{k+1}
    \sum_{i = 0}^k h_j(e_i)\right]
  + \sum_{i = 0}^k v_i h_j(e_i) \le h_j(v).
\end{equation*}
Defining $\alpha_i = v_i - \frac{1 - r}{r(k + 1)} \in \Q$ for $i = 0,
\ldots, k$ and $\alpha_{k+1} = \frac{1-r}{r}$, the inequality
$\varphi_j(r) \le 0$ is equivalent to $\sum_{i=0}^k \alpha_i h_j(e_i) +
\alpha_{k+1} h_j(\ecenter) \le h_j(v)$.  A calculation yields
$\ones^T\alpha = 1$ and $\sum_{i = 0}^k \alpha_i e_i + \alpha_{k + 1}
\ecenter = v$, and applying Lemma~\ref{lemma:order-equivalence-rationals}
immediately yields that $\varphi_1(r) \le 0$ if and only if $\varphi_2(r)
\le 0$ for all $r \in \openleft{0}{1} \cap \Q$.  Noting that $\varphi_1(0)
< 0$ and $\varphi_2(0) < 0$, we obtain
equality~\eqref{eqn:equal-signs-super}.

\section{Discussion}
\label{sec:conclusions}


Rather than recapitulating our contributions, we point out a few directions
we believe will prove interesting for further study.  While
Corollary~\ref{corollary:classification-calibration-restricted-families}
shows that some convex losses are surrogate-risk consistent even with
restricted families of classifiers, it does not apply to the practical case
in which the collection of discriminants $\discfunc$ is a (convex subset of
a) finite-dimensional vector space. This longstanding problem certainly
deserves further work.  Another direction, a bit further afield, is to
investigate the links between this work and objective Bayesian approaches
and reference priors~\cite{Bernardo05, Berger06}.  In this line of work, one
has a family $\{P_\theta\}_{\theta \in \Theta}$ of probability models on an
observation space $\mc{X}$ and before performing inference chooses a prior
$\prior$ on $\theta$ to maximize $I_\prior(X; \theta)$, the (Shannon)
information between $X \sim P_\theta$ and $\theta \sim \prior$. For
tasks \emph{other} than minimizing log
loss, it may be sensible to use a notion of information and entropy
corresponding to the desired loss. Our notions of loss equivalence,
including construction of convex losses equivalent to non-convex losses,
could provide insight in such situations.

\bibliographystyle{abbrvnat}
\bibliography{bib}

\ifdefined\separatesupplement
\else
\newpage

\appendix


\section{Proofs of classification calibration results}
\label{sec:calibration-proofs}

In this section, we prove Propositions~\ref{proposition:calibration}
and~\ref{proposition:zo-uncertainty-calibrated}. Before proving
the propositions proper, we state several technical lemmas
and enumerate continuity properties of Fenchel conjugates that
will prove useful. We also collect a few important definitions related to
convexity and norms here, which we use without comment in this appendix.
For a norm $\norm{\cdot}$ on $\R^k$, we recall the definition of the dual
norm $\dnorm{\cdot}$ as $\dnorm{y} = \sup_{\norm{u} \le 1} u^T y$.
For a convex function $f : \R^k \to \extendedR$, we let
\begin{equation*}
  \partial f(u) = \{g \in \R^k \mid f(v)
  \ge f(u) + g^T(v - u) ~ \mbox{for~all~} v \in \R^k\}
\end{equation*}
denote the subgradient set of $f$ at the point $u$. This set is non-empty if
$u \in \relint \dom f$ (see~\cite[Chapter VI]{HiriartUrrutyLe93ab}).

\subsection{Technical preliminaries}

We provide some background on convex functions. We
recall Definition~\ref{def:uniform-convex} of uniform convexity,
that $f$ is $(\lambda, \kappa, \norm{\cdot})$-uniformly convex over
$C \subset \R^k$ if it is closed and for all $t \in [0, 1]$
and $u_0, u_1 \in C$ we have
\begin{equation}
  \begin{split}
    \lefteqn{f(t u_0 + (1 - t) u_1)} \\
    & \le
    t f(u_0) + (1 - t) f(u_1)
    - \frac{\lambda}{2} t(1 - t) \norm{u_0 - u_1}^\kappa
    \left[(1 - t)^{\kappa - 1} + t^{\kappa - 1}\right].
  \end{split}
    \label{eqn:uniform-convexity}
\end{equation}
We state a related definition of smoothness.
\begin{definition}
  \label{def:smooth-functions}
  A function $f$ is \emph{$(L, \beta, \norm{\cdot})$-smooth}
  if it has $\beta$-H\"older continuous gradient
  with respect to the norm $\norm{\cdot}$, meaning that
  \begin{equation*}
    \dnorm{\nabla f(u_0) - \nabla f(u_1)} \le L \norm{u_0 - u_1}^\beta
    ~~ \mbox{for~} u_0, u_1 \in \dom f.
  \end{equation*}
\end{definition}

Our first technical lemma is an equivalence result for uniform
convexity.
\begin{lemma}
  \label{lemma:equivalent-uniform-convexity}
  Let $f : \Omega \to \R$, where $f$ is closed convex and $\Omega$ is a closed
  convex set. Then $f$ is $(\lambda, \kappa, \norm{\cdot})$-uniformly convex
  over $\Omega$ if and only if for $u_0 \in \relint \Omega$ and all $u_1$,
  \begin{equation}
    \label{eqn:subgrad-uniform-convexity}
    f(u_1) \ge f(u_0) + s_0^T(u_1 - u_0)
    + \frac{\lambda}{2} \norm{u_0 - u_1}^\kappa
    ~ \mbox{for~} s_0 \in \partial f(u_0).
  \end{equation}
  If inequality~\eqref{eqn:uniform-convexity} holds, then
  inequality~\eqref{eqn:subgrad-uniform-convexity} also holds for any points
  $u_0 \in \Omega$ and $s_0$ such that $\partial f(u_0) \neq \varnothing$
  and $s_0 \in \partial f(u_0)$.
\end{lemma}
\noindent
See Section~\ref{sec:proof-equivalent-uniform-convexity} for a proof of this
lemma.  There is also a natural duality between uniform convexity and
smoothness of a function's Fenchel conjugate $f^*(v) = \sup_u \{v^T u -
f(u)\}$; such dualities are common~\cite[Ch.~X.4]{HiriartUrrutyLe93ab}.
\begin{lemma}
  \label{lemma:uniform-to-holder}
  Let $\Omega \subset \R^k$ be a closed convex set and $f : \Omega \to \R$ be
  $(\lambda, \kappa, \norm{\cdot})$-uniformly convex over $\Omega$. Then $f^*$ is
  $(\lambda^{-\frac{1}{\kappa-1}}, \frac{1}{\kappa-1},
  \dnorm{\cdot})$-smooth (Def.~\ref{def:smooth-functions}) over $\dom f^* =
  \R^k$.
\end{lemma}
\noindent
See Section~\ref{sec:proof-uniform-to-holder} for a proof of
Lemma~\ref{lemma:uniform-to-holder}.  We also have two results on the
properties of smooth functions, whose proofs we provide in
Sections~\ref{sec:proof-holder-above}
and~\ref{sec:proof-get-zero-gradients}, respectively.
\begin{lemma}
  \label{lemma:holder-above}
  Let $f$ be $(L, \beta, \norm{\cdot})$ smooth. Then
  \begin{equation*}
    f(u_1) \le f(u_0) + \nabla f(u_0)^T(u_1 - u_0) + \frac{L}{\beta + 1}
    \norm{u_0 - u_1}^{\beta + 1}.
  \end{equation*}
\end{lemma}
\begin{lemma}
  \label{lemma:get-zero-gradients}
  Let $f$ be $(L, \beta, \norm{\cdot})$ smooth over $\R^k$ and
  $\inf_x f(x) > -\infty$. If the sequence $u_n$ satisfies
  $\lim_n f(u_n) = \inf_x f(x)$, then
  $\nabla f(u_n) \to 0$.
\end{lemma}

\subsection{Proof of Proposition~\ref{proposition:calibration}}

We state two intermediate lemmas before proving
Proposition~\ref{proposition:calibration}.
\begin{lemma}
  \label{lemma:symmetric-differentiability}
  If $\entropy$ is symmetric, closed and strictly concave, then $(-\entropy)^*$ is continuously
  differentiable on $\R^k$. If
  $\alpha_i \ge \alpha_j$, then
  $p = \nabla (-\entropy)^*(\alpha)$ satisfies $p_i \ge p_j$.
\end{lemma}
\begin{proof}
  As $\entropy$ is strictly concave and $(-\entropy)^*(\alpha) = \sup_{\prior \in
    \simplex_k} \{\prior^T \alpha + \entropy(\prior)\} < \infty$ for all $\alpha$
  (suprema of closed concave functions over compact sets are attained),
  so $(-\entropy)^*$ is continuously differentiable by standard results in
  convex analysis~\cite[Thm~X.4.1.1]{HiriartUrrutyLe93ab}, and
  $\nabla (-\entropy)^*(\alpha) = \argmax_{p \in \simplex_k} \{p^T\alpha +
  \entropy(p)\}$.

  Now let $\alpha$ satisfy $\alpha_i \ge \alpha_j$. As $\nabla
  (-\entropy)^*(\alpha) = \argmax_{p \in \simplex_k} \{p^T\alpha + \entropy(p)\}$, let us
  assume for the sake of contradiction that $p_i < p_j$. Then letting $A$ be
  the permutation matrix swapping entries $i$ and $j$, the vector $p' = A p$
  satisfies $\entropy(p') = \entropy(p)$ but $\entropy(\half(p' + p)) > \half \entropy(p') + \half \entropy(p)
  = \entropy(p) = \entropy(p')$, and
  \begin{equation*}
    \alpha^T p - \alpha^T p' =
    \alpha_i p_i - \alpha_i p_j + \alpha_j p_j - \alpha_j p_i
    = (\alpha_i - \alpha_j) (p_i - p_j) \le 0.
  \end{equation*}
  Thus we have $-\alpha^T p \ge -\alpha^T p'$, and so
  \begin{equation*}
    \half \alpha^T (p + p') + \entropy\left(\half p + \half p'\right)
    \ge -\alpha^T p + \entropy\left(\half p + \half p'\right)
    > \alpha^Tp + \entropy(p),
  \end{equation*}
  a contradiction to the assumed optimality of $p$. We must have
  $p_i \ge p_j$ whenever $\alpha_i \ge \alpha_j$.
\end{proof}

\begin{lemma}
  \label{lemma:attained-infimum}
  If $\entropy$ is symmetric and $\sum_{i=1}^k \prior_i \loss(\alpha\opt,i)
  = \inf_\alpha \sum_{i=1}^k \prior_i \loss(\alpha, i)$,
  then $\prior_i > \prior_j$ implies that
  $\alpha\opt_i \ge \alpha\opt_j$.
  If $\entropy$ is strictly concave, then $\alpha\opt_i > \alpha\opt_j$.
\end{lemma}
\begin{proof}
  Let $\prior$ satisfy $\prior_i > \prior_j$ as assumed in the lemma, and
  suppose that $\alpha\opt_i < \alpha\opt_j$ for the sake of
  contradiction. Let $A$ be the permutation matrix that swaps $\alpha\opt_i$
  and $\alpha\opt_j$. Then $(-\entropy)^*(A \alpha\opt) = (-\entropy)^*(\alpha\opt)$, and
  $(-\entropy)^*$ is also symmetric, and
  \begin{equation*}
    \sum_{i=1}^k \prior_i \loss(A \alpha\opt, i)
    - \sum_{i=1}^k \prior_i \loss(\alpha\opt, i)
    = -\prior^T A \alpha\opt + \prior^T \alpha\opt
    = (\prior_i - \prior_j) (\alpha\opt_i - \alpha\opt_j) < 0,
  \end{equation*}
  a contradiction to the optimality of $\alpha\opt$.
  If $\entropy$ is strictly concave,
  Lemma~\ref{lemma:symmetric-differentiability} implies that for $\alpha\opt$
  minimizing $\sum_{i=1}^k \prior_i \loss(\alpha, i)$, we have $\prior =
  \nabla (-\entropy)^*(\alpha\opt)$. Moreover, by
  Lemma~\ref{lemma:symmetric-differentiability},
  if $\prior_i > \prior_j$ we must have $\alpha\opt_i > \alpha\opt_j$.
\end{proof}

\begin{proof-of-proposition}[\ref{proposition:calibration}]
  If the infimum in $\sum_{i=1}^k \prior_i \loss(\alpha, i)$
  is attained, then Lemma~\ref{lemma:attained-infimum} gives the result.
  Otherwise, recall that $\entropy(\prior) = \inf_\alpha \{\sum_{i=1}^k \prior_i
  \loss(\alpha, i)\} > -\infty$, and let $\prior_i > \prior_j$. Let
  $\alpha^{(m)}$ be any sequence such that $\sum_{i=1}^k \prior_i
  \loss(\alpha^{(m)}, i) \to \entropy(\prior)$.
  
  Using that $\entropy$ is uniformly concave,
  we have $\nabla (-\entropy)^*$ is H\"older continuous over $\R^k$ (recall
  Lemma~\ref{lemma:uniform-to-holder}).  This implies that that $\prior -
  \nabla (-\entropy)^*(\alpha^{(m)}) \to 0$ as $m \to \infty$ (recall
  Lemma~\ref{lemma:get-zero-gradients}), or
  \begin{equation*}
    \lim_{m \to \infty} \nabla (-\entropy)^*(\alpha^{(m)}) = \prior.
  \end{equation*}
  Now, by
  Lemma~\ref{lemma:symmetric-differentiability}, for any $p^{(m)} = \nabla
  (-\entropy)^*(\alpha^{(m)})$, we have that $\alpha_i^{(m)} \ge \alpha_j^{(m)}$
  implies $p^{(m)}_i \ge p^{(m)}_j$. Thus, if $\prior_i > \prior_j$, it must
  be the case that eventually we have $\alpha^{(m)}_i >
  \alpha^{(m)}_j$. Moreover, if $\liminf |\alpha^{(m)}_i - \alpha^{(m)}_j| =
  0$, then we must have $\liminf |p^{(m)}_i - p^{(m)}_j| = 0$, which would
  contradict that $\prior_i > \prior_j$, as we have $p^{(m)}_i - p^{(m)}_j
  \to \prior_i - \prior_j$. We thus find that
  $\liminf_m (\alpha_i^{(m)} - \alpha_j^{(m)}) > 0$ if the sequence
  tends to the infimum $\entropy(\prior)$, which implies that
  \begin{equation*}
    \inf_\alpha \left\{\sum_{i=1}^k \prior_i \loss(\alpha, i)
    : \alpha_i \le \alpha_j \right\}
    > \entropy(\prior).
  \end{equation*}
  The loss~\eqref{eqn:loss-def} is thus classification calibrated
  (Def.~\ref{definition:classification-calibrated}).
\end{proof-of-proposition}

\subsection{Proof of Proposition~\ref{proposition:zo-uncertainty-calibrated}}

Without loss of generality, we assume that $\prior_k < \max_j \prior_j$, so
that restricting to $\alpha_k \ge \max_j \alpha_j$ forces $\alpha$ to have
larger zero-one risk than $1 - \max_j \prior_j$.  We present two lemmas
based on convex duality that imply the result.  In each,
we let $v_{\setminus i} = [v_1 ~ \cdots ~ v_{i-1} ~ v_{i+1} ~ \cdots
  v_k]^T \in \R^{k-1}$ be the vector $v \in \R^k$ without its $i$th element.
\begin{lemma}
  \label{lemma:weird-linear-inf}
  Let $v \in \R^k$. Then
  \begin{align*}
    \inf_{\alpha_k \ge \max_j \alpha_j}
    v^T \alpha
    = \begin{cases} 0 & \mbox{if~} v_k = -\sum_{i=1}^{k-1} v_i,
      v_{\setminus k} \preceq 0 \\
      -\infty & \mbox{otherwise.}
    \end{cases}
  \end{align*}
\end{lemma}
\begin{proof}
  By introducing Lagrange multipliers $\beta \in \R^{k-1}_+$ for the
  constraints $\alpha_k \ge \alpha_j$ for $j \neq k$, we have Lagrangian
  \begin{equation*}
    \mc{L}(\alpha, \beta)
    = \left(v + \left[\begin{matrix} \beta \\ -\beta^T \ones
      \end{matrix}\right]\right)^T \alpha,
    ~~~ \mbox{so} ~~~
    \inf_\alpha \mc{L}(\alpha, \beta)
    = \begin{cases}
      0 & \mbox{if}~ v_k = \beta^T \ones, v_{\setminus k} = -\beta \\
      -\infty & \mbox{otherwise.}
    \end{cases}
  \end{equation*}
  Substuting $\beta = -v_{\setminus k}$ and noting that
  $\beta \succeq 0$ gives the result.
\end{proof}

Thus, if we define the matrix $C = \ones \ones^T - I_{k \times k}$ with
columns $c_l$, we find by strong duality that
\begin{align*}
  \lefteqn{\inf_{\alpha_k \ge \max_j \alpha_j}
    \left\{-\prior^T \alpha
    + (-\entropy_{\zoloss})^*(\alpha) \right\}} \\
  & =
  \inf_{\alpha_k \ge \max_j \alpha_j}
  \left\{-\prior^T \alpha
  + \sup_{p \in \simplex_k} \inf_{q \in \simplex_k}
  \left\{p^T\alpha + q^T C^T p\right\}\right\} \\
  & = \sup_{p \in \Delta_k}
  \inf_{\alpha_k \ge \max_j \alpha_j}
  \left\{\min_l c_l^T p
  + (p - \prior)^T \alpha \right\} \\
  & = \sup_{p \in \simplex_k}
  \left\{1 - \max_l p_l
  \mid 
  \prior_{\setminus k} \succeq p_{\setminus k},
  p_k = \prior_k + \ones^T(\prior_{\setminus k} - p_{\setminus k})
  \right\},
\end{align*}
by Lemma~\ref{lemma:weird-linear-inf}.
The next lemma then immediately implies
Proposition~\ref{proposition:zo-uncertainty-calibrated}
once we note that $\entropy_{\zoloss}(\prior) = 1 - \max_j \prior_j$.
\begin{lemma}
  Let $G_k(\prior) = \inf_{\alpha_k \ge \max_j \alpha_j}
  \{-\prior^T \alpha + (-\entropy_{\zoloss})^*(\alpha)\}$.
  Then
  \begin{equation*}
    G_k(\prior) - (1 - \max_j \prior_j)
    \ge \frac{1}{k} (\max_j \prior_j - \prior_k).
  \end{equation*}
\end{lemma}
\begin{proof}
  \newcommand{\tlow}{t_{\rm low}}
  \newcommand{\thigh}{t_{\rm high}}
  We essentially construct the optimal $p \in \simplex_k$ vector
  for the supremum in the definition of $G_k$. Without
  loss of generality, we assume that $\prior_1 = \max_j \prior_j > \prior_k$,
  as the result is trivial if $\prior_k = \max_j \prior_k$.
  For $t \in \R$, define
  \begin{equation*}
    L(t) = \prior_k + \sum_{j=1}^k \hinge{\prior_j - t}
    ~~ \mbox{and} ~~
    R(t) = t.
  \end{equation*}
  By definining $\tlow = \prior_k$ and
  $\thigh = \prior_1$ we have
  \begin{align*}
    L(\tlow) & = \prior_k + \sum_{j=1}^k \hinge{\prior_j - \prior_k}
    \ge \prior_k + \hinge{\prior_1 - \prior_k} = \prior_1 > \prior_k
    = R(\tlow), \\
    L(\thigh) & = \prior_k < \prior_1 = R(\thigh),
  \end{align*}
  and the fact that $L$ is strictly decreasing in
  $[\tlow, \thigh]$ and $R$ is strictly increasing implies
  that there exists a unique root $t^\star \in (\tlow, \thigh)$ such that
  $L(t^\star) = t^\star = R(t^\star)$.
  Now, we define the vector $p$ by
  \begin{equation*}
    p_j = \min\{t^\star, \prior_j\} ~~ \mbox{for~} j \le k-1,
    ~~
    p_k = t^\star.
  \end{equation*}
  Then we have $\ones^T p = t^\star + \sum_{j=1}^{k-1} t^\star \wedge \prior_j
  = \prior_k + \sum_{j=1}^{k-1} (\hinge{\prior_j - t^\star} + t^\star
  \wedge \prior_j) = 1$, and moreover, we have
  $G_k(\prior) \ge 1 - \max_l p_l = 1 - t^\star$.

  It remains to show that $1 - t^\star - (1 - \prior_1) \ge
  \frac{1}{k} (\prior_1 - \prior_k)$; equivalently, we must show that
  $t^\star \le (1 - \frac{1}{k}) \prior_1
  + \frac{1}{k} \prior_k$. Suppose for the sake of
  contradiction that this does not hold, that is,
  that $t^\star > (1 - \frac{1}{k}) \prior_1 + \frac{1}{k} \prior_k$.
  We know that $\tlow = \prior_k < t^\star < \prior_1 = \thigh$ by assumption.
  With $t^\star$ satisfying these two constraints, we have that
  \begin{align*}
    L(t^\star)
    = \prior_k + \sum_{j=1}^k \hinge{\prior_j - t^\star}
    & \le \prior_k + \sum_{j=1}^{k-1} \hinge{\prior_1 - t^\star}
    = \prior_k + (k - 1) (\prior_1 - t^\star) \\
    & < \prior_k + (k-1)\left(\prior_1 - \frac{k-1}{k}
    \prior_1 - \frac{1}{k} \prior_k\right) \\
    & = \frac{k-1}{k} \prior_1 + \frac{1}{k} \prior_k
    < t^\star,
  \end{align*}
  the two strict inequalities by assumption on $t^\star$. But then
  we would have $L(t^\star) < R(t^\star) = t^\star$, a contradiction
  to our choice of $t^\star$. Because $L$ is decreasing, it is thus
  the case that $t^\star \le (1 - \frac{1}{k}) \prior_1 + \frac{1}{k}
  \prior_k$, giving the result.
\end{proof}

\subsection{Proof of Observation~\ref{observation:weighted-calibration}}
\label{sec:proof-weighted-calibration}

Our proof is essentially a trivial modification of Zhang~\cite[Theorem
  8]{Zhang04a}.  We assume without loss of generality that $\phi(\cdot) \ge
0$. Let $\prior \in \simplex_k$, and recalling that the cost matrix $C =
[c_1 ~ \cdots ~ c_k]$, let
\begin{equation*}
  L(\prior, \alpha) = \sum_{y = 1}^k \prior_y \sum_{i = 1}^k c_{yi}
  \phi(-\alpha_i)
  = \sum_{i = 1}^k \prior^T c_i \, \phi(-\alpha_i),
\end{equation*}
noting that $\prior^T c_i \ge 0$ for each $i$. Without loss of generality,
we may assume that $\prior^T c_1 > \prior^T c_2 = \min_l \prior^T c_l$.  If
we can show that $\inf_\alpha L(\prior, \alpha) < \inf_{\alpha_1 \ge \max_j
  \alpha_j} L(\prior, \alpha)$, then the proof will be complete. Let
$\alpha^{(m)} \in \R^k$ be any sequence satisfying $\ones^T \alpha^{(m)} =
0$ and $\alpha_1^{(m)} \ge \max_j \alpha_j^{(m)}$ such that $L(\prior,
\alpha^{(m)}) \to \inf_{\alpha_1 \ge \max_j \alpha_j} L(\prior, \alpha)$.

We first show that it is no loss of generality to assume that $\alpha^{(m)}$
converges. Suppose for the sake of contradiction that $\limsup_m
\norm{\alpha^{(m)}} = \infty$. Then as $\ones^T \alpha^{(m)} = 0$, we must
have $\limsup_m \alpha_1^{(m)} = \infty$, so that
$\limsup_m \phi(-\alpha_1^{(m)}) = \infty$ because $\phi'(0) < 0$ and $\phi$
is convex. As it must be the case that $\prior^T c_1 \phi(-\alpha_1^{(m)})$
remains bounded (for the convergence of $L(\prior, \alpha^{(m)})$), we would
then have that $\prior^T c_1 = 0$, which is a contradiction because
$\prior^T c_1 > \min_l \prior^T c_l \ge 0$. Thus we must have
$\limsup_m \norm{\alpha^{(m)}} < \infty$, and so there is a subsequence
of $\alpha^{(m)}$ converging; without loss of generality, we assume
that $\alpha^{(m)} \to \alpha\opt$. Then $L(\prior, \alpha\opt) = 
\inf_{\alpha_1 \geq \max{\alpha_j}} L(\prior, \alpha)$ by continuity of $\phi$. 

We show that by swapping the value of $\alpha_1\opt$ with the value
$\alpha_2\opt$ (or increasing the latter slightly), we can always improve
the value $L(\prior, \alpha\opt)$. We consider three cases, noting
in each that $\alpha_1\opt \ge 0$ as $\ones^T \alpha\opt = 0$.
\begin{enumerate}[1.]
\item Let $\alpha_1\opt = \alpha_2\opt \ge 0$. Then $\phi'(-\alpha_1\opt) =
  \phi'(-\alpha_2\opt) \le \phi'(0) < 0$, and additionally $c_1^T \prior
  \phi'(-\alpha_1\opt) - c_2^T \prior \phi'(-\alpha_2\opt) < 0$ because
  $c_2^T \prior < c_1^T \prior$. For sufficently small $\delta > 0$, we
  thus have
  \begin{equation*}
    c_1^T \prior \phi(-\alpha_1\opt + \delta)
    + c_2^T \prior \phi(-\alpha_2\opt - \delta)
    < c_1^T \prior \phi(-\alpha_1\opt) + c_2^T \prior \phi(-\alpha_2\opt),
  \end{equation*}
  whence $L(\prior, \alpha\opt) > L(\prior, \alpha\opt - \delta(e_1 -
  e_2))$.
\item Let $\alpha_1\opt > \alpha_2\opt$ and
  $\phi(-\alpha_1 \opt) > \phi(-\alpha_2\opt)$. Then
  taking $\alpha \in \R^k$ such that $\alpha_1 = \alpha_2\opt$,
  $\alpha_2 = \alpha_1\opt$, and $\alpha_i = \alpha_i\opt$ for $i \ge 3$,
  it is clear that
  $L(\prior, \alpha) < L(\prior, \alpha\opt)$.
\item Let $\alpha_1\opt > \alpha_2\opt$ and $\phi(-\alpha_1\opt) \le
  \phi(-\alpha_2\opt)$. Using that $-\alpha_1\opt < 0$, we see that for
  sufficiently small $\delta > 0$ we have $\phi(-\alpha_1\opt + \delta) <
  \phi(-\alpha_1\opt)$ and $\phi(-\alpha_2\opt) \ge \phi(-\alpha_2\opt -
  \delta)$, because $\phi$ must be non-decreasing at the point
  $-\alpha_2\opt$ as $\phi'(0) < 0$. Thus we have
  $L(\prior, \alpha\opt - \delta(e_1 - e_2)) < L(\prior, \alpha\opt)$.
\end{enumerate}

\subsection{Biconjugates of the zero-one loss}
\label{sec:proof-conjugate-zo-loss}

In this subsection, we calculate the conjugate of the generalized entropy $\entropy(\prior)
= 1 - \max_j \prior_j$ (pointwise Bayes risk) associated with the zero-one
loss, demonstrating equality~\eqref{eqn:conjugate-zo-loss}.  Let
$(-\entropy)^*(\alpha) = \sup_{\prior \in \simplex_k} \{\prior^T\alpha + 1-
\linf{\prior}\}$.  Formulating the Lagrangian for the supremum with dual
variables $\theta \in \R$, $\lambda \in \R^k_+$, we have
\begin{equation*}
  \mc{L}(\prior, \theta, \lambda)
  = \prior^T\alpha - \linf{\prior} - \theta(\ones^T \prior - 1)
  + \lambda^T \prior,
\end{equation*}
which has dual objective
\begin{equation*}
  \sup_\prior \mc{L}(\prior, \theta, \lambda)
  = \begin{cases}
    \theta & \mbox{if}~ \lone{\alpha + \lambda - \theta \ones} \le 1 \\
    -\infty & \mbox{otherwise}.
  \end{cases}
\end{equation*}
As $\inf_{\lambda \ge 0} \lone{\alpha + \lambda - \theta \ones} =
\sum_{j=1}^k \hinge{\alpha_j - \theta}$ and strong duality obtains (the
problems are linear), we have that the supremum in
expression~\eqref{eqn:conjugate-zo-loss} is
\begin{equation*}
  \sup_\prior \left\{\prior^T \alpha - \linf{\prior}
  \mid \prior^T \ones = 1, \prior \ge 0 \right\}
  = \inf \bigg\{\theta \mid \sum_{j=1}^k \hinge{\alpha_j - \theta} \le 1
  \bigg\}.
\end{equation*}
Without loss of generality we may assume that $\alpha_1 \ge \alpha_2
\ge \cdots$ by symmetry, so that over the domain
$\theta \in \openleft{-\infty}{\alpha_1}$, the function
$\theta \mapsto \sum_{j=1}^k \hinge{\alpha_j - \theta}$ is strictly
decreasing. Thus, there is a unique smallest $\theta$ satisfying
$\sum_{j=1}^k \hinge{\alpha_j - \theta} \le 1$ (attaining the equality),
and by inspection, this must be one of
\begin{equation*}
  \theta \in \left\{\alpha_1 - 1,
  \frac{\alpha_1 + \alpha_2}{2} - \half,
  \frac{\alpha_1 + \alpha_2 + \alpha_3}{3} - \frac{1}{3},
  \ldots, \frac{\ones^T\alpha}{k} - \frac{1}{k}
  \right\}.
\end{equation*}
(Any $\theta$ makes some number of the terms $\hinge{\alpha_j - \theta}$
positive; fixing the number of terms and solving for $\theta$ gives the
preceding equality.) Expression~\eqref{eqn:conjugate-zo-loss} follows.

\subsection{Pointwise infimal risks of hinge losses}
\label{sec:proof-conjugate-hinge-loss}

We demonstrate equality~\eqref{eqn:conjugate-hinge-loss} with
$C = [c_1 ~ \cdots ~ c_k]$. We note that
\begin{equation*}
  \entropy(\prior) = \inf_{\alpha^{T}\ones = 0}
  \left\{\sum_{y=1}^{k} \prior_y \sum_{i = 1}^k c_{yi}
  \hinge{1+\alpha_i} \right\}
  = \inf_{\alpha^{T}\ones = 0}
  \left\{\sum_{i=1}^{k} \prior^T c_i
  \hinge{1+\alpha_i} \right\},
\end{equation*}
and formulating the Lagrangian by introducing
dual variable $\theta \in \R$, we have
\begin{equation*}
  \mc{L}(\alpha,\theta) = \sum_{i=1}^{k} \prior^T c_i
  \hinge{1+\alpha_{i}}
  -\theta \ones^T\alpha.
\end{equation*}
The generalized KKT conditions~\cite{Bertsekas99} for this problem 
are given by taking subgradients of the Lagrangian. 
At optimum, we must have
\begin{equation*}
  \nu_i \in \partial \hinge{1 + \alpha_i}
  = \begin{cases}
    0 & \mbox{if}~ \alpha_i < -1 \\
    [0, 1] & \mbox{if}~ \alpha_i = -1 \\
    1 & \mbox{if~} \alpha_i > -1,
  \end{cases}
  ~~ \mbox{and} ~~
  \prior^T c_i \, \nu_i - \theta = 0, ~
  \ones^T\alpha = 0,
  ~ i \in [k].
\end{equation*}
Without loss of generality, assume that
$\prior^T c_1 = \min_j \prior^T c_j$. Set $\alpha\opt \in \R^k$
and $\theta\opt$ via
\begin{align*}
  \alpha\opt_1 = (k - 1),
  ~~ \alpha\opt_i = -1
  ~ \mbox{for~} i \ge 2,
  ~~~ \theta\opt = \min_j \prior^T c_j = \prior^T c_1.
\end{align*}
We have $\ones^T \alpha\opt = (k - 1) - (k - 1) = 0$, and setting $\nu_i =
1$ for $i$ such that $\prior^T c_i = \min_j \prior^T c_j$ and $\nu_i =
\frac{\min_j \prior^T c_j}{ \prior^T c_i} \in [0, 1]$ otherwise, we see
that $\prior^T c_i \nu_i - \theta\opt = \min_j \prior^T c_j - \min_j
\prior^T c_j = 0$ for all $i$; the KKT conditions are satisfied.  Thus
$\alpha\opt$ and $\theta\opt$ are primal-dual optimal, yielding
expression~\eqref{eqn:conjugate-hinge-loss}.

\subsection{Classification calibration of hinge losses}
\label{sec:hinge-loss-calibration}

We provide a quantitative
guarantee on the classification calibration of hinge-like losses.
\begin{lemma}
  \label{lemma:hinge-great-gap}
  Let $C \in \R^{k \times k}_+$ and let $\wzoloss(\alpha, y) = \max_i
  \{c_{yi} \mid \alpha_i = \max_j \alpha_j\}$ be the cost-weighted
  loss~\eqref{eqn:weighted-zero-one}. Let
  $\loss(\alpha, y) = \sum_{i = 1}^k c_{yi} \hinge{1 + \alpha_i}
  + \bindic{\ones^T \alpha = 0}$ or
  $\loss(\alpha, y) = \sum_{i = 1}^k c_{yi} \hinge{1 + \alpha_i - \alpha_y}$.
  Then for any $\prior \in \simplex_k$ and
  $\alpha \in \R^k$,
  \begin{equation*}
    \sum_{i = 1}^k \prior_i
    \loss(\alpha, i)
    - \inf_{\alpha'}
    \sum_{i = 1}^k \prior_i \loss(\alpha',i)
    \ge \sum_{i = 1}^k \prior_i
    \wzoloss(\alpha, i)
    - \inf_{\alpha'}
    \sum_{i = 1}^k \prior_i \wzoloss(\alpha',i).
  \end{equation*}
\end{lemma}
\begin{proof}
  First, recall from Example~\ref{example:hinge-loss}
  that $\inf_\alpha \sum_{i = 1}^k \prior_i \loss(\alpha, i) = k \min_l
  \prior^T c_l$.
  Defining $y(\alpha) = \argmax_j \alpha_j$ (breaking ties in some
  arbitrary deterministic order), it is sufficient to
  argue that
  for $\loss(\alpha, y) = \sum_{i=1}^k c_{yi} \hinge{1 +
    \alpha_i}$, we have
  \begin{equation}
    \sum_y \prior_y \loss(\alpha, y)
    - k \min_l \prior^T c_l
    \ge \left[c_{y(\alpha)}^T \prior - \min_l \prior^T c_l
      \right]
    \label{eqn:cw-loss-pointwise-gap}
  \end{equation}
  for any vector $\prior \in \simplex_k$
  and $\alpha$ with $\ones^T \alpha = 0$.

  To show inequality~\eqref{eqn:cw-loss-pointwise-gap}, assume without loss
  of generality that there exists an index $l\opt < k$ such that $c_1^T
  \prior = c_2^T \prior = \cdots = c_{l\opt}^T \prior = \min_l c_l^T
  \prior$, while $c_l^T \prior > c_1^T \prior$ for $l > l\opt$. (If $l\opt =
  k$, then inequality~\eqref{eqn:cw-loss-pointwise-gap} is trivial.) We
  always have $\sum_y \prior_y \loss(\alpha, y) \ge k\min_l c_l^T \prior$;
  let us suppose that $\alpha_l \ge \max_j \alpha_j$ for some $l > l\opt$;
  without loss of generality take $l = k$. Then we have
  \begin{align}
    \nonumber
    \sum_{y=1}^k \prior_y \loss(\alpha, y)
    & \ge \inf_{\alpha_k \ge \max_j \alpha_j,
      \ones^T \alpha = 0}
    \sum_{y=1}^k \prior_y \sum_{j=1}^k c_{yj} \hinge{1 + \alpha_j} \\
    & = \inf_{\alpha_k \ge \max_j \alpha_j, \ones^T \alpha = 0}
    \sum_{l=1}^k c_l^T \prior \hinge{1 + \alpha_l}.
    \label{eqn:weighted-problem-to-solve}
  \end{align}
  Writing the Lagrangian for this problem and introducing variables
  $\lambda \ge 0$ for the inequality $\alpha_k \ge \max_j \alpha_j$ and
  $\theta \in \R$ for the equality $\ones^T \alpha = 0$, we have
  \begin{equation*}
    \mc{L}(\alpha, \lambda, \theta)
    = \sum_{l=1}^k \prior^T c_l\hinge{1 + \alpha_l}
    + \theta \ones^T \alpha + \lambda \left(\max_j \alpha_j - \alpha_k
    \right).
  \end{equation*}
  Set $\alpha\opt_1 = \cdots = \alpha\opt_{l\opt} = \frac{k - l\opt -
    1}{l\opt + 1}$, $\alpha_k = \frac{k - l\opt - 1}{l\opt + 1}$, and
  $\alpha_l = -1$ for $l > l\opt$, $l \neq k$. We claim that these are
  optimal. Indeed, set $\theta\opt = -(1 - \epsilon) \prior^T c_1 - \epsilon
  \prior^T c_k$ for some $\epsilon \in (0, 1)$, whose value we specify
  later. Taking subgradients of the Lagrangian with respect to $\alpha$, we
  have
  \begin{equation*}
    \partial_\alpha \mc{L}(\alpha\opt, \lambda, \theta\opt)
    = \diag(\nu)
    \left[\begin{matrix} c_1^T \\ \vdots \\ c_k^T \end{matrix}\right]
    \prior
    + \theta\opt \ones + \lambda \left[\conv\{e_1, \ldots, e_{l\opt},
      e_k\} - e_k \right],
  \end{equation*}
  where $\nu_i \in [0, 1]$ are arbitrary scalars satisfying $\nu_i \in
  \partial \hinge{1 + \alpha_i\opt}$, i.e.\ $\nu_1 = \cdots = \nu_{l\opt} =
  \nu_k = 1$, $\nu_l \in [0, 1]$ for $l \in \{l\opt + 1, \ldots, k-1\}$.
  Notably, we have $\prior^T c_l + \theta\opt = \epsilon \prior^T (c_1 -
  c_k) \le 0$ for $l \in \{1, \ldots, l\opt\}$ and $\prior^T c_k +
  \theta\opt = (1 - \epsilon) \prior^T(c_k - c_1) \ge 0$.  Choosing
  $\epsilon$ small enough that $\epsilon \prior^T c_k + (1 - \epsilon)
  \prior^T c_1 < \prior^T c_l$ for $l \in \{l \opt + 1, \ldots, k-1\}$, it
  is clear that we may take $\nu_l = \frac{(1 - \epsilon) \prior^T c_1 +
    \epsilon \prior^T c_k}{\prior^T c_l} \in [0, 1]$ for each $l \not \in
  \{1, \ldots, l\opt, k\}$, and we show how to choose $\lambda\opt \ge 0$ so
  that $\zeros \in \partial_\alpha \mc{L}(\alpha\opt, \lambda,
  \theta\opt)$. Assume without loss of generality (by scaling of $\lambda
  \ge 0$) that $\prior^T(c_k - c_1) = 1$. Eliminating extraneous indices in
  $\partial_\alpha \mc{L}$, we see that we seek a setting of $\lambda$ and a
  vector $v \in \simplex_{l\opt +1}$ such that
  \begin{equation*}
    e_{l\opt + 1} - \epsilon \ones_{l\opt + 1}
    + \lambda (v - e_{l\opt + 1}) = 0.
  \end{equation*}
  This is straightforward: take $v = \frac{1}{l\opt \epsilon + (1 -
    \epsilon)}[\epsilon ~ \cdots ~ \epsilon ~ (1 - \epsilon)]^T \in
  \simplex_{l\opt + 1}$, and set $\lambda\opt = l\opt \epsilon + (1 -
  \epsilon) \ge 0$.

  Summarizing, we find that our preceding choices of $\alpha\opt$ were
  optimal for the problem~\eqref{eqn:weighted-problem-to-solve}, that is,
  $\alpha_1\opt = \cdots = \alpha_{l\opt} = \frac{k - l\opt - 1}{l\opt + 1}
  = \alpha_k$, with $\alpha_l = -1$ for $l\opt < l < k$. We thus have
  \begin{align*}
    \sum_{y = 1}^k \prior_y \loss(\alpha, y)
    & \ge \sum_{l = 1}^{l\opt} \prior^T c_l
    \left(1 + \frac{k - l\opt - 1}{l\opt + 1} \right)
    + \prior^T c_k \left(1 + \frac{k - l\opt - 1}{l\opt + 1}\right) \\
    & = \min_l \prior^T c_l
    \left(l\opt + \frac{l\opt}{l\opt + 1}(k - l\opt - 1)\right)
    + \frac{k}{l\opt + 1} c_k^T \prior.
  \end{align*}
  In particular, we have
  \begin{align*}
    \lefteqn{\sum_{y=1}^k \prior_y \loss(\alpha, y) - k \min_l \prior^T c_l}
    \\
    & \ge \min_l \prior^T c_l
    \left(l\opt + \frac{l\opt}{l\opt + 1}(k - l\opt - 1)\right)
    + \frac{k}{l\opt + 1} c_k^T \prior
    - k \min_l c_l^T \prior \\
    & = \frac{k}{l\opt + 1} \left[c_k^T \prior - \min_l c_l^T \prior
      \right].
  \end{align*}
  Recalling that we must have had $l\opt < k$, we obtain
  inequality~\eqref{eqn:cw-loss-pointwise-gap}.
\end{proof}

\subsection{Proofs of technical lemmas on smoothness}

\subsubsection{Proof of Lemma~\ref{lemma:equivalent-uniform-convexity}}
\label{sec:proof-equivalent-uniform-convexity}

Let $t \in (0, 1)$ and $u^t = t u_1 + (1 - t) u_0$. Assume that
inequality~\eqref{eqn:subgrad-uniform-convexity} holds and let $u_0 \in
\relint \Omega$.  Then
\begin{align*}
  f(u_1) & 
  \ge f(u^t) + s_t^T(u_1 - u^t) + \frac{\lambda}{2} \norm{u^t - u_1}^\kappa \\
  & = f(u^t) + (1 - t) s_t^T(u_1 - u_0) + \frac{\lambda}{2}
  (1 - t)^\kappa \norm{u_0 - u_1}^\kappa,
\end{align*}
where $s_t \in \partial f(u^t)$, which must exist as $u^t \in \relint \Omega$
(see Lemma~\ref{lemma:interior-line-segments} and~\cite[Chapter
  VI]{HiriartUrrutyLe93ab}). Similarly, we have
\begin{equation*}
  f(u_0) \ge f(u^t) + t s_t^T(u_0 - u_1)
  + \frac{\lambda}{2} t^\kappa \norm{u_0 - u_1}^\kappa.
\end{equation*}
Multiplying the preceding inequalities by $t$ and
$(1 - t)$, respectively, gives
\begin{align*}
  t f(u_1) + (1 - t) f(u_0)
  & \ge f(u^t)
  + \frac{\lambda}{2} \norm{u_0 - u_1}^\kappa \left[t(1 - t)^\kappa
    + (1 - t) t^\kappa \right] \\
  & = f(u^t) + \frac{\lambda}{2} t(1 - t) \norm{u_0 - u_1}^\kappa
  \left[(1 - t)^{\kappa - 1} + t^{\kappa - 1}\right]
\end{align*}
for any $u_0 \in \relint \Omega$.

We now consider the case that $u_0 \in \Omega
\setminus \relint \Omega$, as the preceding display is equivalent to the
uniform convexity condition~\eqref{eqn:uniform-convexity}.
Let $u_0 \in \Omega \setminus \relint \Omega$ and $u_0' \in \relint \Omega$. 
For $a \in [0, 1]$ define
the functions $h(a) = t f(u_1) + (1 - t) f((1 - a) u_0 + a u_0')$ and
\begin{align*}
  h_t(a) & =
  f(t u_1 + (1 - t)((1 - a) u_0 + a u_0')) \\
  & \quad ~ 
  + \frac{\lambda}{2} t(1 - t) \norm{(1 - a) u_0 + au_0' - u_1}^\kappa
  \left[(1 - t)^{\kappa - 1} + t^{\kappa - 1}\right].
\end{align*}
Then $h$ and $h_t$
are closed one-dimensional convex functions, which are thus
continuous~\cite[Chapter I]{HiriartUrrutyLe93ab}, and we have $h(a) \ge
h_t(a)$ for all $a \in (0, 1)$ as $(1-a)u_0 + au_0' \in \relint \Omega$.
Thus
\begin{align*}
  \lefteqn{t f(u_0) + (1 - t) f(u_1) = h(0)
    = \lim_{a \to 0} h(a)
    \ge \lim_{a \to 0} h_t(a)
    = h_t(0)} \\
  & \qquad = f(tu_1 + (1 - t) u_0)
  + \frac{\lambda}{2} t(1 - t) \norm{u_0 - u_1}^\kappa
  \left[(1 - t)^{\kappa-1} + t^{\kappa - 1}\right].
\end{align*}
This is equivalent to the uniform convexity
condition~\eqref{eqn:uniform-convexity}.

We now prove the converse. Assume the
uniform convexity condition~\eqref{eqn:uniform-convexity},
which is equivalent to
\begin{equation*}
  \frac{f(u^t) - f(u_0)}{t} + \frac{\lambda}{2} (1 - t)
  \norm{u_0 - u_1}^\kappa [(1 - t)^{\kappa - 1} + t^{\kappa - 1}]
  \le f(u_1) - f(u_0)
\end{equation*}
for all $u_0, u_1 \in \Omega$ and $t \in (0, 1)$. Let $f'(x, d) = \lim_{t
  \downarrow 0} \frac{f(x + t d) - f(x)}{t}$ be the directional derivative
of $f$ in direction $d$, recalling that if $\partial f(x) \neq \varnothing$
then $f'(x, d) = \sup_{g \in \partial f(x)} \<g, d\>$
(see~\cite[Ch.~VI.1]{HiriartUrrutyLe93ab}). Then taking $t \downarrow 0$, we
have
\begin{equation*}
  f'(u_0, u_1 - u_0) + \frac{\lambda}{2} \norm{u_0 - u_1}^\kappa
  \le f(u_1) - f(u_0).
\end{equation*}
Because
$f'(u_0, u_1 - u_0) = \sup_{g \in \partial f(u_0)} \<g, u_1 - u_0\>$,
this implies the subgradient condition~\eqref{eqn:subgrad-uniform-convexity}.

\subsubsection{Proof of Lemma~\ref{lemma:uniform-to-holder}}
\label{sec:proof-uniform-to-holder}

First, we note that as $\dom f = \Omega$ and $f$ is uniformly convex, it is
$1$-coercive, meaning that $\lim_{\norm{u} \to \infty} f(u) / \norm{u} =
\infty$.  Thus $\dom f^* = \R^k$; see~\cite[Proposition
  X.1.3.8]{HiriartUrrutyLe93ab}.

As $f$ is strictly convex by assumption, we have that $f^*$ is
differentiable~\cite[Theorem X.4.1.1]{HiriartUrrutyLe93ab}.  Moreover, as $f$
is closed convex, $f = f^{**}$, and we have for any $s \in \R^k$ that $u_0 =
\nabla f^*(s)$ if and only if $s \in \partial f(u_0)$, meaning that $f$ is
subdifferentiable on the set $\image \nabla f^* = \{\nabla f^*(s) : s \in
\R^k\}$, whence $\Omega \supset \image \nabla f^*$.  Now, let $s_0, s_1 \in \R^k$
and $u_0 = \nabla f^*(s_0)$ and $u_1 = \nabla f^*(s_1)$. We must then have $s_0
\in \partial f(u_0)$ and $s_1 \in \partial f(u_1)$ by standard results in convex
analysis~\cite[Corollary X.1.4.4]{HiriartUrrutyLe93ab}, so that $\partial
f(u_0) \neq \varnothing$ and $\partial f(u_1) \neq \varnothing$.

Now we use the uniform convexity
condition~\eqref{eqn:subgrad-uniform-convexity} of
Lemma~\ref{lemma:equivalent-uniform-convexity} to see
\begin{align*}
  f(u_1) - f(u_0)
  & \ge \<s_0, u_1 - u_0\> + \frac{\lambda}{2} \norm{u_0 - u_1}^\kappa
  ~~ \mbox{and} \\
  f(u_0) - f(u_1) &
  \ge \<s_1, u_0 - u_1\> + \frac{\lambda}{2} \norm{u_0 - u_1}^\kappa.
\end{align*}
Adding these equations, we find that
\begin{equation*}
  \<s_0 - s_1, u_0 - u_1\> \ge \lambda \norm{u_0 - u_1}^\kappa,
  ~~~ \mbox{so} ~~~
  \dnorm{s_0 - s_1} \norm{u_0 - u_1} \ge \lambda \norm{u_0 - u_1}^\kappa
\end{equation*}
by H\"older's inequality. Dividing each side by $\norm{u_0 - u_1}$, we
obtain $\norm{u_0 - u_1} \le \lambda^{-\frac{1}{\kappa-1}} \dnorm{s_0 -
  s_1}^\frac{1}{\kappa-1}$, the desired result
once we note that $\nabla f^*(s_i) = u_i$.

\subsubsection{Proof of Lemma~\ref{lemma:holder-above}}
\label{sec:proof-holder-above}

Define $u_t = u_0 + t (u_1 - u_0)$ for $t \in [0, 1]$. Then
using Taylor's theorem, we have
\begin{align*}
  \lefteqn{f(u_1)
    = f(u_0) + \int_0^1 \<\nabla f(t u_1 + (1 - t) u_0), u_1 - u_0\> dt} \\
  & = f(u_0) + \<\nabla f(u_0), u_1 - u_0\>
  + \int_0^1 \<\nabla f(u_t) - \nabla f(u_0), u_1 - u_0\> dt \\
  & \le f(u_0) + \<\nabla f(u_0), u_1 - u_0\>
  + \int_0^1 \dnorm{\nabla f(u_t) - \nabla f(u_0)}
  \norm{u_0 - u_1} dt \\
  & \le f(u_0) + \<\nabla f(u_0), u_1 - u_0\>
  + \int_0^1 L t^\beta \norm{u_0 - u_1}^{\beta + 1} dt.
\end{align*}
Computing the final integral as $\int_0^1 t^\beta dt = \frac{1}{\beta + 1}$
gives the result.

\subsubsection{Proof of Lemma~\ref{lemma:get-zero-gradients}}
\label{sec:proof-get-zero-gradients}

Let $g_n = \nabla f(u_n)$ and suppose for the sake of contradiction that
$g_n \not \to 0$. Then there is a subsequence, which without loss of
generality we take to be the full sequence, such that $\dnorm{g_n} \ge c >
0$ for all $n$. Fix $\delta > 0$, which we will choose later. By
Lemma~\ref{lemma:holder-above}, defining $y_n = u_n - \delta g_n /
\dnorm{g_n}$, we have
\begin{equation*}
  f(y_n) \le f(u_n) - \delta \norm{g_n} + \frac{L}{\beta + 1}
  L \delta^{\beta + 1}
  \le f(u_n) - \delta \left(c - \frac{L}{\beta+1} \delta^\beta\right).
\end{equation*}
In particular, we see that if
\begin{equation*}
  \delta \le \left(\frac{c}{2} \cdot \frac{\beta + 1}{L}\right)^{1/\beta},
  ~~ \mbox{then} ~~
  f(y_n) \le f(u_n) - \frac{\delta c}{2}
\end{equation*}
for all $n$, which contradicts the fact that $f(u_n) \to \inf_x f(x) >
-\infty$.


\section{Order Equivalence of Convex Functions}
\label{sec:order-equivalence-proofs}

In this appendix, we collect the proofs of our various technical results
on order equivalent functions (Definition~\ref{def:order-equivalent}).

\subsection{Proof of Lemma~\ref{lemma:order-univ}}
\label{sec:proof-order-univ}

Before proving the lemma, we give a matrix characterization of
non-negative vectors with equal sums similar to the characterization of
majorization via doubly stochastic matrices (cf.~\cite{MarshallOlAr11}; we
temporarily defer the proof of the lemma to
Appendix~\ref{sec:proof-grid}).
\begin{lemma}
  \label{lemma:grid}
  Vectors $a, b \in \R^m_+$ satisfy $\ones^T a = \ones^Tb$ if and only if
  there exists a matrix $Z \in \R^{m \times m}_+$ 
  such that $Z \ones = a$ and $Z^T \ones = b$.
\end{lemma}

Returning to the proof of Lemma~\ref{lemma:order-univ} proper,
let $f_i = f_{\loss^{(i)}, \prior}$ for shorthand. Note that $\dom f_1 =
\dom f_2 = \R^{k-1}_+$ by the construction~\eqref{eqn:f-from-loss},
because $\min_i \inf_{\alpha \in \R^k} \loss(\alpha,i) > -\infty$. Now,
given matrices $A, B \in \R^{(k-1) \times m}_+$ satisfying $A\ones = B
\ones$ (where $A = [a_1 ~ \cdots ~ a_m]$ and $B = [b_1 ~ \cdots ~ b_m]$),
we construct distributions $P_i$ and quantizers $\quant_1$ and $\quant_2$
such that for any $f$ we have
\begin{align*}
  \fdiv{P_1, \ldots, P_{k-1}}{P_k \mid \quant_1}
  & = C \sum_{j=1}^m f(a_j)
  ~~ \mbox{and} \\
  \fdiv{P_1, \ldots, P_{k-1}}{P_k \mid \quant_2}
  & = C \sum_{j=1}^m f(b_j),
\end{align*}
where $C > 0$ is a constant.  We then use
Definition~\ref{def:equivalent-losses} of loss equivalence to show that
$f_1$ and $f_2$ are order equivalent.

With that in mind, take $M$ to be any
positive integer such that $M > \max\{\linf{A \ones}, m\}$.  We enlarge $A$
and $B$ into matrices $A\extend, B\extend \in \R^{(k-1) \times M}_+$
respectively, adding $M - m$ columns. To construct these matrices, let
$a\extend_{i, m+1} = M - \sum_{j=1}^m a_{ij}$ and $b\extend_{i,m+1} = M
- \sum_{j=1}^m b_{ij}$, for $i = 1, \ldots, k-1$, and set $a\extend_{il}
= b\extend_{il} = 0$ for all $m+1 < l \le M$. The enlarged matrices
$A\extend$ and $B\extend$ thus satisfy $\frac{1}{M} A\extend \ones =
\frac{1}{M} B\extend \ones = \ones$, and
their columns belong to $\dom f_1 = \dom f_2 = \R^{k-1}_+$.

Let the spaces $\mc{X} = \{(i, j) \mid 1\leq i\leq M, 1\leq j \leq
M\}$ and $\mc{Z}= \{1, 2, \ldots, M\}$. Define quantizers $\quant_1,
\quant_2 : \mc{X} \to [M]$ by $\quant_1(i, j) = i$ and $\quant_2(i, j) =
j$.  As $A\extend \ones = B\extend \ones = M \ones$, Lemma
\ref{lemma:grid} guarantees the existence of matrices $Z^l = [z_{ij}^l]
\in \R^{M \times M}_+$ such that
\begin{equation*}
  Z^l \ones =
  \frac{1}{M} [a\extend_{lj}]_{j=1}^M
  \in \R^M_+
  ~~~ \mbox{and} ~~~
  (Z^l)^T \ones = \frac{1}{M}[b\extend_{lj}]_{j=1}^M \in \R^M_+,
\end{equation*}
which implies that for all $l \in [k-1]$, the matrix $Z^l$ satisfies
$\sum_{ij} z^l_{ij} = 1$.  For each $l \in [k-1]$, define the probability
distribution $P_l$ on $\mc{X}$ by $P_l((i, j)) = z_{ij}^{l}$ for $i, j \in
[M]$. Let $P_k$ be the distribution defined by $P_k((i, j)) = M^{-2}$.

Under this quantizer design and choice of distributions $P_l$, we have for
any prior $\prior \in \R^k_+$ (upon which the functions $f_1$ and $f_2$
implicitly depend) and $f = f_1$ or $f = f_2$ that
\begin{align*}
  \fdiv{P_1, \ldots, P_{k-1}}{P_k \mid \quant_1}
  & = \frac{1}{M} \sum_{j=1}^M f(a\extend_j)
  ~~ \mbox{and} \\
  \fdiv{P_1, \ldots, P_{k-1}}{P_k \mid \quant_2}
  & = \frac{1}{M} \sum_{j=1}^M f(b\extend_j),
\end{align*}
because $\sum_{i=1}^M P_l((i,j)) = \sum_{i=1}^M z_{ij}^l = a\extend_{lj} /
M$ and $\sum_{i=1}^M P_k((i,j)) = M^{-1}$, and similarly for $B\extend$.
By the loss equivalence of $\loss^{(1)}$ and $\loss^{(2)}$ (recall
Def.~\ref{def:equivalent-losses}), we obtain that
\begin{equation*}
  \sum_{j=1}^M f_1(a\extend_j) \le \sum_{j=1}^M f_1(b\extend_j)
  ~~ \mbox{if and only if} ~~
  \sum_{j=1}^M f_2(a\extend_j) \le \sum_{j=1}^M f_2(b\extend_j).
\end{equation*}
Note that $a\extend_j = b\extend_j \in \dom f_i$ for each $j > m$ and $i =
1, 2$.  Moreover, $a\extend_j=a_j$ and $b\extend_j=b_j$ for $1 \le j \le
m$, so the preceding display is equivalent to
the order equivalence~\eqref{eqn:equivalent-order-equivalence}.

\subsubsection{Proof of Lemma~\ref{lemma:grid}}
\label{sec:proof-grid}

One direction of the proof is easy: if $Z \in \R^{m \times m}_+$ and
$Z \ones = a$ while $Z^T \ones = b$, then $\ones^T Z \ones = \ones^T a
= b^T \ones$.

We prove the converse using induction. When $m = 1$, the result
is immediate.  We claim that it is no loss of generality to assume that $a_1
\ge a_2 \ge \cdots \ge a_m$ and $b_1 \ge \cdots \ge b_m$; indeed, let $P_a$
and $P_b$ be permutation matrices such that $P_a a$ and $P_b b$ are in
sorted (decreasing) order. Then if we construct $\wt{Z} \in \R^{m \times
  m}_+$ such that $\wt{Z} \ones = P_a a$ and $\wt{Z}^T \ones = P_b b$, we
have $Z = P_a^T \wt{Z} P_b$ satisfies $Z \in \R^{m \times m}_+$ and $Z \ones
= P_a^T \wt{Z} \ones = P_a^T P_a a = a$ and $Z^T \ones = P_b^T \wt{Z} \ones
= b$.

Now, suppose that the statement of the lemma is true for all vectors of
dimension up to $m - 1$; we argue that the result holds for $a, b \in
\R_+^m$. It is no loss of generality to assume that $a_m \leq b_m$.  Let
$z_{m,m} = a_m$, $z_{i,m} = 0$ for $2 \le i \le m-1$,
and set $z_{1, m} = b_m - a_m \ge
0$. Additionally, we set $z_{m, i} = 0$ for $1 \le i \le m -1$.
Now, let $Z_{\rm inner} \in \R^{m-1 \times m-1}_+$ be the matrix
defined by the upper $(m-1) \times (m-1)$ sub-matrix of $Z$, and
define the vectors
\begin{equation*}
  a^{\rm inner} = [a_1 + a_m - b_m ~~ a_2 ~~ a_3 ~ \cdots ~~ a_{m-1}]^T
  ~~ \mbox{and} ~~
  b^{\rm inner} = [b_1 ~ b_2 ~ \cdots ~ b_{m-1}]^T.
\end{equation*}
Then as $a_1 \ge (1/m) \ones^Ta = (1/m) \ones^T b \ge b_m$, we have
$a^{\rm inner} \ge 0$ and $b^{\rm inner} \ge 0$, and moreover, $\ones^T
a^{\rm inner} = \ones^T a - b_m = \ones^T b - b_m = \ones^T b^{\rm
  inner}$. In particular, we have by the inductive hypothesis that we may
choose $Z_{\rm inner}$ such that $Z_{\rm inner} \ones = a^{\rm inner}$ and
$Z_{\rm inner}^T \ones = b^{\rm inner}$. By inspection, setting
\begin{equation*}
  Z = \left[\begin{matrix} \left[ \begin{matrix} & & \\
          & Z_{\rm inner} & \\ & & \end{matrix}\right] &
      \begin{array}{c} b_m - a_m \\ 0 \\ \vdots \\ 0 \end{array} \\
      \begin{array}{ccc} 0 & \cdots & 0 \end{array}
      & a_m
    \end{matrix}
    \right],
\end{equation*}
we have $Z \in \R^{m \times m}_+$ and $Z \ones = a$ and $Z^T \ones = b$.

\subsection{Proof of Lemma~\ref{lemma:equivalence-subset}}
\label{sec:proof-lemma-eq-sub}

As in our discussion in Section~\ref{sec:char}, we prove the lemma in a
sequence of steps and auxiliary lemmas.  The roadmap is as follows: first,
we show that we may assume the set $\Omega$ in
Lemma~\ref{lemma:equivalence-subset} has non-empty interior
(\S~\ref{sec:has-interior}). With this done, we can consider any affinely
independent subset of $k + 1$ points from $\Omega$, whence
Lemma~\ref{lemma:equivalence-on-simplex} holds. With this done, we can cover
$\Omega$ with overlapping simplices (\S~\ref{sec:cover-simplices}),
and extend this to all of $\Omega$ (\S~\ref{sec:extend-to-everything}).

\subsubsection{It is no loss of generality to
  assume $\Omega$ has
  non-empty interior in Lemma~\ref{lemma:equivalence-subset}}
\label{sec:has-interior}

Let $H = \affine{\Omega}$ be the affine hull of $\Omega$, where
$\dim H = l \ge 1$. (If $\dim{H} = 0$, then $\Omega$ is a single point, and
Lemma~\ref{lemma:equivalence-subset} is trivial.) We argue that if
Lemma~\ref{lemma:equivalence-subset} holds for sets $\Omega$ such that
$\interior{\Omega} \neq \varnothing$ it holds generally; thus
we temporarily assume its truth
for convex $\Omega$ with $\interior{\Omega} \neq \varnothing$.

Since $\dim{H} = l$, we have $H = \left\{Av + d \mid v \in \R^l \right\}$
for some full column-rank matrix $A \in \R^{k \times l}$ and $d \in \R^k$.
As $\Omega$ has non-empty interior relative to $H$ (e.g.~\cite[Theorem
  2.1.3]{HiriartUrrutyLe93ab}), we have $\Omega = \{Av + d \mid v \in \Omega_0
\}$ for a convex set $\Omega_0 \subset \R^l$ with $\interior \Omega_0 \neq
\varnothing$. Defining $\wt{f}_1(v) = f_1(A v + d)$ and $\wt{f}_2(v) =
f_2(Av + d)$, where $\wt{f}_i(v) = \infty$ for $v \not \in \Omega_0$, we
have that $\wt{f}_1$ and $\wt{f}_2$ are order equivalent on $\Omega_0
\subset \R^l$. By assumption, Lemma~\ref{lemma:equivalence-subset} holds
for $\Omega_0$, so there exist $a>0, \wt{b}\in \R^l, \wt{c} \in \R$ such that
\begin{equation*}
 f_1(Av + d) = \wt{f}_1(v) = a \wt{f}_2(v) + \langle \wt{b}, v\rangle + c,
 ~~ \mbox{for~}
 v \in \Omega_0.
\end{equation*}

As $A$ is full column rank, for all $t \in \Omega$ there exists a
unique $v_t \in \Omega_0$ such that $t = A v_t + d$, and the mapping $t
\mapsto v_t$ is linear, i.e., $v_t = Et + g$ for some $E \in \R^{l \times
 k}$ and $g \in \R^l$ (we may take $E \in \R^{l \times k}$ to be any
left-inverse of $A$ and $g = -E d$, so $v = Et - Eg$). We obtain
that $f_i(t) = \wt{f}_i(Et + g)$ for $i = 1, 2$, whence
\begin{equation*}
  f_1(t) = \wt{f}_1(Et + g)
  = a \wt{f}_2(Et + g) + \langle\wt{b}, Et + g\rangle + \wt{c}
  = a f_2(t) + \langle E^T \wt{b}, t\rangle + \langle\wt{b}, g\rangle + \wt{c}
\end{equation*}
for all $t \in \Omega$,
which is our desired result.
From this point forward, we thus assume w.l.o.g.\ that $\interior{\Omega}
\neq \varnothing$.

\subsubsection{Covering sets with simplices}
\label{sec:cover-simplices}

With Lemma~\ref{lemma:equivalence-on-simplex}
in hand, we now show that the special case for simplices is sufficient to
show the general Lemma~\ref{lemma:equivalence-subset}.
First, we show that simplices essentially cover convex sets $\Omega$.
\begin{lemma}
  \label{lemma:points-in-simplices}
  Let $v, w$ be arbitrary points in $\interior \Omega \subset \R^k$. Then,
  there exist $u_0, u_1, \ldots, u_k \in \Omega$ such that $v, w \in
  \interior E$, where $E = \conv\{u_0, u_1, \ldots, u_k\}$. For
  any points $u_2, \ldots, u_k$ that make $v, w, u_2, \ldots, u_k$ affinely
  independent (Def.~\ref{def:affine-independent}), there exist $u_0, u_1 \in
  \Omega$ such that $v, w \in \interior \conv\{u_0, u_1, \ldots, u_k\}$.
\end{lemma}
\noindent
Before proving Lemma~\ref{lemma:points-in-simplices}, we state a technical
lemma about interior points of convex sets.
\begin{lemma}[Hiriart-Urruty and Lemar\'echal~\cite{HiriartUrrutyLe93ab},
    Lemma III.2.1.6]
  \label{lemma:interior-line-segments}
  Let $C \subset \R^k$ be a convex set, $u \in \relint C$ and $v \in \cl
  C$. Then for any $\lambda \in \openright{0}{1}$, we have $\lambda u + (1 -
  \lambda) v \in \relint C$.
\end{lemma}
\begin{proof-of-lemma}[\ref{lemma:points-in-simplices}]
  \begin{figure}[ht]
    \begin{center}
      \psfrag{x}{$v$}
      \psfrag{y}{$w$}
      \psfrag{x0}{$u_0$}
      \psfrag{x1}{$u_1$}
      \psfrag{x2}{$u_2$}
      \psfrag{Omega}{{\Large $\Omega$}}
      \includegraphics[width=.6\columnwidth,clip=true]{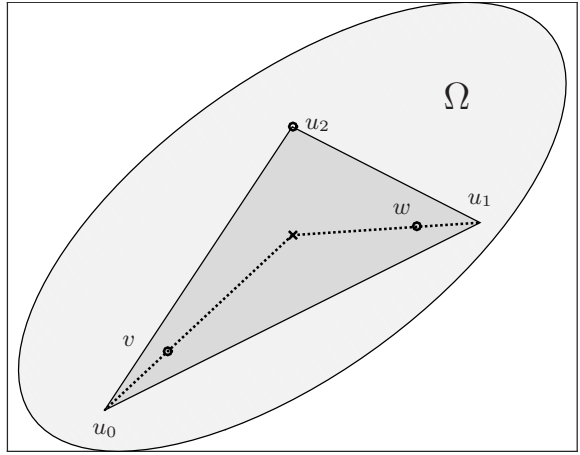}
      \caption{\label{fig:points-in-simplices} The construction in
        Lemma~\ref{lemma:points-in-simplices}.}
    \end{center}
  \end{figure}
  Take $u_2, \ldots, u_k \in \Omega$ arbitrarily but in general position, so
  that the points $v, w, u_2, \ldots, u_k$ and $w, v, u_2, \ldots, u_k$ are
  affinely independent (Def.~\ref{def:affine-independent}). Now, define
  $\ucent = \frac{1}{k+1} (v + w + \sum_{i=2}^k u_i)$, so that $\ucent \in
  \interior \Omega$ (Lemma~\ref{lemma:interior-line-segments}), and choose
  $\epsilon > 0$ small enough that the points $u_0 = v + \epsilon(v -
  \ucent)$ and $u_1 = w + \epsilon(w - \ucent)$ are both in $\Omega$. (This
  is possible as $v, w \in \interior \Omega$; see
  Figure~\ref{fig:points-in-simplices}.)  Then we find that $v = (u_0 +
  \epsilon \ucent) / (1 + \epsilon)$ and $w = (u_1 + \epsilon \ucent) / (1 +
  \epsilon)$, so that
  \begin{equation*}
    \ucent = \frac{1}{k + 1}
    \left(\frac{u_0}{1 + \epsilon} + \frac{u_1}{1 + \epsilon}
    + \frac{2 \epsilon}{1 + \epsilon} \ucent + \sum_{i=2}^k u_i
    \right),
  \end{equation*}
  or by rearranging, $\ucent = \lambda_0 (u_0 + u_1) + \lambda_1
  \sum_{i=2}^k u_i$, where
  \begin{align*}
    \lambda_0 & = \left(1 - \frac{2\epsilon}{(k + 1)(1 + \epsilon)}\right)^{-1}
    \frac{1}{(k + 1)(1 + \epsilon)}
    ~~ \mbox{and} ~~ \\
    \lambda_1 & = \left(1 - \frac{2\epsilon}{(k + 1)(1 + \epsilon)}\right)^{-1}
    \frac{1}{k + 1}.
  \end{align*}
  Noting that $2 \lambda_0 + (k-1) \lambda_1 = 1$ and $\lambda_i > 0$,
  we obtain that
  $\ucent \in \interior \conv\{u_0, \ldots, u_k\}$.
  As the points $u_i$ are in general position,
  and as
  \begin{equation*}
    v = \frac{1}{1 + \epsilon} u_0 + \frac{\epsilon}{1 + \epsilon} \ucent
    ~~ \mbox{and} ~~
    w = \frac{1}{1 + \epsilon} u_1 + \frac{\epsilon}{1 + \epsilon} \ucent,
  \end{equation*}
  we have $v, w \in \interior \conv\{u_0, \ldots, u_k\}$ by
  Lemma~\ref{lemma:interior-line-segments}.
\end{proof-of-lemma}

\subsubsection{Extension from a single simplex to all of $\Omega$}
\label{sec:extend-to-everything}

We use Lemma~\ref{lemma:points-in-simplices} to show the following lemma,
which implies Lemma~\ref{lemma:equivalence-subset}.
\begin{lemma}
  \label{lemma:sub-simplices}
  In addition to the conditions of Lemma~\ref{lemma:equivalence-subset},
  assume that for all simplices $E \subset
  \interior{\Omega}$ there exist $a_E > 0$, $b_E \in \R^k$, and $c_E \in
  \R$ such that $f_1(t) = a_E f_2(t) + b_E^T t + c_E$ for $t \in E$. Then
  there exist $a > 0, b\in \R^k$, and $c \in \R$ such that
  \begin{equation*}
    f_1(t) = af_2(t) + b^T t + c
    ~~ \mbox{for~} t \in \Omega.
  \end{equation*}
\end{lemma}
\noindent
Coupled with Lemma~\ref{lemma:equivalence-on-simplex},
Lemma~\ref{lemma:sub-simplices} immediately yields
Lemma~\ref{lemma:equivalence-subset}; indeed,
Lemma~\ref{lemma:equivalence-on-simplex} shows that for
any simplex $E$ the conditions of Lemma~\ref{lemma:sub-simplices}
holds, so that Lemma~\ref{lemma:equivalence-subset}
follows, i.e., $f_1(t) = a f_2(t) + b^Tt + c$ for all $t \in \Omega$.

\begin{proof}
  For $i \in \{1, 2\}$ define the sets
  \begin{align*}
    S_i & = \big\{\{(u, v) \in \interior \Omega \times \interior \Omega
    \mid \nabla f_i(u) \neq \nabla f_i(v)\} \\
    & ~~~~
    \cup \left\{(u, v) \in \interior \Omega \times \interior \Omega
    \mid  \nabla 
    f_i(u) \mbox{ or } \nabla f_i(v) \mbox{ does not exist}\right\}
    \big\} \subset \Omega \times \Omega.
  \end{align*}
  We divide our discussion into two cases.

  \paragraph{Case 1} First we suppose that $S_1 = S_2 = \varnothing$. 
  Then for $i=1, 2$, $f_i$ are differentiable in $\interior \Omega$ in this case, we
  have that $\nabla f_i(u) = \nabla f_i(v)$ for all $u, v \in \interior \Omega$.
  Then by continuity of the $f_i$ on
  $\interior \Omega$, we must have $f_i(t) = b_i^Tt + c_i$ for $i = 1, 2$.  The
  result follows by taking $a = 1$, $b = b_1 - b_2$ and $c = c_1 - c_2$
  and applying Lemma~\ref{lemma:equal-functions}.
  
  \paragraph{Case 2}
  We have that at least one of $S_1$ and $S_2$ is non-empty.
  Without loss of generality, say $S_1 \neq \varnothing$. 
  Choose a pair $(u\opt, v\opt) \in S_1$ and
  consider the collection of sets
  \begin{equation*}
    \mc{M} =
    \left\{E = \conv\{u_0, u_1, \ldots, u_k\} \mid
    u\opt, v\opt \in \interior  E ~ \mbox{and} ~ u_0, u_1, \ldots u_k \in \Omega
    \right\}.
  \end{equation*} 
  We show that for any $E_1, E_2 \in \mc{M}$, we have $a_{E_1} = a_{E_2}$,
  $b_{E_1} = b_{E_2}$ and $c_{E_1} = c_{E_2}$. We know that $\interior E_1
  \cap \interior E_2 \neq \varnothing$, as $u\opt \in \interior E$ for all $E
  \in \mc{M}$. Now, assume for the sake of contradiction that $a_{E_1} \neq
  a_{E_2}$. For all $t \in E_1 \cap E_2$ we have
  \begin{equation}
    \label{eqn:equal-f-on-simplices}
    f_1(t) = a_{E_1}f_2(t) + b_{E_1}^T t + c_{E_1}
    ~~ \mbox{and} ~~
    f_1(t) = a_{E_2}f_2(t) + b_{E_2}^T t + c_{E_2}.
  \end{equation}
  By subtracting the preceding equations from one another after multiplying
  by $a_{E_2}$ and $a_{E_1}$, respectively, one obtains
  that for
  $t \in E_1 \cap E_2$,
  \begin{equation*}
    f_1(t) = \frac{1}{a_{E_2} - a_{E_1}}\left[(a_{E_2} b_{E_1} - 
      a_{E_1} b_{E_2})^T t + (a_{E_2} c_{E_1} - a_{E_1} c_{E_2}) \right].
  \end{equation*}
  This yields a contradiction, since we have assumed that either (i) $\nabla
  f_1(u\opt)$ or $\nabla f_2(v\opt)$ does not exist or (ii) $\nabla
  f_1(u\opt) \neq \nabla f_1(v\opt)$, while $u\opt, v\opt \in \interior E_1
  \cap \interior E_2 =\interior (E_1 \cap E_2)$.  Thus $a_{E_1} =
  a_{E_2}$. To obtain $b_{E_1} = b_{E_2}$, note that $\interior (E_1 \cap
  E_2) \neq \varnothing$, as $u\opt, v\opt \in \interior (E_1 \cap E_2)$.
  Subtracting the equalities~\eqref{eqn:equal-f-on-simplices}, we find $0 =
  f_1(t) - f_1(t) = (b_{E_1} - b_{E_2})^T t + c_{E_1} - c_{E_2}$ for $t \in
  E_1 \cap E_2$, which has non-empty interior.  That $b_{E_1}=b_{E_2}$
  immediately implies $c_{E_1}=c_{E_2}$. Hence, there exist some $a > 0,
  b\in \R^k$, and $c\in \R$ such that for all sets $E\in \mc{M}$, we have
  $a_E = a, b_E = b$, and $c_E = c$.

  We complete the proof by showing that $\bigcup \{E \mid E \in \mc{M}\}$
  is dense in $\Omega$.
  Define $\Omega^\circ \subset \Omega$ as
  \begin{equation*}
    \Omega^\circ = \left\{t \in \Omega
    \mid
    \begin{array}{l}
      \mbox{there~exist~} u_3, \ldots, u_k \\
      \mbox{~~such that~} u\opt, v\opt, t, u_3,
      \ldots, u_k ~ \mbox{are~affinely~independent}
    \end{array}
    \right\}.
  \end{equation*}
  The set $\Omega^\circ$ forms a dense subset of $\Omega$, as $u\opt \in
  \interior \Omega \neq \varnothing$.  For any $t \in \Omega^\circ$,
  Lemma~\ref{lemma:points-in-simplices} guarantees the existence of $E \in
  \mc{M}$ such that $t \in E$ and $u\opt, v\opt \in \interior E$. We thus
  have $\Omega^\circ \subset \bigcup\{E \mid E \in \mc{M}\}$.  As $f_1(t) =
  a f_2(t) + b^Tt + c$ for all $t \in \Omega^\circ$ by the previous
  paragraph, Lemma~\ref{lemma:equal-functions} allows us to
  extend the equality to $f_1(t) = a
  f_2(t) + b^Tt + c$ for all $t \in \Omega$.
\end{proof}

\subsection{Proofs of auxiliary lemmas for
  Lemma~\ref{lemma:equivalence-on-simplex}}
\label{sec:proof-auxiliary-simplex-lemmas}

\subsubsection{Proof of Lemma~\ref{lemma:order-equivalence-rationals}}

For each $i$, let $\lambda_i = \hinge{\alpha_i}$ and $\nu_i =
\hingeneg{\alpha_i}$ be the positive and negative parts of the $\alpha_i$,
and let $\lambda_i = \frac{r_i}{s}$ and $-\nu_i = \frac{q_i}{s}$ where
$q_i, r_i, s \in \N$. Then we have
\begin{equation*}
  \sum_{i=1}^m \lambda_i u_i = v - \sum_{i=1}^m \nu_i u_i
  ~~ \mbox{or} ~~
  \sum_{i=1}^m r_i u_i = s v + \sum_{i=1}^m q_i u_i,
\end{equation*}
where $\ones^T r = s + \ones^T q$ as $\ones^T \alpha = 1$.
Defining the matrices
\begin{equation*}
  A = [\underbrace{u_1 ~ \cdots ~ u_1}_{r_1~{\rm times}}
    \cdots \underbrace{u_m ~ \cdots ~ u_m}_{r_m~{\rm times}}
  ]
  ~ \mbox{and} ~
  B = [\underbrace{v ~ \cdots ~ v}_{s~{\rm times}}
    ~ \underbrace{u_1 ~ \cdots ~ u_1}_{q_1~{\rm times}}
    \cdots \underbrace{u_m ~ \cdots ~ u_m}_{q_m~{\rm times}}
  ],
\end{equation*}
we have $A\ones = B\ones$ with columns in $\Omega$.
Then order equivalence~\eqref{eqn:equivalent-order-equivalence}
implies
\begin{equation*}
  \sum_{i=1}^m r_i f_1(u_i)
  \le s f_1(v) + \sum_{i=1}^m q_i f_1(u_i)
  ~~ \mbox{iff} ~~
  \sum_{i=1}^m r_i f_2(u_i)
  \le s f_2(v) + \sum_{i=1}^m q_i f_2(u_i).
\end{equation*}
Each of these is equivalent to
$\sum_{i=1}^m \alpha_i f_j(u_i) \le f_j(v)$ for $j = 1, 2$.

\subsubsection{Proof of Lemma~\ref{lemma:make-same-on-basis}}

We assume that $f_1$ and $f_2$ are non-linear over $\conv\{u_0, \ldots,
u_k\}$, as otherwise the result is trivial.
Let the vectors $g_1$ and $g_2$ and matrix $\entropy \in \R^{k \times k}$ be
defined by
\begin{equation*}
  \entropy = \left[\begin{matrix} u_1^T - u_0^T \\ 
      \vdots \\ u_k^T - u_0^T \end{matrix}\right],
  ~~
  g_1 = \left[\begin{matrix} f_1(u_1) - f_1(u_0) \\ \vdots
      \\ f_1(u_k) - f_1(u_0) \end{matrix}\right],
  ~~
  g_2 = \left[\begin{matrix} f_2(u_1) - f_2(u_0) \\ \vdots
      \\ f_2(u_k) - f_2(u_0) \end{matrix}\right].
\end{equation*}
For any $a > 0$, define the vector $b(a) \in \R^k$ so that
\begin{equation*}
  b(a)^T(u_i - u_0) = f_1(u_i) - f_1(u_0) - a(f_2(u_i) - f_2(u_0))
  = g_{1,i} - a g_{2,i},
\end{equation*}
that is, as $b(a) = \entropy^{-1} (g_1 - a g_2)$, which is possible as $\entropy$ is full
rank. By choosing $c(a) = f_1(u_0) - a f_2(u_0) - b(a)^T u_0$, we then
obtain that
\begin{align*}
  f_1(u_i) & = a f_2(u_i) + b(a)^T u_i + c(a)
\end{align*}
by algebraic manipulations. We now consider $\ucent$. We have
\begin{align*}
  b(a)^T \ucent
  & = \frac{1}{k + 1} \sum_{i=0}^k b(a)^T u_i
  = b(a)^T u_0
  + \frac{1}{k+1} \ones^T \entropy b(a) \\
  & = b(a)^T u_0 + \frac{1}{k+1} \ones^T(g_1 - a g_2),
\end{align*}
so that
\begin{align*}
  \lefteqn{a f_2(\ucent) + b(a)^T\ucent + c(a)} \\
  & = a f_2(\ucent) + 
  \sum_{i=1}^k
  \frac{f_1(u_i) - f_1(u_0) + a f_2(u_0) - a f_2(u_i)}{k+1}
  + f_1(u_0) - a f_2(u_0) \\ 
  & = a f_2(\ucent) + \frac{1}{k+1} \sum_{i=0}^k f_1(u_i)
  - a \frac{1}{k+1} \sum_{i=0}^k f_2(u_i).
\end{align*}
Thus we may choose an $a > 0$ such that our desired equalities hold if and
only if there exists $a > 0$ such that
\begin{equation*}
  f_1(\ucent) - \frac{1}{k+1} \sum_{i=0}^k f_1(u_i)
  = a f_2(\ucent) - a \frac{1}{k + 1} \sum_{i=0}^k f_2(u_i).
\end{equation*}
By the assumption that $f_1$ and $f_2$ are non-linear, we have that (by
Lemma~\ref{lemma:center-makes-linear}) $f_1(\ucent) < \frac{1}{k+1}
\sum_{i=0}^k f_1(u_i)$ and $f_2(\ucent) < \frac{1}{k+1} \sum_{i=0}^k
f_2(u_i)$. Thus, setting
\begin{equation*}
  a = \frac{
    f_1(\ucent) - \frac{1}{k+1} \sum_{i=0}^k f_1(u_i)}{
    f_2(\ucent) - \frac{1}{k+1} \sum_{i=0}^k f_2(u_i)}
  > 0
\end{equation*}
gives the desired result.

\subsubsection{Proof of Lemma~\ref{lemma:center-makes-linear}}
Without loss of generality, we assume that $\lambda_1 = \linf{\lambda}$
and that $\lambda_1 > 1/m$. Then we have
\begin{equation}
  \label{eqn:inequality-to-center-linear}
  \begin{split}
    f\left(\sum_{i=1}^m \lambda_i u_i\right)
    \le \sum_{i=1}^m \lambda_i f(u_i)
    & = \lambda_1 \sum_{i=1}^m f(u_i) + \sum_{i=1}^m (\lambda_i - \lambda_1)
    f(u_i) \\
    & = m \lambda_1 f(\ucent) + \sum_{i=2}^m (\lambda_1 - \lambda_i)
    f(u_i).
  \end{split}
\end{equation}
Let $\Lambda = \sum_{i=1}^m (\lambda_1 - \lambda_i)
= m \lambda_1 - 1 > 0$. Then
\begin{equation*}
  \ucent
  = \frac{1}{m \lambda_1} \sum_{i=1}^m \lambda_i u_i
  + \frac{\Lambda}{m \lambda_1} \sum_{i=1}^m \frac{\lambda_1 - \lambda_i}{
    \Lambda} u_i,
\end{equation*}
and we have
\begin{align*}
  m \lambda_1 f(\ucent)
  & \le n\lambda_1 \left[
    \frac{1}{m \lambda_1} f\left(\sum_{i=1}^m \lambda_i u_i\right)
    + \frac{\Lambda}{m \lambda_1}
    f\left(\sum_{i=2}^m \frac{\lambda_1 - \lambda_i}{\Lambda} u_i\right)
    \right] \\
  & \le f\left(\sum_{i=1}^m \lambda_i u_i\right)
  + \Lambda \sum_{i=2}^m \frac{\lambda_1 - \lambda_i}{\Lambda} f(u_i).
\end{align*}
The inequalities~\eqref{eqn:inequality-to-center-linear}
must have been equalities, giving the result.

\comment{
As outlined above, we first show Lemma~\ref{lemma:equivalence-subset} holds
for simplices. To do this, we require the definition of affine independence.
\begin{definition}
  \label{def:affine-independent}
  Points $u_0, u_1, \ldots u_m \in \R^k$, $m \le k + 1$, are \emph{affinely
    independent} if the vectors
  \begin{equation*}
    u_1 - u_0, u_2 - u_0, \ldots, u_m - u_0,
  \end{equation*}
  are linearly independent. A set $E \subset \R^k$ is a \emph{simplex} if $E
  = \conv\{u_0, u_1, \ldots, u_k\}$ where $u_0, \ldots, u_k$ are affinely
  independent.
\end{definition}
\noindent
Note that the simplex $E = \conv\{u_0, u_1, \ldots, u_k\} \subset \R^k$
has non-empty interior.

We now state and prove four technical lemmas that we use
to prove Lemma~\ref{lemma:equivalence-on-simplex}, which states
that Lemma~\ref{lemma:equivalence-subset} holds on simplices.
\begin{lemma}
  \label{lemma:order-equivalence-rationals}
  Let $u_1, \ldots, u_m \in \Omega$, $\alpha \in \Q^m$ satisfy
  $\ones^T\alpha = 1$, and $y \in
  \Omega$ with $y = \sum_{i=1}^m \alpha_i u_i$. If $f_1, f_2 : \Omega \to \R$
  are order equivalent, then
  \begin{equation*}
    \sum_{i=1}^m \alpha_i f_1(u_i) \le f_1(y)
    ~~ \mbox{if and only if} ~~
    \sum_{i=1}^m \alpha_i f_2(u_i) \le f_2(y).
  \end{equation*}
\end{lemma}
\begin{proof}
  For each $i$, let $\lambda_i = \hinge{\alpha_i}$ and $\nu_i =
  \hingeneg{\alpha_i}$ be the positive and negative parts of the $\alpha_i$,
  and let $\lambda_i = \frac{r_i}{s}$ and $-\nu_i = \frac{q_i}{s}$ where
  $q_i, r_i, s \in \N$. Then we have
  \begin{equation*}
    \sum_{i=1}^m \lambda_i u_i = y - \sum_{i=1}^m \nu_i u_i
    ~~ \mbox{or} ~~
    \sum_{i=1}^m r_i u_i = s y + \sum_{i=1}^m q_i u_i,
  \end{equation*}
  where $\ones^T r = s + \ones^T q$ as $\ones^T \alpha = 1$.
  Defining the matrices
  \begin{equation*}
    A = [\underbrace{u_1 ~ \cdots ~ u_1}_{r_1~{\rm times}}
      \cdots \underbrace{u_m ~ \cdots ~ u_m}_{r_m~{\rm times}}
    ]
    ~ \mbox{and} ~
    B = [\underbrace{y ~ \cdots ~ y}_{s~{\rm times}}
      ~ \underbrace{u_1 ~ \cdots ~ u_1}_{q_1~{\rm times}}
      \cdots \underbrace{u_m ~ \cdots ~ u_m}_{q_m~{\rm times}}
    ],
  \end{equation*}
  we have $A\ones = B\ones$ with columns in $\Omega$.
  Then order equivalence~\eqref{eqn:equivalent-order-equivalence}
  implies
  \begin{equation*}
    \sum_{i=1}^m r_i f_1(u_i)
    \le s f_1(y) + \sum_{i=1}^m q_i f_1(u_i)
    ~~ \mbox{iff} ~~
    \sum_{i=1}^m r_i f_2(u_i)
    \le s f_2(y) + \sum_{i=1}^m q_i f_2(u_i).
  \end{equation*}
  Each of these is equivalent to
  $\sum_{i=1}^m \alpha_i f_j(u_i) \le f_j(y)$ for $j = 1, 2$.
\end{proof}

\noindent
As an immediate consequence
of Lemma~\ref{lemma:order-equivalence-rationals}, we see that if $\alpha \in
\Q^n_+$ satisfies $\ones^T\alpha = 1$ and $u_1, \ldots, u_n \in \R^k_+$,
then
\begin{equation}
  \label{eqn:order-equivalent-equality}
  f_1\left(\sum_{i=1}^n \alpha_i u_i \right)
  = \sum_{i=1}^n \alpha_i f_1(u_i)
  ~~ \mbox{iff} ~~
  f_2\left(\sum_{i=1}^n \alpha_i u_i\right)
  = \sum_{i=1}^n \alpha_i f_2(u_i).
\end{equation}

Closed convex functions equal on a dense subset of a convex set $\Omega$ are
equal on $\Omega$ (e.g.~\cite[Proposition IV.1.2.5]{HiriartUrrutyLe93ab});
this means we need only prove equivalence results for convex
functions on dense subsets of their domains.
\begin{lemma}
  \label{lemma:equal-functions}
  Let $f_1, f_2 : \Omega \to \R$ be closed convex
  functions satisfying $f_1(t)
  = f_2(t)$ for all $t$ in a dense subset of $\Omega$. Then
  $f_1(t) = f_2(t)$ for $t \in \Omega$.
\end{lemma}

The next lemma shows that we can make the
equality~\eqref{eqn:relation-f-order-equivalent} hold for the extreme points
and centroid of any simplex.
\begin{lemma}
  \label{lemma:make-same-on-basis}
  Let $f_1, f_2 : \Omega \to \R$
  be closed convex and let $u_0, u_1, \ldots, u_k \in \Omega \subset \R^k$ be
  affinely independent. Then there exist $a > 0, b \in \R^k$, and $c$ such
  that $f_1(x) = a f_2(x) + b^T x + c$ for $x \in \{u_0, \ldots, u_k,
  \ucent\}$, where $\ucent = \frac{1}{k+1} \sum_{i=0}^k u_i$.
\end{lemma}

\begin{proof}
We assume that $f_1$ and $f_2$ are non-linear over $\conv\{u_0, \ldots,
u_k\}$, as otherwise the result is trivial.
Let the vectors $g_1$ and $g_2$ and matrix $X \in \R^{k \times k}$ be
defined by
\begin{equation*}
  X = \left[\begin{matrix} u_1^T - u_0^T \\ 
      \vdots \\ u_k^T - u_0^T \end{matrix}\right],
  ~~
  g_1 = \left[\begin{matrix} f_1(u_1) - f_1(u_0) \\ \vdots
      \\ f_1(u_k) - f_1(u_0) \end{matrix}\right],
  ~~
  g_2 = \left[\begin{matrix} f_2(u_1) - f_2(u_0) \\ \vdots
      \\ f_2(u_k) - f_2(u_0) \end{matrix}\right].
\end{equation*}
For any $a > 0$, define the vector $b(a) \in \R^k$ so that
\begin{equation*}
  b(a)^T(u_i - u_0) = f_1(u_i) - f_1(u_0) - a(f_2(u_i) - f_2(u_0))
  = g_{1,i} - a g_{2,i},
\end{equation*}
that is, as $b(a) = X^{-1} (g_1 - a g_2)$, which is possible as $X$ is full
rank. By choosing $c(a) = f_1(u_0) - a f_2(u_0) - b(a)^T u_0$, we then
obtain that
\begin{align*}
  f_1(u_i) & = a f_2(u_i) + b(a)^T u_i + c(a)
\end{align*}
by algebraic manipulations. We now consider $\ucent$. We have
\begin{align*}
  b(a)^T \ucent
  & = \frac{1}{k + 1} \sum_{i=0}^k b(a)^T u_i
  = b(a)^T u_0
  + \frac{1}{k+1} \ones^T X b(a)
  = b(a)^T u_0 + \frac{1}{k+1} \ones^T(g_1 - a g_2),
\end{align*}
so that
\begin{align*}
  \lefteqn{a f_2(\ucent) + b(a)^T\ucent + c(a)} \\
  & = a f_2(\ucent) + \frac{1}{k+1}
  \sum_{i=1}^k \left[f_1(u_i) - f_1(u_0) + a f_2(u_0) - a f_2(u_i)\right]
  + f_1(u_0) - a f_2(u_0) \\ 
  & = a f_2(\ucent) + \frac{1}{k+1} \sum_{i=0}^k f_1(u_i)
  - a \frac{1}{k+1} \sum_{i=0}^k f_2(u_i).
\end{align*}
Thus we may choose an $a > 0$ such that our desired equalities hold if and
only if there exists $a > 0$ such that
\begin{equation*}
  f_1(\ucent) - \frac{1}{k+1} \sum_{i=0}^k f_1(u_i)
  = a f_2(\ucent) - a \frac{1}{k + 1} \sum_{i=0}^k f_2(u_i).
\end{equation*}
By the assumption that $f_1$ and $f_2$ are non-linear, we have that (by
Lemma~\ref{lemma:center-makes-linear}) $f_1(\ucent) < \frac{1}{k+1}
\sum_{i=0}^k f_1(u_i)$ and $f_2(\ucent) < \frac{1}{k+1} \sum_{i=0}^k
f_2(u_i)$. Thus, setting
\begin{equation*}
  a = \frac{
    f_1(\ucent) - \frac{1}{k+1} \sum_{i=0}^k f_1(u_i)}{
    f_2(\ucent) - \frac{1}{k+1} \sum_{i=0}^k f_2(u_i)}
  > 0
\end{equation*}
gives the desired result.
\end{proof}

\noindent
Lastly, we characterize the linearity of convex functions over convex
hulls of finite collections of points.
\begin{lemma}
  \label{lemma:center-makes-linear}
  Let $f : \Omega \to \R$ be convex with $u_1, \ldots, u_m \in \Omega$ and
  $\ucent = \frac{1}{m} \sum_{i=1}^m u_i$. If
  $f(\ucent) = \frac{1}{m} \sum_{i=1}^m f(u_i)$, then
  for all $\lambda \in \R^m_+$ such that $\ones^T \lambda = 1$,
  \begin{equation*}
    f\left(\sum_{i=1}^m \lambda_i u_i\right) = \sum_{i=1}^m \lambda_i f(u_i).
  \end{equation*}
\end{lemma}
\begin{proof}
Without loss of generality, we assume that $\lambda_1 = \linf{\lambda}$
and that $\lambda_1 > 1/m$. Then we have
\begin{equation}
  \label{eqn:inequality-to-center-linear}
  \begin{split}
    f\left(\sum_{i=1}^m \lambda_i u_i\right)
    \le \sum_{i=1}^m \lambda_i f(u_i)
    & = \lambda_1 \sum_{i=1}^m f(u_i) + \sum_{i=1}^m (\lambda_i - \lambda_1)
    f(u_i) \\
    & = m \lambda_1 f(\ucent) + \sum_{i=2}^m (\lambda_1 - \lambda_i)
    f(u_i).
  \end{split}
\end{equation}
Let $\Lambda = \sum_{i=1}^m (\lambda_1 - \lambda_i)
= m \lambda_1 - 1 > 0$. Then
\begin{equation*}
  \ucent = \frac{1}{m \lambda_1} \sum_{i=1}^m \lambda_i u_i
  + \frac{1}{m \lambda_1} \sum_{i=1}^m (\lambda_1 - \lambda_i) u_i
  = \frac{1}{m \lambda_1} \sum_{i=1}^m \lambda_i u_i
  + \frac{\Lambda}{m \lambda_1} \sum_{i=1}^m \frac{\lambda_1 - \lambda_i}{
    \Lambda} u_i,
\end{equation*}
and we have
\begin{align*}
  m \lambda_1 f(\ucent)
  & \le n\lambda_1 \left[
    \frac{1}{m \lambda_1} f\left(\sum_{i=1}^m \lambda_i u_i\right)
    + \frac{\Lambda}{m \lambda_1}
    f\left(\sum_{i=2}^m \frac{\lambda_1 - \lambda_i}{\Lambda} u_i\right)
    \right] \\
  & \le f\left(\sum_{i=1}^m \lambda_i u_i\right)
  + \Lambda \sum_{i=2}^m \frac{\lambda_1 - \lambda_i}{\Lambda} f(u_i).
\end{align*}
In particular, the inequalities~\eqref{eqn:inequality-to-center-linear}
must have been equalities, giving the result.
\end{proof}

We now state and prove Lemma~\ref{lemma:equivalence-subset} for simplices
using the preceding four lemmas, assuming that
$\interior{\Omega} \neq \varnothing$.
\begin{lemma}
  \label{lemma:equivalence-on-simplex}
  Let $E = \conv \{u_0, \ldots, u_k\} \subset \Omega$ where $u_0, \ldots, u_k$
  are affinely independent. If $f_1$ and $f_2$ are order equivalent,
  then there exist $a > 0$, $b \in \R^k$, and $c \in \R$ such that
  \begin{equation*}
    f_1(t) = a f_2(t) + b^Tt + c
    ~~ \mbox{for~all~} t \in E.
  \end{equation*}
\end{lemma}
\begin{proof}
  We reduce the problem to applying in the standard simplex for notational
  simplicity.
  Let $\delta_i = u_i - u_0$, $1 \leq i \leq k$,
  be a linearly independent basis for $\R^k$ and let $D =
  [\delta_1 ~ \cdots ~ \delta_k] \in \R^{k \times k}$. Define the two
  convex functions $g_1 : \R^k_+ \to \R$ and $g_2 : \R^k_+ \to
  \R$ via
  \begin{equation*}
    g_j(y) = f_j(u_0 + D y) = f_j\left(u_0 + \sum_{i=1}^k \delta_i y_i \right)
    = f_j\left((1 - \ones^Ty) u_0 + \sum_{i=1}^k u_i y_i \right)
  \end{equation*}
  for $j = 1, 2$.  If we define $Y = \{y \in \R^k_+ \mid \ones^T y \le 1\}$,
  then $g_1$ and $g_2$ are continuous, defined, convex, and
  order equivalent on $Y$. We make one further reduction.
  Let $e_i \in \R^k$ for $1 \leq i \leq k$ be the standard basis for $\R^k$
  and $e_0 = \zeros$ be shorthand for the all-zeros vector. Further, let
  $\ecenter = \frac{1}{k+1} \sum_{i=0}^k e_i$ be the centroid of $Y$ (so $Y
  = \conv\{e_0, \ldots, e_k\}$).  Lemma~\ref{lemma:make-same-on-basis}
  guarantees the existence of $a > 0, b\in \R^k$, and $c \in \R$ such that
  \begin{equation*}
    g_1(y) = a g_2(y) + b^T y + c
    ~~ \mbox{for} ~ y \in \{e_0, e_1, \ldots, e_k, \ecenter\}.
  \end{equation*}
  Now, let $h_1(y) = g_1(y)$ and $h_2(y) = a g_2(y) + b^T y + c$, so $h_1$
  and $h_2$ are convex, order equivalent on $Y$, and satisfy $h_1(y) =
  h_2(y)$ for $y \in \{e_0, \ldots, e_k ,\ecenter\}$. We have reduced
  Lemma~\ref{lemma:equivalence-on-simplex} to showing that
  \begin{equation}
    h_1(y) = h_2(y) ~~ \mbox{for~} y \in Y = \{y \in \R^k_+ \mid
    \ones^Ty \le 1\}.
    \label{eqn:h-equal-goody}
  \end{equation}

  We divide our discussion into two cases.
  
  \textbf{Linear case:} Suppose that $h_1(\ecenter) =
  \frac{1}{k+1}\sum_{i=0}^k h_1(e_i)$.
  Then by order equivalence of $h_1$ and
  $h_2$ (Eq.~\eqref{eqn:order-equivalent-equality}) we have
  $h_2(\ecenter) = \frac{1}{k+1} \sum_{i=0}^{k} h_2(e_i)$.
  Lemma~\ref{lemma:center-makes-linear} thus implies that $h_1$ and $h_2$ are
  linear on $Y = \conv\{e_0, \ldots, e_k\}$, equal on
  the vertices of $Y$, and hence equal on its interior.

  \textbf{Nonlinear case:} By convexity we have
  $h_1(\ecenter) < \frac{1}{k+1} \sum_{i=0}^k h_1(e_i)$, and order equivalence
  (Lemma~\ref{lemma:order-equivalence-rationals}) implies
  $h_2(\ecenter) < \frac{1}{k+1} \sum_{i=0}^k h_2(e_i)$.  For $y \in Y =
  \conv\{e_0, \ldots, e_k\}$, we use $y_0 = 1 - \ones^T y$ for shorthand,
  so we may write $y = \sum_{i=0}^k y_i e_i$ and have $\sum_{i=0}^k y_i =
  1$, so $[y_0 ~ y_1 ~ \cdots ~ y_k]^T \in \simplex_{k+1}$.  Now,
  fix an
  arbitrary $y \in Y \cap \Q^k$. We would like to show that $h_1(y) = h_2(y)$.
  To that end we consider consider the gaps due convexity of $h_j$ from
  $\ecenter$ to the values of $h_j$ at $e_i$ relative to that from $h_j(y)$ to
  $h_j(e_i)$, defining the linear functions $\varphi_j : [0, 1] \to \R$,
  $j = 1, 2$,
  by
  \begin{equation*}
    \varphi_j(r)
    \defeq
    (1 - r)\left[h_j(\ecenter)
      - \frac{1}{k + 1} \sum_{i = 0}^k h_j(e_i)\right]
    + r \left[\sum_{i = 0}^k y_i h_j(e_i)
      - h_j(y) \right].
  \end{equation*}
  Then we have
  \begin{equation*}
    \varphi_j(0) = h_j(\ecenter) - \frac{1}{k+1} \sum_{i=0}^k h_j(e_i) < 0
  \end{equation*}
  by assumption, and by convexity,
  \begin{equation*}
    \varphi_j(1) = \sum_{i = 0}^k y_i h_j(e_i) - h_j(y) \ge 0.
  \end{equation*}
  The key is that the order equivalence of $h_1$ and $h_2$ on
  $Y$ implies that
  \begin{equation}
    \label{eqn:equal-signs-super}
    \sign(\varphi_1(r)) = \sign(\varphi_2(r))
    ~~ \mbox{for~} r \in [0, 1],
  \end{equation}
  so that $\varphi_1$ and $\varphi_2$ have the same zero crossing
  $r\opt > 0$, i.e.\
  there exists $0 < r\opt \le 1$ with $\varphi_1(r\opt) = \varphi_2(r\opt)
  = 0$. (We prove equality~\eqref{eqn:equal-signs-super} presently.)
  At this $r\opt > 0$, we find
  \begin{align*}
    0 & = \varphi_1(r\opt) - \varphi_2(r\opt)
    = -r\opt h_1(y) + r\opt h_2(y),
  \end{align*}
  where we have used that $h_1(e_i) = h_2(e_i)$ for $i = 0, \ldots, k$
  and $h_1(\ecenter) = h_2(\ecenter)$. 
  That is, $h_1(y) = h_2(y)$, and as $y \in Y \cap \Q^k$ is arbitrary and
  $\Q^k$ is dense in $\R^k$,
  Lemma~\ref{lemma:equal-functions} extends the equality $h_1 \equiv h_2$
  to all of $Y$.
  That is, expression~\eqref{eqn:h-equal-goody} holds.

  Returning to the sign equivalence~\eqref{eqn:equal-signs-super},
  for $r > 0$, we may divide $\varphi_j(r)$ by $r$
  and
  $\varphi_j(r) \le 0$ if and only if
  \begin{equation*}
    \frac{1 - r}{r} \left[h_j(\ecenter)
      - \frac{1}{k+1}
      \sum_{i = 0}^k h_j(e_i)\right]
    + \sum_{i = 0}^k y_i h_j(e_i) \le h_j(y).
  \end{equation*}
  Defining
  $\alpha_i = y_i - \frac{1 - r}{r(k + 1)} \in \Q$
  for $i = 0, \ldots, k$ and $\alpha_{k+1} = \frac{1-r}{r}$,
  the inequality $\varphi_j(r) \le 0$ is evidently equivalent to
  $\sum_{i=0}^k \alpha_i h_j(e_i) + \alpha_{k+1} h_j(\ecenter) \le h_j(y)$.
  A calculation shows that $\ones^T\alpha = 1$ and
  $\sum_{i = 0}^k \alpha_i e_i + \alpha_{k + 1} \ecenter
  = y$.
  Applying Lemma~\ref{lemma:order-equivalence-rationals} immediately yields
  that $\varphi_1(r) \le 0$ if and only if $\varphi_2(r) \le 0$ for all $r \in
  \openleft{0}{1} \cap \Q$.  Noting that $\varphi_1(0) < 0$ and $\varphi_2(0) <
  0$, we obtain the sign equality~\eqref{eqn:equal-signs-super}.
\end{proof}
}


\section{Proofs for Bayes Consistency of Empirical Risk Minimization}

\subsection{Proof of Lemma~\ref{lemma:fisher-gap}}
\label{appendix:proof-fisher-gap}

With the definition $\risk\opt(\quant) \defeq
\inf_{\discfunc} \risk(\discfunc \mid \quant)$,
we have
\begin{equation}
  \label{eqn:expansion-iad}
  \risk(\discfunc \mid \quant)
  - \inf_{\discfunc \in \Discfuncs, \quant \in \quantfam}
  \risk(\discfunc \mid \quant)
  = \risk(\discfunc \mid \quant)
  - \risk\opt(\quant) + \risk\opt(\quant)
  - \inf_{\quant \in \quantfam} \risk\opt(\quant).
\end{equation}
For the second two terms, we note that by
Theorem~\ref{theorem:loss-equivalent-entropy}, we have
\begin{equation*}
  \risk\opt(\quant) =
  \E[\entropy_{\wzoloss}(\wt{\prior}(\quant(X)))] = a
  \E[\entropy_\loss(\wt{\prior}(\quant(X)))] + b^T \prior + c,
\end{equation*}
and similarly for $\inf_\quant \risk\opt(\quant)$, so that
\begin{align*}
  \risk\opt(\quant) - \inf_{\quant \in \quantfam} \risk\opt(\quant)
  & = a \left[\surrrisk\opt(\quant) - \inf_{\quant \in \quantfam}
    \surrrisk\opt(\quant)\right] \\
  & \le a \left[\surrrisk(\discfunc \mid \quant) -
    \inf_{\discfunc \in \Discfuncs, \quant \in \quantfam}
    \surrrisk(\discfunc \mid \quant)\right].
\end{align*}
Clearly $t \mapsto a t$ is concave, so that it remains
to bound $\risk(\discfunc \mid \quant) - \risk\opt(\quant)$
in expression~\eqref{eqn:expansion-iad}.
To that end, let $\loss(\alpha) \in \R^k$ denote the
vector of losses $\loss(\alpha, y)$ and define the function
\begin{equation*}
  H(\epsilon)
  = \inf_{\prior \in \simplex_k, \alpha}
  \left\{\sup_{\alpha'}
  \prior^T\left(\loss(\alpha) - \loss(\alpha')\right)
  \mid \sup_{\alpha'}
  \prior^T \left(\wzoloss(\alpha)
  - \wzoloss(\alpha')\right)
  \ge \epsilon \right\}
\end{equation*}
and let $H^{**}$ be its Fenchel biconjugate. Then (see \cite[Proposition 25
  and Corollary 26]{Zhang04a}, as well as the papers~\cite[Proposition
  1]{TewariBa07,Steinwart07,DuchiMaJo13}) we have that $H^{**}(\epsilon) >
0$ for all $\epsilon > 0$, and defining $g(\epsilon) = \sup\{\delta : \delta
\ge 0, H^{**}(\delta) \le \epsilon\}$ yields the desired concave function by
taking $h(t) = g(t) + at$. In passing, we note that we may w.l.o.g.\ replace
$h(t)$ with $\min\{h(t), \max_{ij} c_{ij}\}$.

Now we give the second result. Without loss of generality, we may assume
that the vector $\discfunc(z) \in \R^k$ has a unique maximal coordinate;
we may otherwise assume a deterministic rule for breaking ties. Let
$y(\discfunc(z)) = \argmax_j \discfunc_j(z)$, assumed w.l.o.g.\ to be
unique. Consider that
\begin{align*}
  \risk(\discfunc \mid \quant) - \risk\opt(\quant)
  & = \int \sup_\alpha \sum_{i = 1}^k \left[\wt{\prior}_i(z)
  \wzoloss(\discfunc(z), i)
  - \wzoloss(\alpha, i)\right] d\wb{P}^\quant(z)
\end{align*}
where $\bar{P}^\quant(A) = \sum_{i=1}^k \prior_i P_i(\quant^{-1}(A))$
(for measurable $A \subset \mc{Z}$)
is the push-forward of $\quant$ and
$\wt{\prior}(z) = [\prior_i dP_i^\quant(z)]_{i=1}^k / d\wb{P}^\quant(z)$
is shorthand for the posterior of $Y$ conditional on observing
$z \in \quant(\mc{X})$.  Lemma~\ref{lemma:hinge-great-gap}
immediately implies
\begin{align*}
  \risk(\discfunc \mid \quant) - \risk\opt(\quant)
  & \le 
  \int \sup_\alpha \sum_{i = 1}^k \left[\wt{\prior}_i(z)
  \loss(\discfunc(z), i)
  - \loss(\alpha, i)\right] d\wb{P}^\quant(z) \\
  & = \risk_\loss(\discfunc \mid \quant)
  - \risk_\loss\opt(\quant)
\end{align*}
for either of the hinge-type losses, e.g.\ $\loss(\alpha, y) = \sum_{i =
  1}^k c_{yi} \hinge{1 + \alpha_i - \alpha_y}$.
In the notation of the previous (general) case, we have
$a = 1/k$ because $\entropy_\loss(\prior) = \frac{1}{k}
\entropy_{\wzoloss}(\prior)$, whence we may take $h(t) = (1 + 1/k) t$.

\subsection{Proof of Theorem~\ref{theorem:bayes-consistency}}
\label{appendix:proof-bayes-consistency}

The proof of Theorem~\ref{theorem:bayes-consistency} is almost immediate
from Lemma~\ref{lemma:fisher-gap}.
Indeed,
by
Lemma~\ref{lemma:fisher-gap}, we have
\begin{equation*}
  \risk(\discfunc_n \mid \quant_n)
  - \risk\opt(\quantfam)
  \le h\left(\risk_\loss(\discfunc_n \mid \quant_n)
  - \risk_\loss\opt(\quantfam)\right)
\end{equation*}
for some $h$ concave, bounded, and satisfying $h(0) = 0$.  It is thus
sufficient to show that $\E[\risk_\loss(\discfunc_n \mid \quant_n) -
  \risk_\loss\opt(\quantfam)] \to 0$.  Now, if we let $\discfunc_n\opt$
and $\quant_n\opt$ minimize $\risk_\loss(\discfunc \mid \quant)$ over the
set $\Discfuncs_n \times \quantfam_n$ (or be arbitrarily close to
minimizing the $\loss$-risk), we have
\begin{align*}
  \lefteqn{\risk_\loss(\discfunc_n \mid \quant_n)
    - \risk_\loss(\discfunc_n\opt \mid \quant_n\opt)} \\
  & = \risk_\loss(\discfunc_n \mid \quant_n)
  - \what{\risk}_{\loss,n}(\discfunc_n \mid \quant_n)
  + \what{\risk}_{\loss,n}(\discfunc_n \mid \quant_n)
  - \risk_\loss(\discfunc_n\opt \mid \quant_n \opt) \\
  & \le
  \underbrace{\risk_\loss(\discfunc_n \mid \quant_n)
    - \what{\risk}_{\loss,n}(\discfunc_n \mid \quant_n)
    + \what{\risk}_{\loss,n}(\discfunc_n\opt \mid \quant_n \opt)
    - \risk_\loss(\discfunc_n\opt \mid \quant_n \opt)}_{
    \le 2 \sup_{\discfunc \in \Discfuncs_n, \quant \in \quantfam_n}
    |\what{\risk}_{\loss,n}(\discfunc \mid \quant)
    - \risk_\loss(\discfunc \mid \quant)|} \\
  & \qquad ~  + \what{\risk}_{\loss,n}(\discfunc_n \mid \quant_n)
  - \inf_{\discfunc \in \Discfuncs_n,
    \quant \in \quantfam_n} \what{\risk}_{\loss,n}(\discfunc
  \mid \quant).
\end{align*}
Consequently, we have the expectation bound
\begin{align*}
  \lefteqn{\E\left[\risk_\loss(\discfunc_n \mid \quant_n)
      - \risk_\loss\opt(\quantfam)\right]} \\
  & \le 
  \E\left[\risk_\loss(\discfunc_n \mid \quant_n)
    - \risk_\loss(\discfunc_n\opt \mid \quant_n \opt)\right]
  + \risk_\loss(\discfunc_n\opt \mid \quant_n \opt)
  - \risk_\loss\opt(\quantfam)
  \le 2\epsestimate_n + \epsopt_n
  + \epsapprox_n,
\end{align*}
which converges to zero as desired.


\section{Proofs of basic properties of $f$-divergences}
\label{sec:fdiv-proofs}

In this section, we collect the proofs of the characterizations of
generalized $f$-divergences (Def.~\ref{def:general-fdiv}).

\subsection{Proof of Lemma ~\ref{lemma:fdiv-mu-fine}}
\label{sec:proof-fdiv-mu-fine}

Let $\mu_1$ and $\mu_2$ be
dominating measures; then $\mu = \mu_1 + \mu_2$ also dominates
$P_1, \ldots, P_k$ as well as $\mu_1$ and $\mu_2$. We have
for $\nu = \mu_1$ or $\nu = \mu_2$ that
\begin{equation*}
  \frac{dP_i / d\nu}{
    dP_k / d\nu}
  = \frac{dP_i / d \nu}{dP_k / d\nu} \frac{d\nu / d\mu}{d\nu / d\mu}
  = \frac{dP_i / d\mu}{dP_k / d \mu}
  ~~ \mbox{and} ~~
  \frac{dP_i}{d\nu} \frac{d\nu}{d\mu}
  = \frac{dP_i}{d\mu},
\end{equation*}
the latter two equalities holding $\mu$-almost surely by definition of the
Radon-Nikodym derivative. Thus we obtain for $\nu = \mu_1$ or $\nu = \mu_2$
that
\begin{align*}
  \int f\left(\frac{dP_{1:k-1} / d\nu}{dP_k / d\nu}\right)
  \frac{dP_k}{d\nu} d\nu
  & = \int f\left(\frac{dP_1 / d\nu}{dP_k / d\nu},
  \ldots, \frac{dP_{k-1} / d\nu}{dP_k / d\nu}
  \right) \frac{dP_k}{d\nu} \frac{d\nu}{d\mu} d\mu \\
  & = \int f\left(\frac{dP_1 / d\mu}{dP_k / d\mu},
  \ldots, \frac{dP_{k-1} / d\mu}{dP_k / d\mu}\right)
  \frac{dP_k}{d\mu} d\mu,
\end{align*}
again by definition of the Radon-Nikodym derivative. We have
$\frac{p_i}{p_k} = \frac{dP_i / d\mu}{dP_k / d\mu}$ a.s.-$\mu$,
which shows that the base measure $\mu$ does not affect the integral.

To see the positivity, we may take $\mu = \frac{1}{k} \sum_{i=1}^k P_i$, in
which case Jensen's inequality implies (the
perspective function~\eqref{eqn:closed-perspective} is convex) that
\begin{align*}
  \fdiv{P_1, \ldots, P_{k-1}}{P_k}
  & = \E \left[f\left(\frac{dP_1}{dP_k}(X),
    \ldots, \frac{dP_{k-1}}{dP_k}(X) \right) dP_k(X)\right] \\
  & \ge f\left(\frac{\E[dP_1]}{\E[dP_k]},
  \ldots, \frac{\E[dP_{k-1}]}{\E[dP_k]}\right) \E[dP_k]
  = f(\ones) = 0,
\end{align*}
where the expectation is taken under the distribution $\mu$. The
inequality is strict for $f$ strictly convex at $\ones$ as long as $dP_i /
dP_k$ is non-constant for some $i$, meaning that there exists an $i$ such
that $P_i \neq P_k$.

\subsection{Proof of Proposition~\ref{proposition:fdiv-sup}}
\label{sec:proof-fdiv-sup}

Before proving the proposition, we first establish a 
more general continuity result for $f$-divergences. 
This result is a direct generalization of
results of \cite[Thm.~5]{Vajda72}.
Given a sub-$\sigma$-algebra $\mc{G} \subset \mc{F}$,
we let $P^{\mc{G}}$ denote the restriction of the measure $P$, defined
on $\mc{F}$, to $\mc{G}$. 
\begin{lemma}
  \label{lemma:filtration-limits}
  Let $\mc{F}_1 \subset \mc{F}_2 \subset \ldots$ be a sequence of
  sub-$\sigma$-algebras of $\mc{F}$ and let $\mc{F}_\infty
  = \sigma(\cup_{n \ge 1} \mc{F}_n)$. Then
  $\fdiv{P_1^{\mc{F}_n}, \ldots, P_{k-1}^{\mc{F}_n}}{P_k^{\mc{F}_n}}$
  is non-decreasing in $n$ and
  \begin{equation*}
    \lim_{n \to \infty} \fdiv{P_1^{\mc{F}_n}, \ldots, P_{k-1}^{\mc{F}_n}}{
      P_k^{\mc{F}_n}}
    =  \fdiv{P_1^{\mc{F}_\infty}, \ldots,
      P_{k-1}^{\mc{F}_\infty}}{P_k^{\mc{F}_\infty}}.
  \end{equation*}
\end{lemma}
\begin{proof}
  Define the measure
  $\nu = \frac{1}{k} \sum_{i=1}^k P_i$ and vectors
  $V^n$ via the Radon-Nikodym derivatives
  \begin{equation*}
    V^n = \frac{1}{k} \cdot \left( \frac{dP_1^{\mc{F}_n}}{d\nu^{\mc{F}_n}},
    \frac{dP_2^{\mc{F}_n}}{d\nu^{\mc{F}_n}},
    \ldots, \frac{dP_{k-1}^{\mc{F}_n}}{d\nu^{\mc{F}_n}}\right).
  \end{equation*}
  Then $(1 - \ones^T V^n) d\nu^{\mc{F}_n} = \frac{1}{k}dP_k^{\mc{F}_n}$, and
  $V^n$ is a martingale adapted to the filtration
  $\mc{F}_n$ by standard properties of conditional expectation (under
  the measure $\nu$).
  Letting  $C_k = \{v \in \R^{k-1}_+ \mid \ones^Tv \le 1\}$,
  we define $g : C_k \to \R$ by
  \begin{equation*}
    g(v) = f\left(\frac{v}{1 - \ones^T v}\right) (1 - \ones^Tv).
  \end{equation*}
  We see that $g$ is convex (it is a perspective function), and we have
  \begin{equation*}
    \E_\nu[g(V^n)] =
    \int f\left(\frac{dP_{1:k-1}^{\mc{F}_n}}{
      dP_k^{\mc{F}_n}}
    \right)\left(1 - \frac{1}{k}\sum_{i=1}^{k-1} \frac{dP_i^{\mc{F}_n}}{
      d\nu^{\mc{F}_n}}\right) d\nu
    = \frac{1}{k}
    \fdiv{P_{1:k-1}^{\mc{F}_n}}{P_k^{\mc{F}_n}}.
  \end{equation*}
  Because $V^n \in C_k$ for all $n$, we see that $g(V^n)$ is a
  submartingale. This gives the first result of the lemma,
  that $\fdivs{P_{1:k-1}^{\mc{F}_n}}{P_k^{\mc{F}_n}}$ is non-decreasing in $n$.

  Now,
  assume that the limit in the second statement is finite, as otherwise the
  result is trivial. Using $f(\ones) = 0$, we have by convexity that
  for $v \in C_k$,
  \begin{align*}
    g(v) = f\left(\frac{v}{1 - \ones^Tv}\right)(1 - \ones^Tv)
    + f(1) \ones^Tv
    & \ge f\left(v + 1 (\ones^Tv)\right) \\
    & \ge \inf_{v \in C_k} f(v + 1(\ones^T v)) > -\infty,
  \end{align*}
  the final inequality a consequence of the fact that $f$ is closed
  and hence attains its infimum. In particular,
  the sequence $g(V^n) - \inf_{v \in C_k} g(v)$ is a non-negative
  submartingale, and thus
  \begin{equation*}
    \sup_n \E_\nu\left[ |g(V^n)  - \inf_{v \in C_k} g(v)|\right]
    = \lim_n \E_\nu\left[ g(V^n) - \inf_{v \in C_k} g(v)\right]
    < \infty.
  \end{equation*}
  With this integrability guarantee, Doob's second martingale
  convergence theorem~\cite[Thm.~35.5]{Billingsley86} yields the
  existence of a vector $V^\infty \in \mc{F}_\infty$ such that
  \begin{equation*}
    0 \le \lim_n \E_\nu\left[g(V^n) - \inf_{v \in C_k} g(v)\right]
    = \E_\nu[g(V^\infty)] - \inf_{v \in C_k} g(v)
    < \infty.
  \end{equation*}
  That $\inf_{v \in C_k} g(v) > -\infty$ implies
  $\E_\nu[|g(V^\infty)|] < \infty$, giving the lemma.
\end{proof}

We now give the proof of Proposition~\ref{proposition:fdiv-sup} proper.
Let the the base measure $\mu = \frac{1}{k} \sum_{i=1}^k P_i$ and let $p_i
= \frac{dP_i}{d\mu}$ be the associated densities of the $P_i$. Define the
increasing sequence of partitions $\mc{P}^n$ of $\mc{X}$ by sets
$A_{\alpha n}$ for vectors $\alpha = (\alpha_1, \ldots, \alpha_{k-1})$ and
$B$ with
\begin{equation*}
  A_{\alpha n} = \left\{x \in \mc{X} \mid
  \frac{\alpha_j - 1}{2^n} \le \frac{p_j(x)}{p_k(x)}
  < \frac{\alpha_j}{2^n} ~ \mbox{for~} j = 1, \ldots, k-1\right\}
\end{equation*}
where we let each $\alpha_j$ range over $\{-n 2^n, -n 2^n + 1, \ldots, n
2^n\}$, and define and $B = (\cup_\alpha A_{\alpha n})^c = \mc{X}
\setminus \cup_\alpha A_{\alpha n}$.  Then we
have
\begin{equation*}
  \left(\frac{p_1(x)}{p_k(x)},
  \ldots, \frac{p_{k-1}(x)}{p_k(x)}\right)
  = \lim_{n \to \infty} \sum_{\alpha \in \{-n2^n, \ldots, n2^n\}^{k-1}}
  \frac{\alpha}{2^n}
  \characteristic{A_{\alpha n}}(x) + (n, \ldots, n) \characteristic{B}(x),
\end{equation*}
where $\characteristic{A}$ denotes the $\{0,1\}$-valued characteristic
function of the set $A$.  Each term on the right-hand-side of the previous
display is $\mc{F}_n$-measurable, where $\mc{F}_n$ denotes the
sub-$\sigma$-field generated by the partition $\mc{P}^n$. Defining
$\mc{F}_\infty = \sigma(\cup_{n \ge 1} \mc{F}_n)$, we have
\begin{align*}
  \fdiv{P_{1:k-1}}{P_k \mid \mc{P}^n}
  & = \fdiv{P_1^{\mc{F}_n}, \ldots, P_{k-1}^{\mc{F}_n}}{P_k^{\mc{F}_n}} \\
  & \mathop{\rightarrow}_{n} 
  \fdiv{P_1^{\mc{F}_\infty}, \ldots, P_{k-1}^{\mc{F}_\infty}}{
    P_k^{\mc{F}_\infty}}
  = \fdiv{P_1, \ldots, P_{k-1}}{P_k},
\end{align*}
where the limiting operation follows by Lemma~\ref{lemma:filtration-limits}
and the final equality because of the measurability containment
$(\frac{p_1}{p_k}, \ldots, \frac{p_{k-1}}{p_k}) \in \mc{F}_\infty$.

\subsection{Proof of Proposition~\ref{proposition:data-processing}}
\label{sec:proof-data-processing}

Before proving the proposition, we state an inequality that generalizes the
classical log-sum inequality (cf.~\cite[Theorem 2.7.1]{CoverTh06}).
\begin{lemma}
  \label{lemma:general-log-sum}
  Let $f: \R^k_+ \to \extendedR$ be convex. Let $a : \mc{X} \to \R_+$ and
  $b_{i} : \mc{X} \to \R_+$ for $1 \leq i \leq k$ be nonnegative measurable
  functions. Then for any finite measure $\mu$ defined on $\mc{X}$, we have
  \begin{equation*}
    \left(\int a d\mu \right)
    f\left(\frac{\int b_1 d\mu}{\int a d\mu},
    \ldots, \frac{\int b_k d\mu}{\int a d\mu} \right)
    \le
    \int a(x) f\left(\frac{b_{1}(x)}{a(x)}, \ldots,
    \frac{b_{k}(x)}{a(x)}\right) d\mu(x).
  \end{equation*}
\end{lemma}
\begin{proof}
  Recall that the perspective
  function $g(y, t)$ defined by $g(y, t) = t f(y / t)$ for $t > 0$ is
  jointly convex in $y$ and $t$. The measure $\nu = \mu / \mu(\mc{X})$
  defines a probability measure, so that for $X \sim \nu$,
  Jensen's inequality implies
  \begin{align*}
    \lefteqn{
      \left(\int a d\nu \right)
      f\left(\frac{\int b_1 d\nu}{\int a d\nu},
      \ldots, \frac{\int b_k d\nu}{\int a d\nu} \right)
      \le \E_\nu\left[a(X) f \left(\frac{b_1(X)}{a(X)},
        \ldots, \frac{b_k(X)}{a(X)}\right)\right]} \\
    & \qquad\qquad\qquad\qquad\qquad\qquad~ 
    = \frac{1}{\mu(\mc{X})}
    \int a(x) f\left(\frac{b_{1}(x)}{a(x)}, \ldots,
    \frac{b_{k}(x)}{a(x)}\right) d\mu(x).
  \end{align*}
  Noting that $\int b_i d\nu / \int a d\nu = \int b_i d\mu / \int a d\mu$
  gives the result.
\end{proof}

We use this lemma and Proposition~\ref{proposition:fdiv-sup} to give
Proposition~\ref{proposition:data-processing}.
Proposition~\ref{proposition:fdiv-sup} implies
\begin{align*}
  \lefteqn{\fdiv{Q_{P_1}, \ldots, Q_{P_{k-1}}}{Q_{P_k}}} \\
  & = \sup_{\mc{P}}
  \left\{\fdiv{Q_{P_1}, \ldots, Q_{P_{k-1}}}{Q_{P_k} \mid \mc{P}}
  : \mc{P} ~ \mbox{is~a~finite~partition~of~} \mc{Z}\right\}.
\end{align*}
It is consequently no loss of generality to assume that $\mc{Z}$ is
finite and $\mc{Z} = \{1,2,\ldots,m\}$. Let $\mu = \sum_{i=1}^k P_i$
be a dominating measure and let $p_i = dP_i / d\mu$. Letting
$q(j \mid x) = Q(\{j\} \mid x)$ and
$q_{P_i}(j) = Q_{P_i}(\{j\})$, we obtain
\begin{align*}
  \lefteqn{\fdiv{Q_{P_1}, \ldots, Q_{P_{k-1}}}{Q_{P_k}}
    = \sum_{i = 1}^m q_{P_k}(i) f\left(\frac{q_{P_1}(i)}{q_{P_k}(i)},
    \ldots, \frac{q_{P_{k-1}}(i)}{q_{P_k}(i)}\right)} \\
  & = \sum_{i=1}^{m} \left(\int_{\mc{X}} q(i \mid x) p_k(x) d\mu(x)\right)
  f\left(\frac{\int_{\mc{X}} q(i \mid x)
    [p_1(x), ~ \cdots, ~ p_{k-1}(x)] d\mu(x)}{
    \int_{\mc{X}} q(i \mid x)p_k(x) d\mu(x)}\right) \\
  & \le \sum_{i=1}^{m}
  \int_{\mc{X}} q(i \mid x) f\left(
  \frac{p_1(x)}{p_k(x)}, \ldots, \frac{p_{k-1}(x)}{p_k(x)}
  \right) p_k(x) d\mu(x)
\end{align*}
by Lemma~\ref{lemma:general-log-sum}. Noting that
$\sum_{i=1}^m q(i \mid x) = 1$, we obtain
our desired result.


\fi

\end{document}